\documentclass[12pt,a4paper]{article}

\usepackage[utf8]{inputenc}
\usepackage[left=2cm,right=2cm,top=2cm,bottom=2cm]{geometry}

\usepackage{float}
\usepackage[english]{babel}
 \usepackage{mathrsfs}
 \usepackage{bbm}
\usepackage{amsmath}
\usepackage{amsfonts}
\usepackage{amssymb}
\usepackage{makeidx}
\usepackage{color}
\usepackage{amsthm}
\usepackage{graphicx}
\usepackage{hyperref}
\usepackage{qtree}
\usepackage[all]{xy}
\usepackage{pgfplots}
\usepackage{mathrsfs}
\usetikzlibrary{arrows}

\usepackage[english]{babel}
 \usepackage{mathrsfs}
\usepackage{amsmath}
\usepackage{multirow}

\usepackage{amsfonts}
\usepackage{amssymb}
\usepackage{makeidx}
\usepackage{color}
\usepackage{amsthm}
\usepackage{graphicx}
\usepackage{hyperref}
\usepackage{qtree}
\usepackage[all]{xy}

\newtheorem{thm}{Theorem}[section]
\newtheorem{thm*}{Theorem}

\newtheorem{lem}[thm]{Lemma}
\newcounter{counter_conj-lem}
\newtheorem{conj-lem}[thm]{Conjecture-Lemma}

\newtheorem{prop}[thm]{Proposition}

\newcounter{counter_conj-prop}
\newtheorem{conj-prop}[thm]{Conjecture-Proposition}

\newcounter{counter_conj-thm}
\newtheorem{conj-thm}[thm]{Conjecture-Theorem}

\newtheorem{cor}[thm]{Corollary}

\theoremstyle{remark}

\newtheorem{defi}[thm]{Definition}

\newtheorem{rem}[thm]{Remark}

\newtheorem{ex}[thm]{Example}

\newcommand{\C}{\mathbb{C}}
\newcommand{\R}{\mathbb{R}}
\newcommand{\Q}{\mathbb{Q}}
\newcommand{\N}{\mathbb{N}}
\newcommand{\A}{\mathbb{A}}
\newcommand{\Z}{\mathbb{Z}}
\newcommand{\D}{\mathbb{D}}
\newcommand{\bP}{\mathbb{P}}
\newcommand{\pan}{1, \mathrm{an}}
\newcommand{\Hy}{\mathbb{H}}
\newcommand{\PGL}{\mathrm{PGL}}
\newcommand{\GL}{\mathrm{GL}}
\newcommand{\g}{\gamma}
\newcommand{\La}{\Lambda}
\newcommand{\Ga}{\Gamma}
\newcommand{\Hom}{\mathrm{Hom}}

\newcommand{\Po}{\mathcal{P}}
\newcommand{\la}{\lambda}
\newcommand{\eS}{\mathcal{S}}
\newcommand{\M}{\mathcal{M}}
\newcommand{\Af}{\mathbb{A}}
\newcommand{\Hr}{\mathcal{H}}
\newcommand{\cP}{\mathcal{P}}

\DeclareMathOperator{\diam}{diam}
\DeclareMathOperator{\ch}{ch}

\DeclareMathOperator{\Hdim}{\mathrm{Hdim}}

\DeclareMathOperator{\modu}{mod}
\DeclareMathOperator{\ord}{ord}

\DeclareMathOperator{\an}{\mathrm{an}}

\title{Variation of the Hausdorff dimension and degenerations of Schottky groups}
\author{Nguyen-Bac Dang, Vlerë Mehmeti}
\date{}

\begin{document}

\maketitle

\begin{abstract}
We show that the Hausdorff dimension of the limit set of a Schottky group varies continuously over the moduli space of Schottky groups defined over any complete valued field constructed by Poineau and Turchetti.
To obtain this result, we first study the non-Archimedean case in a setting of Berkovich analytic spaces, where we make use of Poincaré series. We show that the latter can be extended meromorphically over $\C$ and admit a special value at $0$ which is a purely topological invariant.

As an application, we prove results on the asymptotic behavior of the Hausdorff dimension of degenerating families of Schottky groups over $\C$.
For certain families, including Schottky reflection groups, we obtain an exact formula for the asymptotic logarithmic decay rate of the Hausdorff dimension. This generalizes a theorem of McMullen.
\end{abstract}

\newcommand*{\nsection}[1]{
    \section*{#1}
    \addcontentsline{toc}{section}{#1}
}
\setcounter{tocdepth}{2}
\tableofcontents

\nsection{Introduction}

Koebe proved a uniformization theorem for compact Riemann surfaces using a particular class of subgroups of Möbius transformations, known as \emph{Schottky groups}. A major advantage of this specific approach to uniformization theory is that,  unlike the classical one (\emph{cf.} {\cite[\S~4.4]{jost_compact}}),
it can be extended to the non-Archimedean setting.  For example, the notion of Schottky groups was adapted to non-Archimedean fields by Mumford in \cite{mumford} (\emph{cf.} also \cite{gerritzen_van_der_put}, \cite{berkovich_spectral}, \cite{poineau_turcheti_2}), who showed a similar result for an analogue of Riemann surfaces, known as \emph{Mumford curves}. In \cite{poineau_turcheti_universal}, Poineau and Turchetti constructed a moduli space $\eS_g$ of rank $g$ Schottky groups defined over complete valued fields by using \emph{Berkovich spaces over $\Z$}, for which they also proved a uniformization theorem. This setting encompasses both Koebe's and Mumford's results, so that they are not only analogues but parts of a common  theory.
In all cases, a crucial role is played by the \emph{limit set} of the group, \emph{i.e.} the set of accumulation points of orbits for the action on the projective line. It is a fractal set, and over $\C$, the non-integrality of its \emph{Hausdorff dimension} is a manifestation of its self-similar nature.  

The main goal of this paper is the study of variation of the Hausdorff dimension of these limit sets over the moduli space $\eS_g$. A necessary first step is understanding this invariant for a group over a fixed complete valued field. Over $\C$, %
Schottky groups and the  dimension of their limit sets have been extensively studied, \emph{e.g.} in %
\cite{maskit_characterization_schottky}, \cite{chuckrow}, \cite{akaza_example}, \cite{marden}, \cite{hejhal_schottky_teichmuller}. For this purpose, Patterson introduced a special function in \cite{patterson}, the \emph{Poincaré series}, which encodes properties of the group action. By working with \emph{Berkovich analytic spaces},
we show that this is the case over non-Archimedean fields as well. We also investigate the properties of the non-Archimedean  Poincaré series. When the field is locally compact, this has been studied in  \cite{paulin_degen}, \cite{hersonsky_hubbard}, and \cite{quint_overview}.   

There are several approaches to non-Archimedean analytic geometry (\emph{cf.} \cite{tate}, \cite{raynaud}, \cite{berkovich_spectral},  \cite{huber}). We use the approach of Berkovich (\cite{berkovich_spectral}), who provides a theory of analytic varieties over any complete valued field.
For non-Archimedean fields $k$ (such as $\Q_p$ or $\C((t))$), Berkovich analytic curves have the structure of an \emph{infinitely branched} real graph. For example, the Berkovich projective line~$\bP^{\pan}_k$ is a real tree (though the Berkovich topology is not the tree one, and is locally compact). If~${k=\C}$, then we recover the usual theory of complex analytic varieties, and in particular~$\bP^{\pan}_k$ is the Riemann sphere $\bP^1(\C)$. There is a natural action of $\PGL_2(k)$ on $\bP^{\pan}_k$ via Möbius transformations.

Let $k$ be a complete valued field. A Schottky group $\Ga $ over $k$ is a free and finitely generated subgroup of $\PGL_2(k)$, whose non-trivial elements are \emph{loxodromic} (\emph{i.e.} with exactly two fixed points on $\bP^{\pan}_k$), and which acts freely and properly on an open invariant connected subset of~$\bP^{\pan}_k$. Let $\La_{\Ga}$ denote its associated limit set. It is contained in $\bP^{1}(k):=k \cup \{\infty\}$, and is a fractal set. The Hausdorff dimension of $\La_{\Ga}$ is a conformal invariant.

Lemanissier and Poineau \cite{lemanissier_poineau} develop a theory of \emph{global} Berkovich analytic spaces over~$\Z$ or more general Banach rings; \emph{cf.} also \cite{poineau_berkovich_Z}. The moduli space $\eS_g$ from \cite{poineau_turcheti_universal} of rank~$g$ Schottky groups is an open and connected analytic subset of the affine space $\Af^{3g-3, \an}_{\Z}$ if~${g \geqslant 2}$  (and of $\Af^{1, \an}_{\Z}$ if~${g=1}$). It naturally admits a structural morphism $\pi: \eS_g \rightarrow \M(\Z)$, where~$\M(\Z)$ consists of multiplicative semi-norms on $\Z$ bounded by the absolute value $|\cdot|_{\infty}$. We will say that $x \in \eS_g$ is \emph{non-Archimedean} if $\pi(x)$ is a non-Archimedean semi-norm on $\Z$ (\emph{e.g.} $|\cdot|_p^{\alpha}, \alpha \in (0, +\infty)$, $p$-prime), and \emph{Archimedean} otherwise (\emph{e.g.} $|\cdot|_{\infty}^{\varepsilon}, \varepsilon\in (0,1]$).  
The moduli space $\eS_{g, \C}$, constructed by  Bers~\cite{bers}, of complex rank $g$ Schottky groups, shares similar properties (\emph{cf.} \cite{hejhal_schottky_teichmuller}) and can be obtained from~$\eS_g$ via $\pi^{-1}(|\cdot|_{\infty})$.  
Similarly, for a non-Archimedean field $k$, such a space~$\eS_{g,k}$ was constructed by Gerritzen \cite{gerritzen_80,gerritzen_81}. 

More generally, in \cite{poineau_turcheti_universal}, the authors construct in a similar way a moduli space for rank~$g$ Schottky groups over \emph{any} Banach ring $A$, denoted $\eS_{g, A}$. 
Let $\Ga_x, x \in \eS_{g, A}$, be the family of Schottky groups parametrized by $\eS_{g, A}$, and $\La_x$ their corresponding limit sets. With this notation, our main result is the following:

\begin{thm*}[Theorem \ref{thm_contbanach}] \label{thm_continuity_intro}  Let $g \geqslant 1.$ The function $d_A: \eS_{g, A} \rightarrow \R_{\geqslant 0}$, $x \mapsto \Hdim  \La_x$, is continuous, where $\Hdim$ denotes the Hausdorff dimension.   
\end{thm*}

In particular, Theorem \ref{thm_continuity_intro} holds for the moduli space $\eS_g$ over the Banach ring $\mathbb{Z}$ (as well as for the integer rings of number fields, complete valued fields, complete rings, \emph{etc.}). 
For $\eS_{g, \C}$, the function $d_{\C}$ is real-analytic by \cite{ruelle}. If $k$ is non-Archimedean, we deduce a similar result for $\eS_{g,k}$ (\emph{cf.} Remark~\ref{rem_analytic}).

Theorem \ref{thm_continuity_intro} is obtained as a consequence of the special case where $A=\Z$. For the strategy of proof over $\eS_g$, we start by remarking that
the set of Archimedean points is open. Hence, the continuity of $d_{\Z}$ on this open can be obtained as an application of Ruelle's result. 
The main point of difficulty is the continuity on opens containing both types of points. To treat this case, we first employ a result of Poineau-Turchetti assuring the local existence of a continuous \emph{Schottky figure} (a configuration of disks encoding the group action). This is obtained by expanding a Schottky figure over the fiber of a non-Archimedean point, whose existence is a crucial result of Gerritzen. We then deduce a uniform bound on the Hausdorff dimension. The second major ingredient is a method of McMullen (\emph{cf.}  \cite{mcmullen_hausdorff_3})  for certain groups over~$\C$, through which he approximates the Hausdorff dimension,  and which we adapt to non-Archimedean fields. The bound on the Hausdorff dimension allows us to prove the continuous variation of these approximations. We also establish \emph{distortion estimates} over complete valued fields, and then prove that they hold uniformly. Combining all these elements, we show that McMullen's approximations converge \emph{uniformly}, and are able to conclude. This is the core part of our proof.  Over opens containing only non-Archimedean points $x$, we
exploit a simpler characterization of $\Hdim \La_x$, as follows.

\medskip

Let $k$ be a complete non-Archimedean field, and  let $\Ga$ be a Schottky group over $k$. Let $\La_{\Ga}$ denote its limit set. One can associate to $\Ga$ its \emph{Poincaré series}, defined for $s \in \R_{\geqslant 0}$ as $$\Po_{\eta}(s):=\sum_{\g \in \Ga} e^{-s\rho(\eta, \g \eta)},$$  
where $\eta \in \bP^{\pan}_k$ is a well-chosen point, and $\rho$ is a (hyperbolic) metric defined on a $\PGL_2(k)$-invariant subset $\Hy_k$ of $\bP^{\pan}_k$. 

The Hausdorff dimension of $\La_{\Ga}$ is usually the \emph{critical exponent} of $\Po_{\eta}(s)$, \emph{i.e.} the infimum of its locus of convergence. 
As conjectured and studied by Fried  \cite{fried_analytic_torsion}, and by analogy with abundant examples in analytic number theory,  the  meromorphic extensions of certain special functions often encode geometric or topological information. This is illustrated by: the Riemann zeta function, Eisenstein series and special $L$-functions of elliptic curves (see \emph{e.g.} \cite{weil}). 
In the case of Poincaré series this was recently established in  \cite{dang_riviere_poincare} (for surfaces),  \cite{benard_chaubet_dang} (for triangulations),  N. Anantharaman (course at Collège de France; for finite graphs). 

We prove one more occurrence of this phenomenon, this time in a non-Archimedean setting (\emph{cf.} Corollary \ref{cor_critical_exponent_hausdorff}, Corollary \ref{cor_spectral_formula} and Theorem \ref{analyticpoincare}, respectively):

\begin{thm*} \label{analyticpoincare_intro} 
\begin{enumerate}
\item[(i)] The critical exponent of $\Po_{\eta}(s)$ is equal to $\Hdim \La_{\Ga}$ and to the spectral radius of a matrix $M(s) \in M_{2g \times 2g}(\R)$ whose non-zero entries are of the form $e^{-sl}, l \in \R_{>0}$. 
\item[(ii)] (Meromorphic extension) The Poincaré series can be extended meromorphically to $\C$ in the form $$\dfrac{P(e^{-sl_1}, e^{-sl_2}, \dots, e^{-sl_g})}{Q(e^{-sl_1}, e^{-sl_2}, \dots, e^{-sl_g})},$$ with $P,Q$ polynomials over $\Z$, and $l_i \in \R_{>0}$ for all~$i$.

\item[(iii)] (Special value) When $g\geqslant 2$, the extension of $\Po_{\eta}(s)$ is analytic at $s=0$, and $$\Po_{\eta}(0)=\frac{1}{1-g}.$$
\end{enumerate}
\end{thm*}

 For $g=1$, $\Po_{\eta}(s)$ has a pole of order one at $s=0$, 
while for $g\geqslant 2$,  $\Po_{\eta}(0)$ depends uniquely on the Euler characteristic of the  Mumford curve over $k$ that $\Ga$ uniformizes. As this curve can be retracted to a finite metric graph, Theorem \ref{analyticpoincare_intro} can also be interpreted in such a setting.
The connection between $\Hdim \La_{\Ga}$ and $\Po_{\eta}(s)$ over $\C$ was established in \cite{patterson} and~\cite{sullivan_density}.  We deduce it by applying the results of \cite{das_simmons_urbanski}. For the rest, we start by establishing a connection between $\Po_{\eta}(s)$ and a well-chosen matrix $M(s)$. 
 We prove that the Poincaré series can be obtained as a function of $(I-M(s))^{-1}$, from where the meromorphic behavior of $\Po_{\eta}(s)$ follows. Finally, we conclude through a further study of the poles of the resolvent $(I-M(s))^{-1}$. The key point is a technical combinatorial lemma related to the geometry of a fundamental domain of the action. 

\medskip

We now provide an application of this framework where the base field can vary, \emph{i.e.} over~$\eS_g$, to degenerating families $(\Ga_n)_n,$ of \emph{complex} Schottky groups. Heuristically speaking, through an appropriate choice of a sequence $(\varepsilon_n)_n$ which tends to $0$ as $n \rightarrow +\infty$, one can interpret~$\Ga_n$ as a Schottky group $\Ga_n'$ over $(\C, |\cdot|_{\infty}^{\varepsilon_n}),$ hence as a point $x_n$ of $\eS_{g, (\C, |\cdot|_{\infty}^{\varepsilon_n})} \subseteq \eS_g$.
The idea is to extend~${(x_n)_n}$ \emph{continuously} through a limit point in $\eS_g$. When this is possible, the limit point must be non-Archimedean, because the Archimedean part of $\eS_g$ is open and the family is assumed degenerating.   See also the figure below. 

\resizebox{15cm}{!}{%
\centering
\input{degen_scheme}   
 }
 
In the literature, this continuation idea through \emph{scaling} with Berkovich \emph{hybrid spaces or some analog} has been prolific, see \emph{e.g.}  \cite{kiwi}, \cite{arfeux3}, \cite{boucksom_jonsson_degeneration}, \cite{nie_rescaling}, \cite{dujardin_favre_degen}, \cite{favre_degen}, \cite{luo_trees}. 
More generally, studying families of transformation groups by rescaling the metric and then taking a limit which can be interpreted as an action on a real tree dates back to works of Culler, Morgan and Shalen \cite{morgan_shalen1}, \cite{morgan-shalen2}, \cite{culler_morgan}, Paulin \cite{paulin_topo_gromov}, and Bestvina \cite{bestvina}.

In our setting, we obtain the following asymptotic behavior of the Hausdorff dimension. 
Let~$k$ be a complete valued field.
A loxodromic element $\g \in \PGL_2(k)$ is uniquely determined by its attractive and repelling points $\alpha, \beta \in \bP^1(k)$, and its \emph{multiplier} $y \in k$ which satisfies $0<|y|<1$, and for which $l=-\log |y|,$ where $l$ is the \emph{translation length} of $\g$. The data $(\alpha, \beta, y)$ are called the \emph{Koebe coordinates} of $\g$ and they are used to parametrize Schottky groups (\emph{cf.} \cite{bers}, \cite{gerritzen_81}, \cite{poineau_turcheti_universal}). 

\begin{thm*}[Theorem \ref{degeneration}] \label{degeneration_intro} Assume $g \geqslant 2.$ 
Let $p: D \rightarrow \C^{3g-3}$ be a function defined on the open unit disk $D$, meromorphic on a neighborhood $U$ of $0$ and  holomorphic on $U \backslash \{0\}$. Let us denote its coordinate functions by $\alpha_3, \dots, \alpha_g, \beta_2, \dots, \beta_g, y_1, \dots, y_g$. Assume that for every $i=1,2,\dots, g,$
\begin{enumerate}
\item[(i)]  $y_i(z)$ are not constant and $y_i(0)=0$,
\item[(ii)]  for all $u(z), v(z) \in \{\alpha_1(z), \beta_1(z), \dots, \alpha_g(z), \beta_g(z)\} \backslash \{\alpha_i(z), \beta_i(z)\}$, 
$$\mathrm{ord}_{0} \left(y_i(z) \cdot [u(z), v(z); \alpha_i(z), \beta_i(z)] \right) >0,$$
where $\alpha_1(z):=0$, $\beta_1(z):=\infty$, $\alpha_2(z):=1$ for all $z \in D$, and $[a, b; c, d]$ denotes the cross-ratio of $a,b,c,d \in \bP^1(\C)$,
\item[(iii)]  the functions $\alpha_i(z), \beta_i(z)$, $i=1,2,\dots, g$, are all different. 
\end{enumerate}
Then there exists $\eta > 0 $ such that $p(z) \in \eS_{g, \C}$  for $0<|z| < \eta$ and the Hausdorff dimension $s(z)$ of the limit set $\Lambda_z$ associated to $p(z)$ satisfies: 
\begin{equation*}
s(z) \sim - \dfrac{s_0}{\log|z|},
\end{equation*}
as $z \rightarrow 0$, where $s_0>0$ is the Hausdorff dimension of the limit set in $\mathbb{P}^{\pan}_{\C((z))}$ of the Schottky group with generators whose Koebe coordinates are determined by the Laurent series induced by~$p(z)$.
\end{thm*}

To prove this result, the key point is to use the scaling $\varepsilon(z)=-\log (|z|^{-1})$ for $z \neq 0$, and show that this gives an appropriate family in $\eS_{g, A}$, where $A$ is a Banach ring, known as the \emph{hybrid disk}.

\medskip

We provide explicit examples to which Theorem \ref{degeneration_intro} applies. In particular, this is the case for families of \emph{Schottky reflection groups}.
We present a simplified case here, see Theorem \ref{thm_explicit} for the stronger version we prove.

Let $g, k, l$ be integers such that $g \geqslant 2,$ and $k, l\geqslant 1$. Fix $c_1, c_2, \dots, c_{g+1}$ distinct points on the  unit circle. For $z \neq 0$, let $C_{i,z}:=C(c_i, |z|^k)$ for $i=1,2,\dots, g$, and $C_{g+1, z}:=C(c_{g+1}, |z|^l)$, where $C(a, r) \subseteq \C$ denotes the circle centred at $a$ and of radius $r$. 

The group $\Gamma_z$ generated by the reflections with respect to the circle $C_{i,z}$, $i=1,2,\dots, g+1$, is called a \emph{Schottky reflection group} with respect to this configuration of circles. 
Figure \ref{fig_degen_intro} illustrates the degeneration of the family $(\Ga_z)_z$ toward a non-Archimedean Schottky figure. 

\begin{thm*}(Theorem \ref{thm_explicit})
The Hausdorff dimension $s(z)$ of  the limit set of $\Ga_z$ satisfies $$s(z) \sim \frac{-\log u}{2 \log |z^{-1}|}$$
as $z \rightarrow 0$, and where $u$ is the largest root of $X^{k+l}+(1-1/g)X^k-1/g$.
\end{thm*}

In the special case when $g=2, k=l=1$ and $c_1, c_2, c_3$ form an equilateral triangle, we recover McMullen's \cite[Theorem 3.5]{mcmullen_hausdorff_3}.

\begin{figure}[h!]
    \centering
    \includegraphics[scale=0.3]{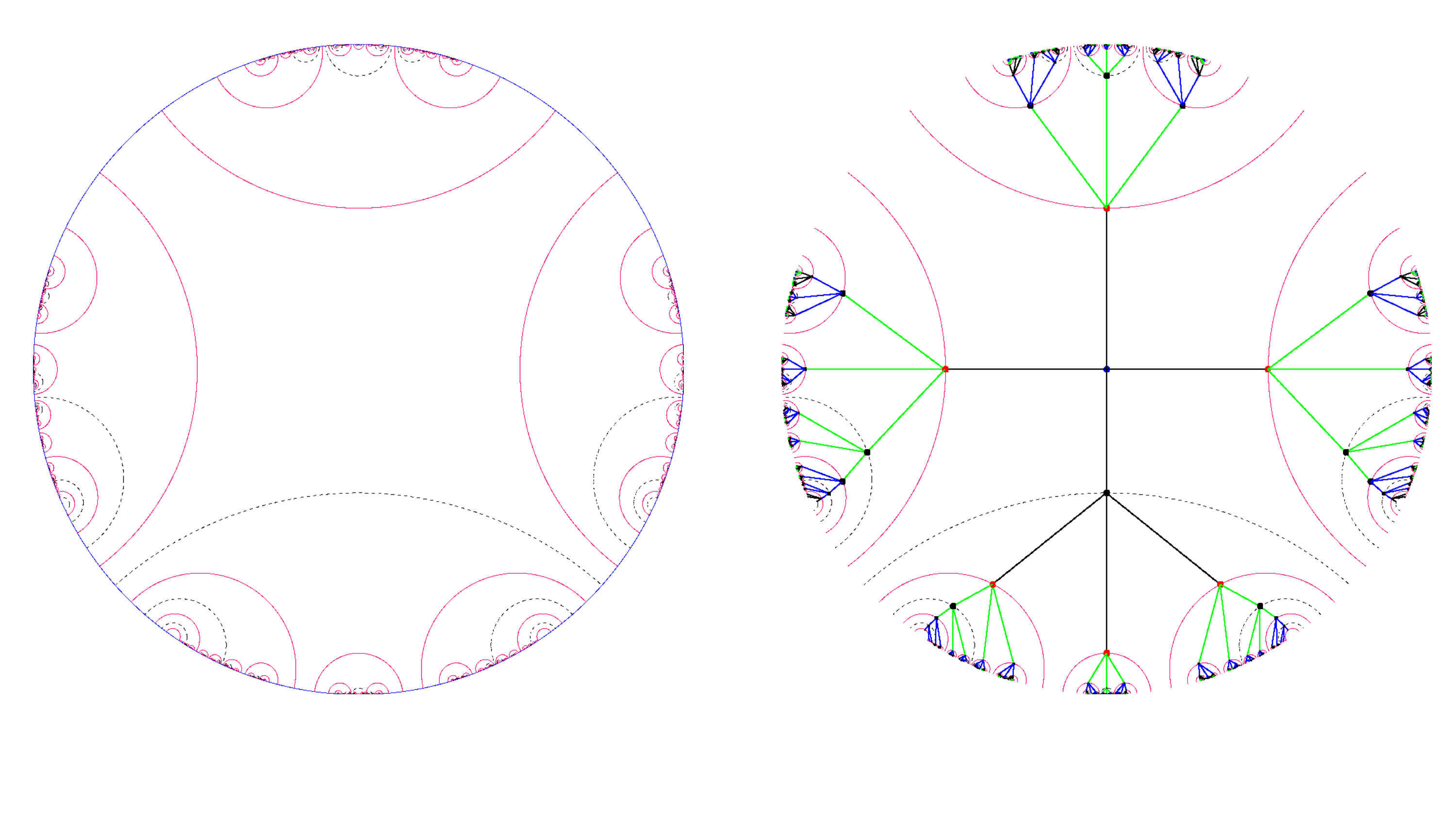}
    \caption{On the left is the image of the first circles centered at $1,i,-1,-i$ by the Schottky reflection group with $g=3$, $z=1.3$, $k=1, l=2$; on the right its non-Archimedean degeneration in $\bP^{\pan}_{\C((z))}$ and a few fundamental domains for the action.}
    \label{fig_degen_intro}
\end{figure}

\subsection*{Structure of the manuscript}
The text is divided into six sections. 

In the first two sections, we introduce some elements and properties of our framework. 
In Section \ref{sect_1}, we briefly present some notions from the theory of Berkovich analytic spaces, as well as Schottky groups defined over any complete valued field. We conclude by comparing the Berkovich topology to the classical Gromov topology on the compactification of a hyperbolic space.

In Section \ref{sect_3}, we define certain Poincaré series over non-Archimedean fields, and establish their connection with the Hausdorff dimension of limit sets of Schottky groups. We also prove that the Poincaré series can be studied through the resolvant of a matrix. We conclude with some bounds on the Hausdorff dimension.  
The main results of this section are Theorem \ref{thm_meromorphic_extension} and Corollary \ref{cor_spectral_formula}.

In Section \ref{sect_4}, we continue the study of Poincaré series in the non-Archimedean setting, and prove that they can be meromorphically extended over $\C$ with a special value at $0$ which is of purely topological nature. The main result of this section is Theorem \ref{analyticpoincare}. 

{In Section \ref{sect_2}, we recall some basic definitions and properties of Berkovich spaces over~$\Z$, as well as the construction of the moduli space $\eS_g$ of Schottky groups of rank $g$.} We also study the \emph{modulus function} over $\Z$.

In Section \ref{section_unifoorm}, we study the variation of the Hausdorff dimension of limit sets of Schottky groups over the moduli space $\eS_g$ for $g \geqslant 1.$ We prove that said variation is continuous. The main result of this section is Theorem \ref{thm_contbanach}. We start by first developing the necessary tools, which comes down to studying the uniformity of the behavior of approximations of the Hausdorff dimension.

In Section \ref{section_degen}, we provide applications of Theorem \ref{thm_contbanach} to the study of Hausdorff dimensions for degenerating families of Schottky groups. We first obtain a result on their asymptotic behavior for a family of Schottky groups over $\C$ (Theorem \ref{degeneration}). Then we give explicit examples to which this result applies, which includes all collision possibilities when the rank is $2$. A particular case is that of Schottky reflection groups (Theorem~\ref{thm_explicit}). 

 \subsection*{Acknowledgements}
 We are very thankful to Gilles Courtois and Frédéric Paulin for many inspiring and clarifying discussions. Our gratitude also goes to Nguyen-Viet Dang and Gabriel Rivière for enthusiastically sharing their results with us and patiently answering our questions, leading us to additional investigations. We are indebted to Jérôme Poineau for thoughtful comments which improved the quality of our manuscript, and to Jialun Li for spotting a flaw in our previous draft.  
Many thanks as well to Omid Amini, Sébastien Boucksom, Laura DeMarco, Antoine Ducros, Charles Favre, Antonin Guilloux, and Daniele Turchetti for their thoughts and suggestions.

\section{Schottky groups over complete fields} \label{sect_1}

We recall here some elements from the theory of Schottky groups in the language of Berkovich spaces developed  by Berkovich (\cite[4.4]{berkovich_spectral}) and by Poineau and Turchetti (\cite{poineau_turcheti_2}).

A field endowed with a multiplicative norm with respect to which it is complete will be called a \emph{complete valued field}.
Given a complete valued field $(k, |\cdot|)$, we will say it is \emph{non-Archimedean} if the ultrametric inequality $|x+y| \leqslant \max(|x|, |y|)$ is satisfied for all $x,y \in k$. Otherwise, $k$ is said to be \emph{Archimedean}, in which case it is isomorphic to either $\C$ or $\R$. 

\subsection{The Berkovich projective line} \label{section_berkoline}

We briefly recall and introduce some notions from the theory of Berkovich spaces (with foundations in \cite{berkovich_spectral}) which we will use (see \emph{e.g.} \cite[Part I]{poineau_turcheti_1} for more details). 

Let $(k, |\cdot|_k)$ be a complete valued field. 

\begin{defi}
Let $n \geqslant 1.$ The Berkovich \emph{analytic affine space}~$\mathbb{A}_k^{n,\an}$ as a set consists of multiplicative semi-norms on $k[T_1, T_2, \dots, T_n]$ which extend the norm on~$k$. Its topology is the coarsest one such that the maps $\varphi_P: \Af_k^{\pan} \rightarrow \R_{\geqslant 0}, |\cdot| \mapsto |P|,$ are continuous for all~${P \in k[T_1, T_2, \dots,T_n]}.$
\end{defi}

The topological space $\Af_k^{n,\an}$ is locally compact and Hausdorff. When $k=\C$, it is homeomorphic to the usual analytic $\Af^{n}(\C)$, and when $k=\R$, it is $\Af^n(\C)/\mathrm{conjugation},$ \emph{i.e.}  conjugated elements of $\C^n$ are identified.

Set $n=1$. If $k$ is non-Archimedean, for $a \in k$ and $r \in \R_{\geqslant 0}$, one can define the point $\eta_{a,r} \in \Af^{\pan}_k$ via $\sum_n a_n(T-a)^n \mapsto \max_{n} |a_n|_k r^n$. It satisfies $|T-a|_{\eta_{a,r}}=r$. Moreover, $\eta_{a,r}=\eta_{b,s}$ if and only if $s=r$ and $|a-b|_k \leqslant r$. The point $\eta_{0,1}$ is called the \emph{Gauss point}.
We note that $\eta_{a,r}$ depends on the choice of a coordinate function. 
Taking $r=0$ induces an embedding ${k \hookrightarrow \Af^{\pan}_k, a \mapsto \eta_{a,0}}.$  When restricted to $k$, the topology of $\Af_k^{\pan}$ is the same as that induced by the norm of $k$. 

The Berkovich \emph{analytic projective line} $\bP^{\pan}_k$ can be obtained from gluing two copies of~$\Af_k^{\pan}$, or equivalently, taking $\bP^{\pan}_k=\Af_k^{\pan} \cup \{\infty\}$, where the point $\infty$ corresponds to the semi-norm~$|\cdot|_{\infty}$ on~$k[\frac{1}{T}]$ such that $|\frac{1}{T}|_{\infty}=0$. It is a compact and Hausdorff topological space containing~${k \cup \{\infty\}:=\bP^{1}(k)}$.

When $k$ is non-Archimedean, $\bP^{\pan}_k$ has the structure of an infinitely branched real tree (\emph{cf.}~\mbox{\cite[3.4.20]{ducros}}). For instance, the unique injective path $[\eta_0, \infty]$ in $\bP^{\pan}_k$ joining $\eta_{a,0}$ and $\infty$ is $\{\eta_{a,r}: r \geqslant 0\} \cup \{\infty\}.$ 
We emphasize that the Berkovich topology on $\bP^{\pan}_k$ does \emph{not} coincide with the tree topology (\emph{cf.} Lemma \ref{finer}).
The points of $\bP_k^{\pan}$ are classified into four types. Let~$\widehat{k^{\mathrm{alg}}}$ denote the completion of an algebraic closure of~$k$. There is a projection morphism of analytic spaces $\pi: \bP^{\pan}_{\widehat{k^{\mathrm{alg}}}} \rightarrow \bP^{\pan}_k$, which amounts to restricting semi-norms from $\widehat{k^{\mathrm{alg}}}[T]$ to $k[T]$.  

The \emph{type 1} points of $\bP^{\pan}_k$ are the points of the set $\pi(\bP^{1}(\widehat{k^{\mathrm{alg}}})).$ They lie in the extremities of the real tree $\bP^{\pan}_k$, though in general there can be other points at the extremities as well. 
In particular, the points of $\bP^{1}(k)$ are all of type 1.

A \emph{closed disk} in $\bP^{\pan}_k$ is a Berkovich analytic space isomorphic to a set $$D(a,r):=\{|\cdot| \in \Af_k^{\pan}: |T-a| \leqslant r\},$$ where $a \in k$ and $r \in \R_{>0}.$ It is a compact subtree of $\bP^{\pan}_k$ containing all the points that retract to $[\eta_{a,0}, \eta_{a,r}]$ when retracting $\bP^{\pan}_k$ to $[\eta_{a,0}, \infty]$. Its boundary is a singleton: $\partial{D}(a,r)=\{\eta_{a,r}\}$. The set $D(a,r) \cap k$ is the closed ball centered at $a$ and of radius $r$ in $(k, |\cdot|_k).$ One can similarly define \emph{open disks} in $\bP^{\pan}_k$.   

\

 We denote by $\PGL_2(k)$ the quotient of $\GL_2(k)$ by its center. Let us fix a coordinate function~$T$ on $\bP^{\pan}_k$.
 The action via \emph{Möbius tranformations} of $\PGL_2(k)$ on $\bP^1(k)$ can be extended by pulling-back the points of $\bP^{\pan}_k$ as semi-norms on $k[T]$ or $k[T^{-1}]$; \emph{cf.} \cite[II.1.3]{poineau_turcheti_2} for more details on said extension.
This action preserves closed (resp. open) disks (\emph{cf.} \cite[Lemma~II.1.8]{poineau_turcheti_1}), as well as the types of points on $\bP^{\pan}_k$.  

\subsection{The Berkovich hyperbolic space} \label{section_hyperbolic}

Let $(k, |\cdot|_k)$ be a complete non-Archimedean field. The following object plays an important role in this manuscript. We also study it further in Section \ref{gromov}.

\begin{defi} \label{hyperbolic} The set 
$\Hy_k:=\{x \in \bP^{\pan}_k: x \ \text{is not of type 1}\}$
will be called the \emph{hyperbolic space over $k$}.
\end{defi}

We remark that $\Hy_k=\bP^{\pan}_k \backslash \bP^1(k)$ if and only if $k$ is algebraically closed.
Both $\Hy_k$ and $\bP^{\pan}_k \backslash \Hy_k$ are dense in $\bP^{\pan}_k.$ As Möbius transformations preserve the types of points, $\PGL_2(k)$ acts on both $\Hy_k$ and its complement. 

\begin{rem}
Other non-Archimedean versions of hyperbolic spaces have been studied in \cite{letelier_hyperbolic}, \cite{baker_rumely}, and \cite{dujardin_favre_degen}. When $k$ is algebraically closed, our definition coincides with that of \cite{letelier_hyperbolic} and \cite{baker_rumely}. It does not coincide with that of the Drinfeld upper half-plane, another candidate for non-Archimedean hyperbolic spaces.
\end{rem}

The space  $\Hy_k$ is endowed with a metric $\rho$, which plays the role of a hyperbolic distance and is interpreted as  
the length of intervals in $\Hy_k$ (\emph{cf.} \cite[\S~I.8]{poineau_turcheti_1}). When $k=k^{\mathrm{alg}}$, this is the path-distance metric in \cite[\S~2.7]{baker_rumely}. We briefly recall its definition.

For $a, b \in k$ and $r,s \in \R_{>0}$, if $\max(s,r)\geqslant |a-b|_k,$ then $\rho(\eta_{a, r}, \eta_{b,s}):=|\log\frac{r}{s}|$.
Otherwise, $[\eta_{a,r}, \eta_{b,s}] = [\eta_{a, r}, \eta_{a, |a-b|_k}] \sqcup (\eta_{b, |a-b|_k}, \eta_{b,s}]$, and we set 
$\rho(\eta_{a, r}, \eta_{b,s}):=|\log\frac{|a-b|_k}{r}|+|\log\frac{s}{|a-b|_k}|$.

Given $x, y \in \bP^{\pan}_k$ of type 2 or 3 (\emph{i.e.} points which don't lie in the extremities of the real tree), {through a finite base change}, the unique path $[x,y]$ can be lifted to a finite disjoint union of paths of type $[\eta_{a,r}, \eta_{b,s}]$. We set $\rho(x,y):=\rho(\eta_{a,r}, \eta_{b,s})$.  
Given arbitrary points $u,v \in \bP^{\pan}_k$, let $\rho(u,v):=\sup \{\rho(x,y): x,y \in [u,v] \ \text{of type 2 or 3}\}.$ Here $[u,v]$ denotes the unique path joining~$u$ and $v$. 

One can show that $\rho$ does not depend on the choice of a coordinate function in $\bP^{\pan}_k$ (\emph{cf.} \cite[Proposition I.4.14]{poineau_turcheti_1}).

\begin{prop} The following properties hold.
\begin{enumerate}
\item \cite[Lemma I.8.12]{poineau_turcheti_1} For two points $x,y \in \bP^{\pan}_k$, one has $\rho(x,y)= +\infty$ if and only if~$x$ or $y$ is a type 1 point. 

\item \cite[Lemma I.8.12, Theorem I.9.17]{poineau_turcheti_1} The interval length function $\rho$ is a $\PGL_2(k)$-invariant metric on $\Hy_k$. 
\end{enumerate}
\end{prop}

\begin{rem}
For $a\in k$, and $s,r \in \R_{>0}$ with $s>r$, the distance $\rho(\eta_{a,r}, \eta_{a,s})=\log \frac{s}{r}$ coincides with the modulus of the annulus  $\{x \in \bP^{\pan}_k : r \leqslant |T-a|_x \leqslant s\}.$
\end{rem}

\begin{rem}\label{rem_radius}
When $\eta_{a,r}$ is contained in the open unit disk in $\bP^{\pan}_k$, then $\rho(\eta, \eta_{a,r})=\log \frac{1}{r}$, where~$\eta:=\eta_{0,1}$ is the Gauss point. Consequently, the radius of the disk $D(a,r)$ can be computed as~$r=e^{-\rho(\eta, \partial{D}(a,r))}$.
\end{rem} 

\begin{lem} \label{finer}
The $\rho$-topology on $\Hy_k$ is \emph{strictly} finer than the Berkovich topology induced on $\Hy_k$ from $\bP^{\pan}_k$.
\end{lem}

\begin{proof}
It suffices to show that for any $P \in k[T]$, the function $\varphi_P: \Hy_k \rightarrow \R_{\geqslant 0}, x \mapsto |P|_x$, is $\rho$-continuous. By \cite[3.4.24.3]{ducros}, there exists a finite graph $G'$ in $\bP^{\pan}_k$ such that $P$ is continuous with respect to the Berkovich topology on $G'$, and locally constant elsewhere. Set $G:=G'\cap \Hy_k$. Let $r_G: \Hy_k \rightarrow G$ denote the deformation retraction of $\Hy_k$ to $G$ obtained by restricting the deformation retraction $\bP^{\pan}_k \rightarrow G'$ (\emph{cf.} \cite[1.5.14]{ducros}). We note that $|P|_x=|P|_{r_G(x)}$ for all $x \in \Hy_k$ (\emph{cf.} \cite[3.4.23.8]{ducros}). 

Let $(x_n)_n$ be a sequence in $\Hy_k$ converging with respect to the $\rho$-metric to a point $x \in \Hy_k$. Set~$y:=r_G(x)$. If $\rho(x,y)>0$, then $x \not \in G$, and for all but finitely many $n \in \N$, one has $\rho(x,x_n) < \rho(x,y)$, implying $x_n$ is contained in the same connected component of $\Hy_k \backslash G$ as $x$. Since $\varphi_P$ is constant on said component, $|P|_{x_n} \rightarrow |P|_x$ as $n  \rightarrow +\infty$. 

Otherwise, assume $\rho(x,y)=0$, \emph{i.e.} that $x \in G$. Set $y_n:=r_G(x_n)$ for $n \in \N$, so that~${|P|_{y_n}=|P|_{x_n}}$. Then $\rho(y_n,x) \leqslant \rho(x_n,x)$, seeing as the injective path $[x_n, x]$ contains~$y_n$ (\emph{cf.} \cite[1.5.15.1]{ducros}). Thus $y_n \rightarrow x$ as $n \rightarrow +\infty$ and $y_n \in G$ for all $n$. We can conclude by the continuity of $\varphi_P$ on $G$ that $|P|_{y_n} \rightarrow |P|_x$ as $n \rightarrow +\infty$. In conclusion, $\varphi_P$ is $\rho$-continuous. 

While the Berkovich topology is locally compact, this is no longer the case with the $\rho$-metric topology, so $\rho$ is strictly finer than the Berkovich topology.
\end{proof}

\begin{prop} \label{complete} The metric space $(\Hy_k, \rho)$ is complete. 
\end{prop}

\begin{proof} The proof of \cite[Proposition 2.28]{baker_rumely} in the case where $k$ is algebraically closed works \emph{mutatis mutandis} in this general case as well, using Lemma \ref{finer}. 
\end{proof}

We will need the following later on.

\begin{lem} \label{wedge} 
Given $p, q \in k$ and $x \in \bP^{\pan}_k$, let $p \wedge_{x} q$ denote the retraction of $x$ on the unique injective path $[p,q]$ in $\bP^{\pan}_k$ connecting $p$ and $q$. Given any coordinate function on $\bP^{\pan}_k$ such that $p,q$ are in the open unit disk, one has
$$|p-q|_k=e^{-\rho(\eta, p \wedge_{\eta} q)},$$
where $\eta$ denotes the Gauss point.  
\end{lem}

\begin{proof}
If $p,q$ are in the open unit disk, then $p \wedge_{\eta} q=\eta_{p, |p-q|_k}$, and $\eta=\eta_{p,1}$, implying $\rho(p \wedge_{\eta} q, \eta)=\log  \frac{1}{|p-q|_k}$.
\end{proof}

\subsection{Definitions of Schottky groups} \label{section_schottky}

In this section, $(k, |\cdot|)$ is an arbitrary complete valued field.

\begin{defi}[{\cite[Definition II.1.9]{poineau_turcheti_2}}] \label{loxodromic}
A transformation $\g \in \PGL_2(k)$ is said to be \emph{loxodromic} if given a representative in $\GL_2(k)$, its eigenvalues in $k^{\mathrm{alg}}$ have different absolute values (by \cite[\S~II.1.4]{poineau_turcheti_2}, the eigenvalues are in $k$).
\end{defi}

Let $\g \in \PGL_2(k)$ be loxodromic, and $A \in \GL_2(k)$ a representative whose eigenvalues we denote by $\lambda, \lambda'$. Assume $|\lambda|< |\lambda'|$, and set $y:=\frac{\lambda}{\lambda'}$. Then $0<|y|<1$. Let $\alpha, \beta \in \bP^1(k)$ be the eigenspaces of $\lambda, \lambda'$, respectively. Note that $\alpha \neq \beta$.

\begin{defi} \label{Koebe coordinates}
The coordinates of the triple $(\alpha, \beta, y)$ are called the \emph{Koebe coordinates} of $\g$, and we denote $\g=M(\alpha, \beta, y).$
\end{defi}

As over $\C$, a loxodromic transformation $\g \in \PGL_2(k)$  has 
exactly two fixed points in $\bP^1(k) \subseteq \bP^{\pan}_k$, one attractive and 
one repelling. If $\g=M(\alpha, \beta, y)$, then $\alpha$ 
is the attractive, and~$\beta$ the repelling point of $\g$. 
Taking the unique injective path $c:=[\alpha, \beta]$ 
connecting $\alpha$ and~$\beta$ in $\bP^{\pan}_k$, for any $x 
\in c \backslash \{\alpha,\beta\}$, one has $\rho(x, \g x)=-\log |y|$, which is the \emph{translation length} of~$\g$. The same holds when replacing $(\Hy_k, \rho)$ by $(\Hy^3, d_{\Hy^3})$ when $k=\C$, respectively $(\Hy^2, d_{\Hy^2})$ when~${k=\R}$, where $d_{\Hy^n}$ denotes the hyperbolic distance.

Inversely, given $\alpha \neq  \beta \in \bP^1(k)$, and $y \in k$ such that $0<|y|<1$, there exists a unique Möbius transformation $\g$ such that $\g=M(\alpha, \beta, y)$ (see \cite[\S~II.1.4]{poineau_turcheti_2} for explicit formulas). 

\begin{rem} For a classification of Möbius transformations in the context of Berkovich spaces, we refer the reader to \cite{dujardin_favre_degen}. For more classical references about general isometric actions on real trees, see  \emph{e.g.} \cite{culler_morgan}, \cite{paulin_bourbaki}, \cite{bridson_haefliger}, \cite{kapovich}.
\end{rem}
 
We study the actions of a particular type of subgroup of $\PGL_2(k).$ 
\begin{defi}[{\cite[Def. II.3.21]{poineau_turcheti_2}}] \label{schottky}
A subgroup $\Gamma$ of $\PGL_2(k)$ is said to be a \emph{Schottky group over $k$} if:
\begin{enumerate}
\item it is free and finitely generated,
\item \label{loxo} its non-trivial elements are loxodromic,
\item there exists a non-empty connected open subset of $\bP^{\pan}_k$ which is $\Gamma$-invariant and on which the action is free and proper. 
\end{enumerate}
\end{defi}
\begin{rem}  When $k=\C$, then $\bP^{\pan}_k=\bP^1(\C),$ and this coincides with the classical notion of Schottky group.
When $k$ is a local field, the above definition is equivalent to assuming $\Ga$ finitely generated and satisfying \eqref{loxo}, \emph{cf.} \cite[Corollary II.3.35]{poineau_turcheti_1}.  \end{rem}

 The minimal number of generators of a Schottky group is called its \textit{rank} and
a minimal set of generators is called a \emph{basis}. 
\begin{defi} \label{schottkyfig} Let $\Ga \subseteq \PGL_2(k)$ be a group generated by $\g_1, \g_2, \dots, \g_g \in \PGL_2(k)$. The group $\Ga$ is said to admit a \emph{Schottky figure adapted to $(\g_1, \ldots , \g_g)$} if there exist pairwise disjoint closed disks $D_{\g_i}, D_{\g_{i}^{-1}}$ in $\bP^{\pan}_k$, $i=1,2, \dots, g,$ such that 
\begin{equation} \label{eq_pingpong}
\g_i(D_{\g_i^{-1}}^c)=D^{\circ}_{\g_i} \ \mathrm{and} \ 
\g_i^{-1}(D_{\g_i}^c)=D^{\circ}_{\g_i^{-1}},
\end{equation} 
where $D^{\circ}_{\g_i}$, resp. $D^{\circ}_{\g_i^{-1}}$, is a maximal open disk inside $D_{\g_i},$ resp. $D_{\g_i^{-1}}$, for any~$i$.  
\end{defi}
 
\begin{rem} When $k$ is non-Archimedean, a closed disk in $\bP^{\pan}_k$ can contain infinitely many maximal (disjoint) open disks.   
\end{rem}

The following result was first proven by Gerritzen, and was then extended to the setting of Berkovich spaces in \cite[Proposition II.3.24, Theorem II.3.26]{poineau_turcheti_2}.

\begin{thm} \label{Gerritzen} Assume $k$ is non-Archimedean. 
A group $\Ga \subseteq \PGL_2(k)$ is a Schottky group over $k$ if and only if there exists a basis of $\Ga$ and a Schottky figure in $\bP^{\pan}_k$ adapted to it. 
\end{thm}

\begin{rem} 
Over $\C$, we know that there exist Schottky groups that don't admit  Schottky figures. One should instead take the closed subsets $D_{\g}$ in Definition~\ref{schottkyfig} to be bounded by Jordan curves in~$\bP^{1}(\C)$.
\end{rem}

In the non-Archimedean case, one can use \emph{twisted Ford disks} to explicitly construct Schottky figures for a well-chosen basis of $\Ga$, \emph{cf.} \cite[\S II.3.3]{poineau_turcheti_1}.

We include an explicit characterization of these groups, needed in Section~\ref{section_degen}.

\begin{cor} (\cite[Proposition 4.4.2]{poineau_turcheti_universal}) \label{cor_charac_Schottky} Let $k$ be a complete non-Archimedean field. For $g \geqslant 2,$ let $\alpha_1, \beta_1, \dots, \alpha_g, \beta_g$ be distinct points in $\bP^1(k)$. Let $y_1, y_2, \dots, y_g \in k$ be such that $0< |y_i|<1$ for all $i$. Set $\g_i=M(\alpha_i, \beta_i, y_i) \in \PGL_2(k), i=1,2,\dots, g$. The following conditions are equivalent:
\begin{enumerate}
\item[(i)] there exists a Schottky figure adapted to the basis $(\g_1, \g_2, \dots, \g_g),$
\item[(ii)] for all $ i \in \{1,2,\dots, g\}$ and for all $u,v \in \{\alpha_1, \beta_1, \ldots, \alpha_g, \beta_g\} \backslash \{\alpha_i, \beta_i\}$, one has 
\begin{equation*}
|y_i| \cdot |[u,v; \alpha_i, \beta_i]|<1
\end{equation*}
 where $[u,v; \alpha_i, \beta_i]$ denotes the cross-ratio of these points. 
\end{enumerate}
\end{cor}

\subsection{Limit set and Schottky disks} \label{section_coding}
Let $k$ be a complete valued field. Recall that if $k=\C$, then $\bP^{\pan}_k=\bP^{1}(\C)$, and if $k=\R$, then $\bP^{\pan}_{\R}=\bP^{1}(\C)/\mathrm{conjugation}$ (\emph{i.e.} conjugated elements of $\C$ are identified). 
\begin{defi} \label{limitset} Let $\Ga \subseteq \PGL_2(k)$ be a Schottky group. 
The \emph{limit set $\La_{\Ga}$ of $\Ga$} is the set of accumulation points of a non-trivial orbit {(\emph{i.e.} not reduced to a point)} for the action of $\Ga$ on~$\bP^{\pan}_k.$
\end{defi}

\begin{thm}[{\cite[\S II.3.1]{poineau_turcheti_1}}] The limit set $\La_{\Ga}$ of a Schottky group $\Ga$ does not depend on the choice of the non-trivial orbit, and is a $\Ga$-invariant compact subset of $\bP^{\pan}_k$ contained in~$\bP^1(k).$ 
\end{thm}

\begin{rem}
\begin{enumerate}
    \item    When $k$ is Archimedean, this is a classical result for arbitrary non-elementary subgroups of $\PGL_2(k)$  (see \emph{e.g.} \mbox{\cite[Corollaries 3.25, 3.26]{kapovich}}). 
    
\item  It was shown in \cite[Theorem 2.20]{dujardin_favre_degen} that, as in the Archimedean case, $\La_{\Ga}$ is the smallest non-empty $\Ga$-invariant closed subset of $\bP^{\pan}_k$. Equivalently, it is the closure of the fixed points of the non-trivial elements of $\Ga$.
 \item The non-triviality of the orbit in Definition \ref{limitset} is necessary \emph{only} when~$\Ga$ is of rank one; otherwise all orbits are non-trivial (two elements of a generating set fix different points).
\end{enumerate}
\end{rem}

Let $\mathcal{B}:=(\g_1, \dots, \g_g)$ be a basis of $\Ga$, and $\{D_{\g}, D_{\g}^{\circ}: l(\g)=1\}$ a Schottky figure adapted to it in $\bP^{\pan}_k$; here $l(\g)$ denotes the length of the reduced word $\g$ in the alphabet induced by $\mathcal{B}$. We recall that for $l(\g)=1$, the  set $D_{\g}$ is  a closed disk, and $D^{\circ}_{\g}$ a maximal open disk thereof. 

\begin{defi} \label{defi_coding}
Let $\g=wu \in \Ga$ be a reduced word of length $n \geqslant 2$ in this alphabet, with~${l(u)=1}$. To $\g$ we associate the closed disk $D_{\g}:=wD_u$, and a maximal open disk  $D_{\g}^{\circ}:=wD_u^{\circ}$ in it.
The disks $D_{\g}, D_{\g}^{\circ}, l(\g)=n$, are said to be \emph{$n$-th generation Schottky disks} for~${(\Ga, \mathcal{B})}$. 
\end{defi}

\begin{prop} \label{coding} The following properties are satisfied.
\begin{enumerate}
\item[(i)](\cite[Lemma II.3.6]{poineau_turcheti_2}) Given two reduced words $\g_1$ and $\g_2$ in $\Ga$, we have $D_{\g_1} \subseteq D_{\g_2}$ if and only if there exists $w \in \Ga$ such that $\g_2w$ is reduced and $\g_1=\g_2w.$ If, moreover, $\g_1 \neq \g_2$, then $D_{\g_1} \subseteq D^{\circ}_{\g_2}.$

\item[(ii)](\cite[ Corollary II.3.13]{poineau_turcheti_2}) Given a strictly decreasing sequence of Schottky disks $(D_{w_n})_n$, the intersection $\bigcap_n D_{w_n}$ is a singleton $\{l\}$ contained in $\La_{\Ga}$. Moreover, the disks $(D_{w_n})_n$ form a basis of neighborhoods of $l$. 

\item[(iii)] Conversely, given a point $l \in \La_{\Ga}$, there exists a sequence of reduced words $(w_n)_n$ in $\Ga$ such that the sequence of closed disks $(D_{w_n})_n$ is strictly decreasing and $\{l\}=\bigcap_{n} D_{w_n}.$
\end{enumerate}
\end{prop}

Essentially, the elements of $\La_{\Ga}$ can be identified with ``infinite reduced words" over the alphabet of $\Ga$ induced by $\mathcal{B}$. We mention some straightforward properties that we will use throughout the text.

\begin{rem} \label{rem_coding}
\begin{enumerate}
\item[(i)]
Given an $n$-th generation Schottky disk $D_{\g}$, there exist exactly $2g-1$ Schottky disks of generation $n+1$ contained in $D_{\g}$, each corresponding to the $2g-1$ reduced words of length $n+1$ starting with the word $\g.$ There are in total $2g(2g-1)^n$ Schottky disks of generation $n+1$.

\item[(ii)] Let $\g=a_1a_2\cdots a_n$ be a reduced word of length $n$ in $\Ga$. Then ${\g^{-1} D_{\g}=a_n^{-1} D_{a_n}=(D_{a_n^{-1}}^{\circ})^{c}},$ and similarly $\g^{-1}D_{\g}^{\circ}=D_{a_{n}^{-1}}^{c}$ (using the relations \eqref{eq_pingpong}). Moreover, $\g^{-1}(D_{\g} \backslash D_{\g}^{\circ})=D_{a_n^{-1}} \backslash D_{a_n^{-1}}^{\circ}$ and  $\g^{-1} \partial{D_{\g}}=\partial{D_{a_n^{-1}}}.$

\item[(iii)] Given $x \in \bP^{\pan}_k \backslash \bigcup_{l(\g)=1}{D_{\g}^{\circ}}$, one has $\g x \in D_{\g}$ for all $\g \in \Ga$. Let $\g=a_1\cdots a_n$ be reduced of length $n$. Then as $x \not \in D_{a_n^{-1}}^{\circ}$, using \eqref{eq_pingpong} we have that $a_nx \in D_{a_n}$, and so $\g x \in D_{\g}.$ If, moreover, $x \not \in \bigcup_{l(\g)=1} D_{\g}$, then $\g x \in D_{\g}^{\circ}$ for all $\g \in \Ga$. Using this, one can show 
Proposition~\ref{coding} (iii). 
\end{enumerate}
\end{rem}

\begin{rem} \label{rem_limitsetall}
{When $k$ is non-Archimedean},
the equality $\La_{\Ga}=\bP^1(k)$ is possible if and only if~$\bP^{1}(k)$ is compact, meaning if and only if~$k$ is locally compact. See also Remark \ref{rem_hdim_loccomp}. This is different from the Archimedean setting where  equality never occurs 
 because by definition the action admits a non-empty fundamental domain on $\bP^1(\C)$ (\emph{cf.} also \cite{maskit_characterization_schottky}).  
\end{rem}

We conclude with a result on the behaviour of Schottky groups with respect to base changes.
An extension of complete valued fields $K/k$ will be called a \emph{complete valued field extension} if the valuation on $K$ extends that of $k$. It induces a projection morphism $\pi_{K/k}: \bP^{\pan}_K \rightarrow \bP^{\pan}_k$, and for any $\alpha \in \bP^1(k)$, the set $\pi_{K/k}^{-1}(\alpha)$ is a singleton in $\bP^1(K) \subseteq \bP^{\pan}_K$. More precisely, the preimage $\pi_{K/k}^{-1}(\bP^{1}(k))$ can be \emph{homeomorphically} identified to the embedding $\bP^{1}(k) \subseteq \bP^1(K)$ induced by the field extension $K/k$.

\begin{lem} \label{lem_schottky_field_extension}
{Let $K/k$ be a complete valued field extension. A subgroup $\Ga \subseteq \PGL_2(k) \subseteq \PGL_2(K)$ is Schottky over $k$ if and only if it is Schottky over $K$. In that case, $\pi_{K/k}(\Lambda_K)=\Lambda_k$, where $\Lambda_K$ (respectively $\Lambda_k$) is the limit set of $\Gamma$ over~$K$ (respectively $k$), and the restriction of~$\pi_{K/k}$ to $\La_K$ is homeomorphic onto $\La_k$ with respect to the Berkovich topology.}  
\end{lem}

\begin{proof}
We note that $\Ga$ being free and finitely generated does not depend on an ambient space. Being loxodromic is a property on the eigenvalues of a matrix representing a transformation, so the non-trivial elements of $\Ga$ are loxodromic in $\PGL_2(k)$ if and only if they are loxodromic in $\PGL_2(K)$. It remains to show that there is a non-empty connected open $\Ga$-invariant subset on which $\Ga$ acts freely and properly in $\bP^{\pan}_{k}$ if and only if there is one in $\bP^{\pan}_K$. By Proposition~I.5.1 of~\cite{poineau_turcheti_1}, $\pi_{K/k}$ is continuous, surjective, and proper. Using the extension of Möbius transformations on the entire Berkovich projective line (\emph{cf.} \cite[\S~II.1.3]{poineau_turcheti_2}), for any $x \in \bP^{\pan}_K$ and any $g \in \Ga$, one has $\pi_{k/k}(gx)=g \pi_{K/k}(x)$. 

Assume that $\Ga$ is Schottky over $k$. Then $\bP^{\pan}_k \backslash \La_k$ is a non-empty connected open $\Ga$-invariant subset of $\bP^{\pan}_k$ on which $\Ga$ acts properly and freely (\cite[Theorem~II.3.18]{poineau_turcheti_2}). As $\La_k \subseteq \bP^1(k)$, one has that~$\pi_{K/k}^{-1}(\La_k)$ is a subset of $\bP^1(k) \subseteq \bP^{1}(K)$ in $\bP^{\pan}_K$, and so~$\bP^{\pan}_K \backslash \pi_{K/k}^{-1}(\La_k)$ is connected. It is also non-empty, open and $\Ga$-invariant. As $\Ga$ acts freely and properly on~$\bP^{\pan}_k \backslash \Lambda_k$, it acts freely and properly on $\pi_{K/k}^{-1}(\bP^{\pan}_k \backslash \Lambda_k)=\bP^{\pan}_K \backslash \pi_{K/k}^{-1}(\La_k)$, so it is Schottky over $K$.

Assume {conversely} that $\Ga$ is Schottky over $K$. Let $\alpha \in \pi_{K/k}^{-1}(\bP^{1}(k))$, so $\Ga \alpha \subseteq \pi_{K/k}^{-1}(\bP^{1}(k))$. As $k$ is complete, $\pi_{K/k}^{-1}(\bP^{1}(k))$ is closed in $\bP^{1}(K)$ with respect to the Berkovich topology, implying  $\La_{K} \subseteq \pi_{K/k}^{-1}(\bP^{1}(k)).$ As $\La_K$ is compact in $\bP^{\pan}_K$,  its image $\pi_{K/k}(\La_K)$ is compact in $\bP^{\pan}_k$, so $\pi_{K/k}(\bP^{\pan}_K \backslash \Lambda_K)$ is a non-empty open subset of $\bP^{\pan}_k$. Morever, it is $\Ga$-invariant and connected as the same is true for $\bP^{\pan}_K \backslash \Lambda_K$. By the properness of $\pi_{K/k}$, as $\Ga$ acts properly on $\bP^{\pan}_K \backslash \Lambda_K$, it also acts properly on    
$\pi_{K/k}(\bP^{\pan}_K \backslash \Lambda_K)$. Assume there exist $g \in \Ga \backslash \{\mathrm{id}\}$ and $x \in \pi_{K/k}(\bP^{\pan}_K \backslash \Lambda_K)$ such that $gx=x$. As $g$ is loxodromic over $K$, it is also loxodromic over $k$. As $x$ is a fixed point of $g$, one has $x \in \bP^{1}(k)$. This implies that $\pi_{K/k}^{-1}(x)$ is a singleton contained in $\pi_{K/k}^{-1}(\bP^1(k)) \subseteq \bP^{1}(K)$, and so $g \pi_{K/k}^{-1}(x)=\pi_{K/k}^{-1}(x)$. Since the action of $\Ga$ on $\bP^{\pan}_K \backslash \Lambda_K$ is free, we obtain $g=\mathrm{id}$, contradiction. Thus, the action of $\Ga$ on $\pi_{K/k}(\bP^{\pan}_K \backslash \Lambda_K)$ is also free, implying $\Ga$ is Schottky over~$k$. 

Finally, for any $\alpha \in \pi_{K/k}^{-1}(\bP^{1}(k))$, we have $\pi_{K/k}(\Ga \alpha)=\Ga \pi_{K/k}(\alpha)$. As $\pi_{K/k}$ is continuous, we obtain that $\pi_{K/k}(\La_K)=\La_{k}$. Moreover, seeing as $\La_{K} \subseteq \pi_{K/k}^{-1}(\bP^{1}(k))$, the restriction of $\pi_{K/k}$ to~$\La_K$ is a homeomorphism onto its image.
\end{proof}

\subsection{Schottky and Mumford uniformizations}
It was shown by Koebe that any compact Riemann surface can be uniformized analytically via Schottky groups; this is known as the \emph{Schottky uniformization}. A similar result is true for \emph{Mumford curves} in the non-Archimedean case (see below). We recall that a Mumford curve is a projective algebraic curve with split degenerate stable reduction. Equivalently, its Berkovich analytification is locally isomorphic to certain opens of $\bP^{\pan}_k$ (\emph{cf.} \cite[Remark II.2.28]{poineau_turcheti_2}). We emphasize that, unlike the complex setting, over a non-Archimedean field, not all smooth curves satisfy this property. 

Let $k$ be a complete valued field.  Let $\Ga$ be a Schottky group over $k$ with a well-chosen basis $\mathcal{B}=\{\g_i: i=1,2,\dots, g\}$ and a Schottky figure $\{D_{\g}, D_{\g}^{\circ}: l(\g)=1\}$ adapted to it (recall that these are closed, respectively open, disks in $\bP^{\pan}_k$). We denote by $\La_{\Ga}$ its limit set. 

\begin{lem} \label{domain} 
The set $F:=\bP^{\pan}_k \backslash (\bigcup_{i=1}^g {D_{\g_i}} \cup \bigcup_{i=1}^g D_{\g_i^{-1}}^{\circ})$ is a fundamental domain for the action of $\Ga$ on $\bP^{\pan}_k \backslash \La_{\Ga}$. 
\end{lem}

\begin{proof}
Let $x \in \bP^{\pan}_k \backslash  (F \cup \Lambda_{\Ga})$. Since $x$ is not in the limit set, using Proposition \ref{coding}, 
there exists $\g \in \Ga$ such that $x \in D_{\g} \backslash 
\bigcup_{l(\g a)=l(\g)+1} D_{\g a}$, where $\g a$ is a reduced word. 
Then either $x \in D_{\g}^{\circ} \backslash 
\bigcup_{l(\g a)=l(\g)+1} D_{\g a}$, in which case $\g^{-1}x \in (\bigcup_{l(a)=1} D_{a})^c \subseteq F$ (\emph{cf.} (ii) in Remark 
\ref{rem_coding}), or $x \in D_{\g} \backslash D_{\g}^{\circ}.
$ In that case, by \emph{loc.cit.}, there exists $a \in \Ga$ with 
$l(a)=1$ such that $\g^{-1}x \in D_a \backslash D_a^{\circ}$. If 
$a \in \{\g_i^{-1}\}_i$, then $\g^{-1}x \in F$. Otherwise, 
$a^{-1}\g^{-1}x \in D_{a^{-1}} \backslash D_{a^{-1}}^{\circ} 
\subseteq F$. 
  
Assume there exist $x \in F$ and $\g \in \Ga$ such that $\g x \in F$. Clearly, this implies $l(\g)=1$.
By assertion  (iii) of Remark \ref{rem_coding}, $x \in \bigcup_{l(a)=1} D_a \backslash D_a^{\circ}$, and so $\g=a^{-1}$. As $x \in F$, $a \in \{\g_i^{-1}\}_i$, and then $a^{-1} x \in D_{a^{-1}} \backslash D_{a^{-1}}^{\circ}$, with $a^{-1} \in \{\g_i\}_i$, so $\g x \not \in F$, contradiction. 
\end{proof}

{The quotient $(\bP^{\pan}_k \backslash \La_{\Ga})/ \Ga$ is a proper analytic curve of genus $g$; it is a compact Riemann surface when $k$ is Archimedean, and a Mumford curve when $k$ is non-Archimedean.}  

\begin{thm} {Assume $k$ is non-Archimedean.}
Given a Mumford curve $C$ over $k$, its fundamental group $\Ga$ 
is a Schottky group over $k$. Moreover, $C^{\an} \cong
(\bP^{\pan}_k \backslash \La_{\Ga})/{\Ga}$, where $C^{\an}$ is the Berkovich analytification of $C$. 
\end{thm}

\begin{rem} 
This result was initially shown by Mumford in \cite{mumford}; recently, Poineau and Turchetti obtain a proof using the formalism of Berkovich spaces in \cite[Theorem II.4.3]{poineau_turcheti_2}. 
\end{rem}

We conclude this section with an important consequence of Lemma \ref{domain} and a formula that we will need later on.
Let $x_a$ denote the unique boundary point of the Schottky disk~$D_a$ for~${a \in \Ga \backslash \{\mathrm{id}\}}$. 
\begin{lem} \label{lem_length_expansion} Assume $k$ is non-Archimedean. Let $x,y  \in \Hy_k \cap F$. 
Let $\g=a_1 a_2 \cdots a_n \in \Ga$ be a reduced word of length $n \geq 1$. Then 
$$\rho(x, \g y)=\rho(x, x_{a_1}) + \sum_{m=2}^{n}\rho(x_{a_{m-1}^{-1}}, x_{a_{m}}) + \rho(x_{a_{n}^{-1}}, y),$$
where $\rho$ denotes the interval-length metric on $\Hy_k$ from Section \ref{section_hyperbolic}.
\end{lem}

\begin{proof}
Let $[x, \g y]$ be the unique injective path in $\bP^{\pan}_k$ connecting $x$ and $\g y$. If $n=1$, then $\partial{D}_{\g} \in [x, \g y]$. Otherwise, since $\gamma y \in D_\g$, one has $ x, \g y \in D_{\g} \backslash D_{\g}^{\circ}$, which implies $y \in D_{\g^{-1}} \backslash D_{\g^{-1}}^{\circ}$ by Remark \ref{rem_coding} (ii). This is impossible seeing as $x,y \in F$, and $x \in D_{\g}, y \in D_{\g^{-1}}$. 

If $n \geqslant 2$, then $\g y \in D_{\g} \subsetneq D_{a_1}^{\circ}$ and $x \not \in D_{a_1}^{\circ}$, so the boundary point $x_{a_1}$ is contained in $[x, \g y].$ Similarly, $\g y \in D_{a_1a_2}$ and $x_{a_1} \not \in D_{a_1a_2}$, implying $x_{a_1a_2} \in [x_{a_1}, \g y]$. Continuing with this reasoning, we obtain that for any $m \in \{2,\dots, n\}$,
$x_{a_1\cdots a_m} \in [x_{a_1\cdots a_{m-1}}, \g y]$. Thus,

\smallskip

{\centering
  $ \displaystyle
    \begin{aligned}
 \rho(x, \g y)=\rho(x, x_{a_1}) + \sum_{m=2}^n \rho(x_{a_1\cdots a_{m-1}}, x_{a_1 \cdots a_{m}}) + {\rho(x_{a_1\cdots a_{n}}, \g y)}.
\end{aligned}
  $ 
\par}

\smallskip
\noindent For $m=3,4,\dots, n,$ as $D_{a_1\cdots a_{m-1}}=a_1a_2 \cdots a_{m-2} D_{a_{m-1}}$ and $D_{a_1 \cdots a_m}=a_1a_2 \cdots a_{m-2} D_{a_{m-1}a_m}$, by the $\PGL_2$-invariance of $\rho$, we obtain $\rho(x_{a_1\cdots a_{m-1}}, x_{a_1 \cdots a_{m}})=\rho(x_{a_{m-1}}, x_{a_{m-1}a_{m}}).$

Moreover, as $a_{m-1}^{-1}x_{a_{m-1}}=x_{a_m^{-1}}$ (\emph{cf.} (ii) of Remark \ref{rem_coding}), one has  $\rho(x_{a_{m-1}}, x_{a_{m-1}a_{m}})=\rho(x_{a_{m-1}^{-1}}, x_{a_m})$ for any $m=2,3,\dots, n$. Thus,
$\rho(x, \g y)=\rho(x, x_{a_1}) + \sum_{m=2}^n \rho(x_{a_{m-1}^{-1}}, x_{a_m}) + \rho(x_{a_1\cdots a_{n-1}}, \g y).$
Similarly, {$\rho(x_{a_1\cdots a_{n}}, \g y)=\rho(x_{a_n^{-1}}, y)$}, thus concluding the proof. 
\end{proof}

\begin{lem} \label{lem_calc_2} Assume $k$ is non-Archimedean and let $x \in \mathbb{H}_k \setminus (\bigcup_{i=1}^g (D_{\gamma_i} \cup D_{\gamma_i^{-1}})) $. Let $\alpha, \alpha_1, \dots, \alpha_{N-1}, \beta \in \Ga$ be reduced words of length $m$.  Assume  there exist letters $a_1, a_2, \dots, a_{N+m}$ in $\Ga$ such that $\alpha_i=a_{i+1}a_{i+2} \cdots a_{m+i}$ for $i=0,1,\dots, N$, where $\alpha_0:=\alpha$ and $\alpha_N:=\beta$. Then one has:

{\centering
  $ \displaystyle
    \begin{aligned}
 \rho(x, a_{N+1}x)-\rho(x, a_1\cdots a_{N+1} x) &=\rho(x, {x}_{a_{N+1}}) - \rho(x, {x_{a_1}})- \sum_{i=1}^N \rho({x_{a_{i}^{-1}}}, {x_{a_{i+1}}}).
\end{aligned}
  $ 
\par}
\end{lem}

\begin{proof}
Let $S$ denote the {expression} on the right hand side. Since $x \in \Hy_k \cap F$, by Lemma \ref{lem_length_expansion}:
$$ \rho(x , a_1 \cdots a_{N+1}x) = \rho(x, x_{a_1})+\rho(x, {x_{a_{N+1}^{-1}}}) +\sum_{i=1}^N \rho({x_{a_i^{-1}}}, {x}_{a_{i+1}}),
$$
so $\sum_{i=1}^N \rho({x_{a_i^{-1}}}, {x_{a_{i+1}}})= \rho(x, a_1\cdots a_{N+1} x)-\rho(x, {x_{a_1}})-\rho(x, {x_{a_{N+1}^{-1}}}).$
Consequently, $$S=\rho(x, {x_{a_{N+1}}}) -\rho(x, {x_{a_1}}) + \rho(x, {x_{a_1}}) + \rho(x, {x_{a_{N+1}^{-1}}})-\rho(x, a_1\cdots a_{N+1} x).$$ As $\rho(x, {x_{a_{N+1}}})+ \rho(x, x_{a_{N+1}^{-1}})=\rho(x, {x_{a_{N+1}}})+\rho(a_{N+1}x, {x_{a_{N+1}}})$ and $x \not \in D_{a_{N+1}}$ but $a_{N+1}x \in D_{a_{N+1}},$ we have $\rho(x, {x_{a_{N+1}}})+\rho(a_{N+1}x, {x_{a_{N+1}}})=\rho(x, a_{N+1}x)$, and we may conclude.
\end{proof}

\subsection{Gromov compactification versus Berkovich topology}
\label{gromov}
When $k = (\C, |\cdot|_\infty)$ {(respectively $k = (\R, |\cdot|_\infty)$)}, one can take $(\Hy_k, \rho)$ to be the hyperbolic three space {(respectively the hyperbolic two-space)} endowed with its hyperbolic metric and take $\bP^1(\C)$ {(respectively~$\bP^1(\R)$)} as its boundary.
A construction due to Gromov allows one to recover the Hausdorff topology  on the Riemann sphere or the unit circle using the Gromov product (see \cite[Chapter III.H]{bridson_haefliger}).   
We apply this construction to the case of $(\Hy_k,\rho)$ with $k$ a non-Archimedean field (\emph{cf.} Section \ref{section_hyperbolic}), and compare it with the Berkovich projective line $\bP^{\pan}_k$ and its topology. See also \mbox{\cite[3.4]{das_simmons_urbanski}}. 

Unless stated otherwise, assume $k$ is a non-Archimedean field. For $x,y,z \in \Hy_k$, let $x \wedge_z y$ be the retraction of $z$ on the unique injective path $[x,y]$ in $\Hy_k$ joining $x$ and $y$; it is at minimal~$\rho$-distance from the point $z$. Let us fix a point $\eta \in \Hy_k$. 

\begin{defi}
The \textit{Gromov boundary} $\partial_{G} \Hy_k$ of $\Hy_k$ is the quotient of sequences $(x_n)_n \subseteq \Hy_k$ such that $\lim_{n,m \rightarrow +\infty}\rho(\eta , x_n\wedge_{\eta} x_m) = +\infty$ by the equivalence relation for which two sequences $(x_n)_n, (y_n)_n$ in $\Hy_k$ are  equivalent if 
\begin{equation*}
\lim_{n,m \rightarrow +\infty}    \rho(\eta, x_n \wedge_{\eta} y_m)= +\infty. 
\end{equation*}

The \emph{Gromov compactification} of $\Hy_k$ is defined to be the disjoint union $\Hy_k \cup \partial_G{\Hy_k}$ with the following topology: on $\Hy_k$ the topology induced from $\rho$, and for $\partial_{G}\Hy_k$, a basis of neighborhoods for $[(x_n)_n] \in \partial_G{\Hy_k}$ is given by the sets: 
\begin{equation}
    N_t([(x_n)_n]) := \{ y \in \Hy_k \cup \partial_G \Hy_k \  | \liminf_{n,m \rightarrow +\infty} \rho(\eta , x_n \wedge_\eta y_m) > t \}, t>0,  
\end{equation}
where if $y \in \Hy_k$, then $y=:y_m$ for all $m$, and if $y \in \partial_G{\Hy_k}$, then $(y_n)_n$ is such that $[(y_n)_n]=y$. 
\end{defi}  

Note that one can prove, as in \cite[Remark 3.4.5]{das_simmons_urbanski}, that the Gromov compactification does not depend on the point $\eta$, and nor do the sets $N_t([(x_n)_n])$ depend on the representative of $[(x_n)_n]$.  In particular, $x_m \rightarrow [(x_n)_n]$ as $m \rightarrow +\infty$ in the Gromov topology.

Let $[(x_n)_n] \in \partial_G{\Hy_k}.$ As $\bP^{\pan}_k$ is compact,  $(x_n)_n$ has a convergent subsequence. Let $x$ be its limit.

\begin{lem}
The function $\varphi: \partial_G \Hy_k \rightarrow \bP^{\pan}_k \backslash \Hy_k$, $[(x_n)_n] \mapsto x,$ is well-defined. Moreover, the sequence $(x_n)_n$ converges to $x$ in $\bP^{\pan}_k.$
\end{lem}

\begin{proof}

For $y \in \Hy_k$, let $r_y$ denote the retraction of $\bP^{\pan}_k$ to $[\eta, y]$ (\emph{cf.} \cite[1.5.14]{ducros}).  The sequence $(r_y(x_n))_n$ in $[\eta, y]$ is eventually constant: otherwise, let us restrict to a subsequence $(u_n)_n$ where all terms are different. Let $(x_n')_n$ be a subsequence of $(x_n)_n$ such that $r_y(x_n')=u_n$. Then $x_n' \wedge_{\eta} x_m' \in \{u_n, u_m\},$ so $\rho(\eta, x_n' \wedge_{\eta} x_m') \leqslant \rho(\eta, y)$ for all $m, n$, impossible as the sequence $(\rho(\eta, x_n' \wedge_{\eta} x_m'))_{m, n}$  is unbounded. We may thus assume that $(r_y(x_n))_n$ is constant.

In particular, if $x \in \Hy_k$, then let $p \in [\eta, x]$ denote the point $r_x(x_n), n \in \N.$ As the Gromov topology does not depend on $\eta$, we may assume $x \neq \eta.$ 
Now let $(y_n)_n$ be a subsequence of~$(x_n)_n$ converging towards $x$ in $\bP^{\pan}_k$. If $r_x(y_n)=p\neq x$, then for a connected neighorhood $U$ of $x$ such that $p \not \in x$, one has $y_n \not \in U$ for all $n$, impossible. Thus, $p=x$, and so $(x_{n})_n$ is contained in branches $b$ emanating from $x$ and which are different from the one corresponding to $[\eta, x]$. If two such branches contained each a subsequence $(s_n)_n$ and $(t_n)_n$ of $(x_{n})_n$, then $\rho(\eta, s_n \wedge_{\eta} t_n)=\rho(\eta, x)$ for all $n$, which contradicts the assumption $[(x_n)_n] \in \partial_G \Hy_k$. Thus, there exists a unique such branch $b$ containing almost all terms of the sequence $(x_{n})_n$.

For a fixed $N \in \N,$ set $w:=x_{N} \wedge_{\eta} x_{N+1} \in b$. We may assume the sequence $r_{w}(x_n)$ to be constant. By construction, $q:=r_{w}(x_n) \in [x,w] \backslash \{x\}$. Let $C$ be the connected component of~$\bP^{\pan}_k \backslash \{q\}$ containing $x$. Then $(y_n)_n \not \subseteq C$, impossible, seeing as $y_n \rightarrow x$. Thus, $x \not \in \Hy_k.$

Let us show that $x_n \rightarrow x$ as $n \rightarrow +\infty$ 
in $\bP^{\pan}_k$. Let $U$ be an affinoid neighborhood of $x$ not 
containing $\eta$; we recall that its boundary $\partial{U}$ is a 
finite set of points in~$\Hy_k$. Let ${\{b_1, b_2, \dots, 
b_m\}:=r_{x}(\partial{U})}$. As $\eta \not \in U$, the set $
\partial{U} \cap [\eta, x]$ is non-empty; let $b_1$ be the 
corresponding intersection point. Let us also suppose that $b_i \in 
[\eta, b_m]$ for all $i=1,2,\dots, m-1$.

Assume there exists a subsequence $(z_n)_n$ of $(x_n)_n$ not contained in $U$. Then either $r_x(z_n) \in \{b_2, \dots, b_m\}$ or $r_x(z_n) \in [\eta, b_1]$, meaning $r_x(z_n) \in [\eta, b_m],$ so $r_x(z_n)=r_{b_m}(z_n)$ for all $n$. As $b_m \in \Hy_k$, we may assume that $r_x(z_n)$ is a constant point $q \in [\eta, b_m]$. Let $V$ be the connected component of $\bP^{\pan}_k \backslash \{b_m\}$ containing $x$. Then $V$ contains almost all terms of the sequence~$(y_n)_n$ which converges to $x$, and so $r_x(y_n) \in [b_m, x] \backslash \{b_m\}$. Thus, for $n$ large enough, $y_n \wedge_{\eta} z_n = q$, so $\rho(\eta, y_n \wedge_{\eta} z_n) \leqslant \rho(\eta, b_m)$, contradiction. Hence, the entire sequence $(x_n)_n$ converges to $x$ as~$n \rightarrow +\infty$.

Finally, assume $(x_n)_n, (z_n)_n$ are such that $[(x_n)_n]= [(z_n)_n] \in \partial_G{\Hy_k}$. Let $x_n \rightarrow x$ and $z_n 
\rightarrow z$ as $n \rightarrow + \infty$ in $\bP^{\pan}_k$. Assume $x \neq z$.  Let 
$d:=x \wedge_{\eta} z \in \Hy_k$. As $x,z$ 
are of type $1$, $d \not \in \{x, z\}$. Let~${U_x, U_z}$ be disjoint 
connected neighborhoods of $x,{z}$, respectively, not containing $d$. Then almost all 
terms of the sequence $(x_n)_n$ are contained in $U_x$, and 
similarly for $(z_n)_n$ and~$U_z$. This implies that $x_n 
\wedge_{\eta} z_n=d$ for $n$ large enough, meaning $(\rho(\eta, 
x_n \wedge_{\eta} z_n))_{n}$ is bounded, which contradicts 
the assumption $(x_n)_n \sim (z_n)_n$. Thus $x=z$, and we 
conclude that $\varphi$ is well-defined. 
\end{proof}

\begin{lem}
The extension of {$\mathrm{id}_{\Hy_k} : \Hy_k \rightarrow \Hy_k$} by $\varphi$ to the map $\phi: \Hy_k \cup \partial_G{\Hy_k} \rightarrow \bP^{\pan}_k$ is bijective. 
\end{lem}

\begin{proof}
Let us show that $\phi$ is injective. It suffices to prove it for $\varphi=\phi_{|\partial_G{\Hy_k}}$. Let $[(x_n)_n], [(y_n)_n] \in \partial_{G}{\Hy_k}$ be such that they have the same limit $x$ in $\bP^{\pan}_k$. Let $U$ be any connected neighborhood of $x$. Then for $m, n$ large enough, as $x_n, y_m \in U$, one obtains $w_{m,n}:=x_n \wedge_{\eta} y_m \in U$. Consequently, $w_{m,n} \rightarrow x$ as $m,n \rightarrow +\infty$ in $\bP^{\pan}_k$. Assume there exists a bounded subsequence of $(\rho(\eta, w_{m,n}))_{m,n}$, bounded by $M>0$. Let $z \in [\eta, x]$ be such that $\rho(\eta, z)>M$. 
Let $V$ be the connected component of $\bP^{\pan}_{k}\setminus \{ z\}$ containing $x$. Then, by construction, there exists a subsequence of $(w_{m, n})_{m,n}$ not contained in $V,$ contradiction. Thus, $\lim_{m,n} \rho(\eta, w_{m,n}) = +\infty$, meaning $[(x_n)_n] =[(y_n)_n].$   

To show that $\varphi$ is surjective, let $x \in \bP^{\pan}$ be of type $1$, and for $n \in \N$, let $x_n \in [\eta, x]$ be such that $\rho(\eta, x_n)=n$. Then, $[(x_n)_n] \in \partial_G{\Hy_k}$ and ${\varphi([(x_n)_n])=x}$. 
\end{proof}

\begin{thm}
The identity map $(\Hy_k, \rho) \rightarrow (\Hy_k, \mathrm{Berk})$ can be extended to a bijective continuous map $\phi: \Hy_k \cup \partial_G\Hy_k \rightarrow \bP^{\pan}_k$.
Moreover, a sequence $(a_n)_n$ in $\Hy_k$ converges to a point $a \in \partial_G{\Hy_k}$ in $\Hy_k \cup \partial_G{\Hy_k}$ if and only if it converges to $\varphi(a)$ in $\bP^{\pan}_k$. 
\end{thm}

\begin{proof} Let $x \in \bP^{\pan}_k$ and $U$ a connected open neighborhood of $x$ in $\bP^{\pan}_k$. We may assume that $\partial{U} \subseteq \Hy_k$ is finite. If $x \in \Hy_k$, let $r:=\min_{y \in \partial{U}} \rho(x, y).$ Then the open $\rho$-disk $B_{\rho}(x, r)$ is contained in $U$. 

Assume now that $x$ is of type 1, and that $\eta \not \in U$. Let $[(x_n)_n] \in \partial_G{\Hy_k}$ such that ${x_n \rightarrow x}$ in~$\bP^{\pan}_k.$
We consider again the retraction $r$ of $\bP^{\pan}_k$ on $[\eta, x]$. Let $b \in r(\partial{U})$ such that $\rho(\eta, b)=\max_{y \in \partial{U}} {\rho(\eta, r(y))}.$
Note that for $p_n:=r(x_n),$ one has $p_n \rightarrow x$ as $n \rightarrow +\infty$ using the continuity of the retraction $r$ on $\bP^{\pan}_k$ (\emph{cf.} \cite[Théorème 1.6.13]{ducros}). %
We assume, without loss of generality, that $(p_n)_n \subset [b, x].$ Set $t:=\rho(\eta, b).$ Note that, by construction, if $z \in \Hy_k$ is such that $\rho(\eta, z \wedge_{\eta} x_n) >t$ for at least one $x_n \in U$, then $z \in U.$
 
Let us show that $N_t([(x_n)_n]) \subseteq \phi^{-1}(U)$. If $z \in N_t([(x_n)_n]) \cap \Hy_k,$ it is immediate that $z=\phi(z) \in U$. If $z=[(z_n)_n] \in N_t([(x_n)_n]) \cap \partial_G{\Hy_k}$, then by construction of $t$, there exists a subsequence of $(z_n)_n$ contained in $U$. Let $z:=\phi([(z_n)_n]) \in \bP^{\pan}_k$. If $z \not \in U$, there exists a neighborhood $V$ of $z$ in $\bP^{\pan}_k$ such that $V \cap U=\emptyset$, contradiction. Thus $z \in U$, concluding the proof that $\phi$ is continuous.  

Let $(a_n)_n$ be a sequence in $\Hy_k$ converging to a type $1$ point $a$ in $\bP^{\pan}_k$. Let $b_n=r_a(a_n)$, where~$r_a$ is the retraction of $\bP^{\pan}_k$ to $[\eta, a]$. As already remarked, $b_n \rightarrow a$ in $\bP^{\pan}_k$ when ${n \rightarrow +\infty}$. As $a$ is a type 1 point, it is non-branching, so $b_n \neq a$ for all $n$. Then we remark that for any $m,n$, $\rho(\eta, a_m \wedge_{\eta}, a_n) \geqslant \min(\rho(\eta, b_n), \rho(\eta, b_m)) \rightarrow + \infty$ as $m,n \rightarrow +\infty$. This implies $\phi^{-1}(a_n)=a_n \rightarrow [(a_m)_m]=\phi^{-1}(a)$ as $n \rightarrow +\infty$ in $\Hy_k \cup \partial_G{\Hy_k}$. The other direction is due to the continuity of~$\phi.$
\end{proof}

{For $k$ any complete valued field,} let $\Ga \subseteq \PGL_2(k)$ be a {Schottky} group. As it acts on $\Hy_k$ by isometries, its action can be extended to $\Hy_k \cup \partial_G{\Hy}$. One can define the associated \emph{limit set}~$\La_{\Ga}^G$ with respect to the Gromov topology. {If $k$ is Archimedean, we continue to denote by $\phi$ the identification of the Gromov compactification of $\Hy_k$ to $\Hy_k \cup \bP^{1}(k)$ (by \emph{e.g.} \cite[Chapter~III.H, Proposition 3.11, Lemma 3.11, Examples II.8.11]{bridson_haefliger}}). We note that in this case $\phi$ is a homeomorphism. 
As a consequence of the previous theorem:

\begin{cor}{Let $k$ be a complete valued field}. Let $\Ga$ be a Schottky group over $k$. Then ${\phi(\La_{\Ga}^{G})=\La_{\Ga}}$.
\end{cor}

We will need the following classical notion {of \emph{radial convergence} later on, which can also be referred to as \emph{conical convergence} (see \emph{e.g.} \cite[Definition 1.14]{kapovich_leeb}):} 

\begin{defi} [{\cite[Definition 7.1.2]{das_simmons_urbanski}}]  \label{radiallimitset} Let $k$ be a complete valued field. 
A sequence $(x_n)_n \subseteq \Hy_k$ is said to \emph{converge 
radially} toward $x \in  \bP^{\pan}_k \backslash \Hy_k$ if $x_n \rightarrow x$ in $\bP^{\pan}_k$ when $n \rightarrow +\infty$, and there exists $\sigma > 0$ such that
$\rho(x_n, \eta \wedge_{x_n} x ) \leqslant \sigma$ for all $n$.

\textit{The radial limit set} $\Lambda_{\Ga}^r$ is the set of points $x \in \bP^{\pan}_k \backslash \Hy_k$ for which there exists a sequence $(\gamma_n)_n \subseteq \Ga$ such that $\gamma_n \eta$ converges radially toward $x$ when $n \rightarrow +\infty$. 
\end{defi}

\begin{lem} \label{prop_radial}Let $k$ be a complete valued field. If $\Ga$ is a Schottky group over $k$, then $\Lambda_{\Ga}^r = \Lambda_{\Ga}$. 
\end{lem}

\begin{proof} {When $k$ is Archimedean, this follows from \cite[Theorem 1.38]{kapovich_leeb} since Schottky groups do not contain any parabolic elements. Let us assume that $k$ is non-Archimedean.}
By \cite[Corollary 7.4.4]{das_simmons_urbanski}, $\La_{\Ga}^r \subseteq \phi(\La_{\Ga}^G)=\La_{\Ga}$. Let  $x \in \Lambda_{\Ga}$. Let us take a Schottky figure for $\Ga$ in $\bP^{\pan}_k$ adapted to a well-chosen basis. 
By Proposition \ref{coding} (iii), there exists a strictly decreasing sequence of Schottky disks $(D_{\g_n})_n$ in this figure such that $\{x\} = \bigcap_{n} D_{\gamma_n}$. Choose $\eta \in \Hy_k$ so that it is not in this Schottky figure, meaning $\g_n \eta \in D_{\g_n} \backslash D_{\g_{n+1}}$ for all $n$ and~${\g_n \eta \rightarrow x}$ in $\bP^{\pan}_k$.   

We also note that $\partial{D_{\g_n}} \in [\eta, x]$ and $\partial{D}_{\g_{n+1}} \in [\partial{D_{\g_n}}, x]$ for all $n$. 
Then $\rho(\g_n\eta, \eta \wedge_{\g_n \eta} x)=\rho(\g_n \eta, 
\partial{D_{\g_n}} \wedge_{\g_n \eta} \partial{D_{\g_{n+1}}})=\rho(\eta, \partial{D}_a \wedge_{\eta} \partial{D}_b)$, where the last equality is due to the $\PGL_2(k)$-invariance of $\rho$, and $a,b \in \Ga$ are such that $l(a)=l(b)=1$ (\emph{cf.} Remark \ref{rem_coding} (ii)). As the set $\{\rho(\eta, \partial{D}_c \wedge_{\eta} \partial{D}_d): l(c)=l(d)=1\}$ is finite, there exists $\sigma>0$ such that for any $n$, $\rho(\g_n\eta, \eta \wedge_{\g_n \eta} x) \leqslant \sigma,$ meaning $(\g_n \eta)_n$ converges radially toward $x$, hence $x \in \La_{\Ga}^r$.             
\end{proof}

\section{Hausdorff dimension and Berkovich spaces} \label{sect_3}
In this section we study the Hausdorff dimension of limit sets of Schottky groups defined over a complete non-Archimedean field using Berkovich analytic spaces. We start by relating the Hausdorff dimension with the critical exponent of a Poincaré series, which we then interpret as the spectral radius of a Perron matrix. We end the section with some bounds {on} and examples of Hausdorff dimensions. 
\subsection{The Hausdorff dimension as a critical exponent} \label{sect_hausdim}

In this subsection, we recall the definition of the Hausdorff dimension (which gives a notion of size for fractals) and show some useful properties. Our main focus will be subsets of the projective line. See \emph{e.g.} \cite{falconer} for a general reference.

Let $(X, d)$ be a metric space and $E \subseteq X$. 
For $s, \delta>0$, let 
\begin{equation*}
\mathcal{H}_{s,\delta}(E) := \inf \left\{ \sum_{i=1}^{\infty} (\diam P_i)^s | \diam(P_i)< \delta, E \subseteq \bigcup_{i=1}^{\infty} P_i \right\},
\end{equation*}
where the infimum is taken over all coverings $\{P_i\}_i$ of $E$, and where $\diam(P_i)$ denotes the diameter of $P_i$ in $(X,d)$. Set $\mathcal{H}_s(E):=\lim_{\delta \rightarrow 0} \mathcal{H}_{s,\delta}(E).$

\begin{defi} \label{defi_Hdim}
Given $E \subseteq X$, the \emph{Hausdorff dimension} $\Hdim_{(X,d)}E$ of $E$ is the real number
\begin{equation}
\inf \left\{s > 0 | \lim_{\sigma\rightarrow 0} \mathcal{H}_{s,\delta}(E) < +\infty \right\} = \sup \left\{ s> 0 | \lim_{\sigma\rightarrow 0} \mathcal{H}_{s,\delta}(E) = +\infty\right\}.
\end{equation}
When there is no risk of confusion, we will denote it by $\Hdim_X E$ or $\Hdim E$.
\end{defi}

\begin{lem} \label{hdim_disks}
Let $(X,d)$ be a metric space and $E \subseteq X$. For $s, \delta>0$, let 
$$\mathcal{D}^X_{s, \delta}(E):=\inf \left\{\sum_{i=1}^{\infty} (\mathrm{diam} \ B_i )^s| \mathrm{diam}(B_i) < \delta, E\subseteq \bigcup_{i=1}^{\infty} B_i\right\},$$
where the infimum is taken over all coverings $\{B_i\}_i$ of $E$ consisting of \emph{disks centered at points of $E$}. Then, 
$\Hdim_{X} E=\inf \left\{s>0| \lim_{\delta \rightarrow 0} \mathcal{D}_{s, \delta}^X(E) < +\infty \right\}.$
\end{lem}
\begin{proof} To simplify the notation, we will simply write $\mathcal{D}_{s, \delta}$ and forget the superscript ``$X$''. 
We note that $\mathcal{H}_{s, \delta}(E) \leqslant \mathcal{D}_{s, \delta}(E)$. Let $\{P_i\}_{i=1}^{\infty}$ be a covering of $E$ with $\diam (P_i) < \delta/2$. We may assume that $P_i \cap E \neq \emptyset$ for all~$i.$ Following \cite[\S~2.4]{falconer}, let $B_i$ be a disk centered at a point of~$P_i \cap E$, and of radius $\diam(P_i)$, thus containing $P_i$. We note that  ${\diam(B_i) \leqslant 2\diam(P_i) < \delta}$.  Clearly, $\{B_i\}_i$ is a covering of $E$, and $\sum_{i=1}^{+\infty}(\diam B_i)^s \leqslant 2^s  \sum_{i=1}^{+\infty}(\diam P_i)^s$. Taking the appropriate infima, we obtain that $\mathcal{D}_{s, \delta}(E) \leqslant 2^s \mathcal{H}_{s,\delta/2}(E)$. Taking the limit as $\delta \rightarrow 0$, 
$$\mathcal{H}_{s}(E) \leqslant \mathcal{D}_s(E):=\lim_{\delta \rightarrow 0} \mathcal{D}_{s, \delta}(E) \leqslant 2^s\mathcal{H}_{s}(E).$$ 
Hence, $\mathcal{D}_s(E)$ is finite if and only if $\mathcal{H}_s(E)$ is finite, and we may conclude.  
\end{proof}

\begin{cor} \label{hdim_ambient}
Let $(X,d)$ be a metric space and $Y \subseteq X$. Then for any $E \subseteq Y$, we have that $\Hdim_{(Y,d_{|Y})} E =\Hdim_{(X,d)} E$.
\end{cor}

\begin{proof}
 Let $\delta>0$. Let $\{B_i\}_{i=1}^{+\infty}$ be a covering of $E$ consisting of disks in $Y$ centered at points of~$E$ and such that $\diam(B_i)<\delta$. For any $i$, let $B_i'$ denote a disk in $X$ centered at a point of~$B_i \cap E$, and of radius $\diam(B_i),$ thus containing $B_i$. We note that $\{B_i'\}_i$ is a covering of $E$, and $\diam(B_i') \leqslant 2\diam(B_i)< 2\delta$. For $s>0,$ we also have $\sum_{i=1}^{+\infty} (\diam B_i')^s \leqslant 2^s \sum_{i=1}^{+\infty} (\diam B_i)^s$. Taking the appropriate infima, we obtain that $\mathcal{D}_{s, 2\delta}^X(E) \leqslant 2^s \mathcal{D}_{s, \delta}^Y(E).$

 Similarly, let $\{D_i\}_{i=1}^{+\infty}$ be a covering of $E$ consisting of disks in $X$ centered at points of~$E$ and such that $\diam(D_i)<\delta$. Then $\{D_i \cap Y\}_i$ is a covering of $E$ consisting of disks in $Y$ centered at points of~$E$ and such that $\diam(D_i \cap Y) \leqslant \diam(D_i) < \delta$. Consequently,  for $s>0$, we have $\mathcal{D}_{s, \delta}^Y(E) \leqslant \mathcal{D}_{s, \delta}^X(E)$. Thus, $\mathcal{D}_{s, 2\delta}^X(E) \leqslant 2^s \mathcal{D}_{s, \delta}^Y(E) \leqslant 2^s \mathcal{D}_{s, \delta}^X(E)$, meaning $\lim_{\delta \rightarrow 0} \mathcal{D}_{s, \delta}^X(E)< +\infty$ if and only if $\lim_{\delta \rightarrow 0} \mathcal{D}_{s, \delta}^Y(E) <+\infty$ for any $s>0$. We now conclude from Lemma \ref{hdim_disks} that $\Hdim_{(X,d)} E=\Hdim_{(Y, d_{|Y})} E$.
\end{proof}

From now on, let $(k, |\cdot|)$ be a complete valued field.

\begin{rem} \label{rem_scaling} If $\varepsilon>0$ is such 
that $|\cdot|^{\varepsilon}$ is a norm on $k$, then $
\Hdim_{(k, |\cdot|^{\varepsilon})} E=\frac{1}{\varepsilon} \Hdim_{(k, |\cdot|)} E$. 
\end{rem}

\begin{rem} \label{rem_sphericalmetric} We are interested in the Hausdorff dimension of certain subsets of $\bP^1(k)$, and hence need a metric on the latter.
When $k$ is Archimedean, we take the usual spherical metric on~$\bP^{1}(k)$. When $k$ is non-Archimedean, we endow $\bP^1(k)$ with the \emph{spherical metric}, \emph{i.e.} the metric $d_{|\cdot|}$ for which if $x,y \in k$ are in the unit disk $D$, then $d_{|\cdot|}(x,y)=|x-y|$, if $x,y \not \in D$, then $d(x,y)=|1/x-1/y|$, and if $x \in D, y\not \in D$, then $d_{|\cdot|}(x,y)=1$ (see \emph{e.g.} \cite{baker_rumely}). 
\end{rem}

Before a further study of the Hausdorff dimension of limit sets of Schottky groups, let us make a remark on its behaviour with respect to base change (\emph{cf.} also Lemma \ref{lem_schottky_field_extension}). 
\begin{lem} \label{cor_hausdorff_dim_extension}
Let $K/k$ be a complete valued field extension, and $\pi_{K/k}: \bP^{\pan}_K \rightarrow \bP^{\pan}_k$ the projection. Let $\Ga \subseteq \PGL_2(k) \subseteq \PGL_2(K)$ be a Schottky group. Let $\La_k$, respectively $\La_K$, denote the limit set of $\Ga$ when acting on~$\bP^{\pan}_k$, respectively $\bP^{\pan}_K$. Let us endow $\bP^{1}(k)$ and $\bP^1(K)$ with metrics such that the morphism $\pi_{K/k}$ restricted to $\pi_{K/k}^{-1}(\bP^{1}(k)) \subseteq \bP^{1}(K)$ is an isometry onto~$\bP^1(k)$ (\emph{e.g.} the {respective} spherical metrics). Then $\Hdim \La_K =\Hdim \La_k$.
\end{lem}

\begin{proof} By Corollary \ref{hdim_ambient}, as $\La_K \subseteq \pi_{K/k}^{-1}(\bP^1(k))$, we have that $\Hdim_{\pi_{K/k}^{-1}(\bP^1(k))} \La_K=\Hdim_{\bP^1(K)} \La_K$. We may conclude seeing as isometries preserve the Hausdorff dimension and ${\pi_{K/k}(\La_K)=\La_k}$. 
\end{proof}
There is a relation between the critical exponents of Poincaré series and the Hausdorff dimension of limit sets of groups. It dates back to Patterson's \cite{patterson} for Fuchsian groups, and Sullivan's \cite{sullivan_discrete_conformal} for convex-cocompact groups of isometries of real hyperbolic spaces. 
This method was later extended to more general hyperbolic spaces by Das-Simmons-Urbanski~\cite{das_simmons_urbanski}.   
 
When $k$ is Archimedean, we shall denote by $\rho$  the  hyperbolic distance in the real hyperbolic space $\Hy_k:=\Hy^3$ if $k=\C$, respectively $\Hy_k:=\Hy^2$ if $k=\R$. When $k$ is non-Archimedean,  we take~$\Hy_k$ to be the Berkovich hyperbolic space as defined in Section \ref{section_hyperbolic}, and by $\rho$ the interval length distance recalled  in \emph{loc.cit.}. Let $\Ga \subseteq \PGL_2(k)$ be a {Schottky group}. 

\begin{defi} \label{poincareseries}
Let $x, y \in \Hy_k$.For $s \in \R_{\geqslant 0}$, the \emph{Poincaré series of $\Ga$ from $x$ to $y$} is: 
\begin{equation*}
\Po_{x,y}(s):=\sum_{\gamma \in \Ga} e^{-s\rho(x, \g y)}.
\end{equation*}
If $x=y$, we write $\mathcal{P}_x(s)$ and call it the \emph{Poincaré series of $\Ga$ at $x$}.
The \emph{critical exponent $d_{\Ga}$} of~$\Ga$ is  defined to be  $\inf \{ s > 0 | \cP_x(s) < +\infty \}$.
\end{defi}

By an application of the triangular inequality, neither the convergence of $\Po_{x,y}(s)$ nor the critical exponent $d_{\Ga}$ of $\Ga$ depend on $x$ and $y$.  

\begin{rem} \label{funddomain1} 
Assume $\mathcal{D}: =\{D_{\g}: \g \in \Ga \backslash \{\mathrm{id}\}\}$ is a Schottky figure for {$\Ga$} in $\bP^{\pan}_k$. Let $x,y \in \Hy_k$. When considering the Poincaré series $\Po_{x,y}(s)$, we may assume that $x, y \not \in \bigcup_{\g \in \Ga \backslash \{\mathrm{id}\}} D_{\g}^{\circ}$. This is because by Lemma \ref{domain}, $\bP^{\pan}_k \backslash \bigcup_{\g \in \Ga \backslash \{\mathrm{id}\}} D_{\g}^{\circ}$ contains a fundamental domain of the action of $\Ga$, so there exist $\alpha, \beta \in \Ga$ such that $\alpha x, \beta y \not \in \bigcup_{\g \in \Ga \backslash \{\mathrm{id}\}} D_{\g}^{\circ}$. We may now conclude seeing as
$$\Po_{\alpha x, \beta y}(s)=\sum_{\g \in \Ga} e^{-s\rho(\alpha x, \gamma \beta y)}=\sum_{\g \in \Ga} e^{-s\rho(x, \alpha^{-1}\g \beta y)}= \Po_{x,y}(s).$$
\end{rem}

\begin{lem} \label{stdis} The group $\Ga$ has a  \emph{strongly discrete} action on the metric space $(\Hy_k, \rho)$, meaning for any~$x\in \Hy_k$ and any $R> 0$, there exist only finitely many elements $\g \in \Ga$ such that~${\rho(x, \g x)<R}$.
\end{lem}

\begin{proof} When $k$ is Archimedean, this follows from the equivalence between the discreteness of the action and its properness on $\Hy^3$ and $\Hy^2$ (see \emph{e.g.} \cite[Theorem~5.3.2]{beardon}).   
When~$k$ is non-Archimedean, using the triangular inequality, we are reduced to prove the statement for a {single} point $x= x_0$. 
 By Theorem \ref{Gerritzen}, there exists a Schottky figure for $\Ga$ in $\bP^{\pan}_k$. Let~${x_0 \in \Hy_k}$ be such that it does not belong to any of the Schottky disks of this figure. 
By Lemma~\ref{lem_length_expansion}, there exists a constant $L>0$, depending only on $x_0$ and the Schottky figure, such that $\rho(x,\g x) \geqslant L(l(\g)+1)$, where $l(\g)$ denotes the length of the word~$\g$. Hence if $\rho(x, \g x)<R$, one has $L(l(\g)+1)<R$, implying $l(\g)< R/L -1$, and we may conclude.
\end{proof}

We will use the following extension of the Bishop-Jones Theorem \cite[Theorem 1.1]{bishop_jones}, due to Das-Simmons-Urbanski.

\begin{thm} \label{dsu} 
Let $\Ga$ be a Schottky group over $k$. Let $x \in \Hy_k$. The critical exponent of the Poincaré series $\Po_x(s)$ is equal to the Hausdorff dimension of the radial limit set of $\Ga$ (\emph{cf.} Definition~\ref{radiallimitset}).
\end{thm}

 \begin{proof} When $k$ is non-Archimedean, $(\Hy_k, \rho)$ is a complete metric space which is a real tree, so it is a uniquely geodesic CAT$(-1)$-space (in the sense that any two points are connected by a unique geodesic, and a geodesic triangle in the tree is smaller than its comparison triangle in a hyperbolic space with constant curvature $-1$;  {see \emph{e.g.} \cite[Chapter II.1, Definition 1.2]{bridson_haefliger} for a precise definition)}.
When $k$ is Archimedean, $(\Hy_k, \rho)$ is a complete metric space,  uniquely geodesic {(see \emph{e.g.} \cite[I.2, Corollary 2.8]{bridson_haefliger})} and CAT$(-1)$ by {definition}.  
In both cases, $(\Hy_k, \rho)$ is \emph{regularly geodesic} in the sense of \cite[Definition 4.4.5]{das_simmons_urbanski}.

Moreover, by Lemma \ref{stdis}, $\Ga$ acts strongly discretely on $(\Hy_k, \rho)$, so it acts \emph{moderately discretely} in the sense of \cite[Definition 5.2.1]{das_simmons_urbanski}.
By \mbox{\cite[Proposition 9.3.1]{das_simmons_urbanski}}, the group $\Ga$ is \emph{Poincaré regular} (\emph{cf.} \cite[Definition 8.2.1]{das_simmons_urbanski}), so the \emph{modified Poincaré exponent} (\cite[Definition 8.2.3]{das_simmons_urbanski}) and the critical exponent of $\Po_x(s)$ are  equal. 
We apply \cite[Theorem~1.2.3]{das_simmons_urbanski} to conclude.
\end{proof}

\begin{cor}\label{cor_critical_exponent_hausdorff} Let $\Gamma$ be a Schottky group over a complete valued field $k$. The critical exponent~$d_{\Ga}$ of $\Ga$ is equal to the Hausdorff dimension of the limit set~$\Lambda_{\Ga}$ of $\Ga$.  
\end{cor}
  
\begin{proof}When $k$ is Archimedean, this is the content of \cite[Theorem 7]{sullivan_density}.
Assume now that $k$ is non-Archimedean. By Theorem \ref{dsu}, $d_{\Ga}$ is the Hausdorff dimension of the radial limit set of~$\Ga$. The radial limit set and the limit set coincide by Proposition \ref{prop_radial}, concluding the proof.
\end{proof}

\subsection{Hausdorff dimension and spectral radius of a Perron matrix} \label{section_hausdorff_spectral}
Let $(k, |\cdot|)$ be a non-Archimedean field. 
Let $\Ga \subseteq \PGL_2(k)$ be a Schottky group of rank~$g \geqslant 1$. Recall from Theorem \ref{Gerritzen} that  there exists a basis $(\g_1, \g_2, \dots, \g_g)$ of $\Ga$ for which there is a Schottky figure in $\bP^{\pan}_k$ consisting of disks $D_{\g}, D_{\g}^{\circ}, \g \in \Ga \backslash \{\mathrm{id}\}.$  

Let $\sigma:  \{1,2,\dots, 2g\} \rightarrow \{\g_i, \g_i^{-1}, i=1,2,\dots, g\}$ be the bijection given by $2i-1 \mapsto \g_i$ and $2i \mapsto \g_i^{-1}$, for any $i.$

\begin{defi} \label{matrixM}
Let $s \in \R_{\geqslant 0}$. To the Schottky figure $\mathcal{D}:=\{D_{\g_i}, D_{\g_i^{-1}}, i=1,2,\dots, g\}$ in $\bP^{\pan}_k$, we associate a matrix $M_{\mathcal{D}}(s) \in M_{2g \times 2g}(\mathbb{R})$ whose $(j,l)$ entry is given by $$m_{j,l}(s):=\begin{cases}e^{-s\rho(\partial{D}_{\sigma(j)^{-1}}, \partial{D}_{\sigma(l)})} & \text{if} \ \sigma(j)^{-1} \neq \sigma(l), \\ 0 & \text{otherwise}.
\end{cases}$$

We denote $M_{\mathcal{D}}:=M_{\mathcal{D}}(1)$ and $m_{j,l}:=m_{j,l}(1).$ When there is no risk of ambiguity, we will forget the subscript $\mathcal{D}$ and simply write $M(s)$ and $M.$
\end{defi}

\begin{rem}  \label{propertiesM} \begin{enumerate}
\item We note that even though $M(s)$ is not symmetric, 
$$m_{j,l}(s)=m_{l+1, j+1}(s) \ \text{if} \ j,l \ \text{are odd},$$
$$m_{j,l}(s)=m_{l-1, j+1}(s) \ \text{if} \ j \ \text{odd and} \ l \ \text{even},$$
$$m_{j,l}(s)=m_{l+1, j-1}(s) \ \text{if} \ j \ \text{even and} \ l \ \text{odd}.$$

\item The entries of $M(s)$ are in the interval $[0,1)$. Essentially, the matrix $M$ contains the data $\rho(\partial{D}_{\g}, \partial{D}_{\g'})$ for $\g,\g' \in \Ga$ of length $1$. 

\item We note that $m_{j,l}(s)=0$ if and only if $(j,l) \in \{(2i, 2i-1), (2i-1, 2i) | i=1,2,\dots, g\}$.  There is exactly one zero in each row and each column of the matrix $M(s).$

\item For any $i=1,2,\dots, g$, we have $m_{2i-1, 2i-1}(s)=m_{2i, 2i}(s).$
\end{enumerate}
\end{rem}
The following will be needed later on:

\begin{lem} \label{lem_characteristic_zero} The matrix $M(0)$ is symmetric and its characteristic polynomial is:
\begin{equation*}
\chi_{M(0)}(X) :=(X- (2g - 1)) (X+1)^{g-1}  (X-1)^{g}.
\end{equation*}
A system of eigenvectors is given by $U: = \sum_{j=1}^{2g} e_j$, $u_i:= e_{1} + e_{2} - e_{2i-1} - e_{2i}$ for $i=2,3,\dots, g$, and $v_l :=  e_{2l-1} - e_{2l}$ for $l=1,2,\dots, g,$ where $(e_1, e_2 \dots, e_{2g})$ denotes the standard basis in $\R^{2g}.$
\end{lem}

\begin{proof}  
The given vectors are eigenvectors by a direct computation using the fact that~$m_{j,l}(0)=0$ if and only if $(j,l) \in \{(2i, 2i-1), (2i-1, 2i) | i=1,2,\dots, g\}$, and $m_{j,l}(0)=1$ otherwise. We then conclude using the spectral theorem. 
\end{proof}

\begin{ex}
If $g=2$, then
$$M=\begin{pmatrix}
e^{-\rho(\partial{D_{\g_1^{-1}}, \partial{D}_{\g_1}})} & 0 & e^{-\rho(\partial{D_{\g_1^{-1}}, \partial{D}_{\g_2}})} & e^{-\rho(\partial{D_{\g_1^{-1}}, \partial{D}_{\g_2^{-1}}})} \\
 0 & e^{-\rho(\partial{D_{\g_1}, \partial{D}_{\g_1}^{-1}})} & e^{-\rho(\partial{D_{\g_1}, \partial{D}_{\g_2}})} & e^{-\rho(\partial{D_{\g_1}, \partial{D}_{\g_2^{-1}}})}\\
 
e^{-\rho(\partial{D_{\g_2^{-1}}, \partial{D}_{\g_1}})} &  e^{-\rho(\partial{D_{\g_2^{-1}}, \partial{D}_{\g_1^{-1}}})} & e^{-\rho(\partial{D_{\g_2^{-1}}, \partial{D}_{\g_2}})} & 0 \\

e^{-\rho(\partial{D_{\g_2}, \partial{D}_{\g_1}})} &  e^{-\rho(\partial{D_{\g_2}, \partial{D}_{\g_1^{-1}}})} & 0 & e^{-\rho(\partial{D_{\g_2}, \partial{D}_{\g_2^{-1}}})} \\
\end{pmatrix}.$$
\end{ex}

The following is the reason these matrices play a central role in the study of the limit set and its Hausdorff dimension. Let $x, y \in \Hy_k \cap F$, where $F$ is the fundamental domain constructed in Lemma \ref{domain}. For $s \in \R_{\geqslant 0}$, let 
\begin{equation}
 X_x(s):=(e^{-s\rho(x, \partial{D}_{\g_1})}, e^{-s\rho(x,  \partial{D}_{\g_1^{-1}})}, \dots, e^{-s\rho(x,  \partial{D}_{\g_g^{-1}})}),
\end{equation}
 and
 \begin{equation}
  Y_y(s):=(e^{-s\rho(y, \partial{D}_{\g_1^{-1}})}, e^{-s\rho(y, \partial{D}_{\g_1)}}, \dots, e^{-s\rho(y,  \partial{D}_{\g_g})}).   
 \end{equation}

\begin{thm} \label{thm_meromorphic_extension} As formal power series in $s$,
 $$
\Po_{x,y}(s)=\sum_{n=0}^{+\infty} X_x(s)M(s)^n Y_y(s)^{t}+ e^{-s\rho(x,y)} =X_x(s)(I - M(s))^{-1} Y_y(s)^{t}+ e^{-s\rho(x,y)}.$$
Since for $s \in \R_{\geqslant 0}$ large enough the Poincaré series converges, $\Po_{x,y}(s)$ admits a meromorphic continuation over $\C$ and is equal to a ratio of multivariate exponential polynomials.  
\end{thm}

\begin{proof}
By Lemma \ref{lem_length_expansion}, for any $n \geqslant 1$,  
$\sum_{l(\g)=n} e^{-s\rho(x, \g y)}= X_x(s) M(s)^{n-1} Y_y(s)^{t}
$ and the first equality follows. The second one is due to formal 
computations.

Finally, observe that the coefficients of the matrix $M(s)$ converge to zero as $s$ tends to infinity.  
So for $s$ large enough, the spectral radius of $M(s)$ is strictly smaller than $1$ and the final part follows. We conclude by observing that this last formula admits a meromorphic continuation over $\C$:
\begin{equation} \label{meropo}
    \Po_{x,y}(s):= \dfrac{1}{\det(I - M(s))} X_x(s) C(s)^{t} Y_y(s)^{t}+ e^{-s\rho(x,y)},  
\end{equation}
where $C(s)$ denotes the cofactor matrix of $I - M(s).$
\end{proof}

\begin{rem} \label{rem_meroext}
Using \eqref{meropo}, let the meromorphic extension be written as $$\Po_{x,y}(s)=\frac{P(e^{sl_1}, e^{sl_2}, \dots e^{sl_{4g^2}})}{Q(e^{sl_1}, e^{sl_2}, \dots e^{sl_{4g^2}})} + e^{sl},$$ where $l, l_i \in \R_{<0}$. Then one has that $P,Q \in \Z[X_1, X_2, \dots, X_{4g^2}]$. Moreover, the coefficients of the polynomials $P$ and $Q$ are all equal to $1$ or $-1$. One can also compute that in general $\deg{P} \leqslant 2g+1$ and $\deg{Q} \leqslant 2g.$ 
\end{rem}

\begin{rem} \label{rem_meroext2}
By Remark \ref{funddomain1}, given \emph{any} $x, y \in \Hy_k$, the Poincaré series $\Po_{x,y}(s)$ has a meromorphic extension as in Remark \ref{rem_meroext}.\end{rem}

Note that if $g>1$, then
for any $s \geqslant 0$, the matrix $M(s)^2$ has  positive entries, hence $M(s)$ is primitive. 

\begin{prop} \label{PoM2} Assume $g>1$. Let $x, y\in \Hy_k$.
The Poincaré series $\Po_{x,y}(s)$ converges if and only if the matrix series $\sum_{n=0}^{+\infty} M(s)^n$ converges. Consequently, $\Po_{x,y}(s)$ converges if and only if $\la(M(s))<1$, where $\la(M(s))$ denotes the spectral radius of $M(s).$
\end{prop}

\begin{proof} By Remark \ref{funddomain1}, we may assume $x, y  \in \Hy_k \cap F$ with $F$ as in Lemma \ref{domain}.
If $\la(M(s)) < 1$, then the  series $\sum_{n=0}^{+\infty} M^n(s)$
is convergent, so $\Po_{x,y}(s)$ is convergent.
Let us now assume that $\Po_{x,y}(s)$ converges, so $X_x(s)M(s)^nY_y(s)^{t}$ tends to $0$ as ${n \rightarrow +\infty} $. By the Perron-Frobenius Theorem, its spectral radius $\la(M(s))$ is strictly positive, has multiplicity $1$ and all its other eigenvalues $\lambda_i(s)$, satisfy $|\lambda_i(s)| < \la(M(s))$ for $i=1,2,\dots, 2g-1$.  

Let $Y_y(s)^{t}=\alpha(s) u(s) + \sum_{i=1}^{g-1} \alpha_i(s) u_i(s)$ be the decomposition in the {direct} sum $\R^{2g}=E_{\la(s)} \oplus \bigoplus_{i=1}^{2g-1} E_i$, where $E_{\la(s)}, E_i$ are the eigenspaces, and $u(s), u_i(s)$ are eigenvectors of $\la(M(s)), \la_i(s)$, respectively.  
Then 

\smallskip

{\centering
  $ \displaystyle
    \begin{aligned}
 X_x(s)M(s)^n Y_y(s)^{t}= \alpha(s) \la(M(s))^n X_x(s)u(s) + \sum_{i=1}^{2g-1} \alpha_i(s) \la_i(s)^n X_x(s)u_i(s),
\end{aligned}
  $ 
\par}

\smallskip

\noindent so $ X_x(s)M(s)^n Y_y(s)^{t}=C(s) \la(M(s))^n + o(\la(M(s))^n)$ as $n \rightarrow +\infty$, where $C(s):=\alpha(s) X_x(s)u(s).$ Thus, $\la(M(s))^n \rightarrow 0$ when $n \rightarrow + \infty$, implying $\la(M(s))<1$, and hence that $\sum_{n=0}^{+\infty} M(s)^n$ is convergent.
\end{proof}

\begin{cor} \label{cor_spectral_formula} Assume $g>1$.
Let $\Hdim  \La_{\Ga}$ denote the Hausdorff dimension of the limit set~$\La_{\Ga}$. Then $s=\Hdim  \La_{\Ga}$ if and only if $\la(M(s))=1$, where $\la(M(s))$ denotes the spectral radius of $M(s).$ 
\end{cor}

\begin{proof} By Corollary \ref{cor_critical_exponent_hausdorff}, $\mathrm{Hdim} \La_{\Ga}$ is the critical exponent $d_{\Ga}$ of the Poincaré series $\Po_x(s)$ with~${x \in \Hy_k}$.
By Proposition \ref{PoM2}, $s>d_{\Ga}$ implies $\la(M(s))<1$ and $s< d_{\Ga}$ implies ${\la(M(s)) \geqslant 1}$. As $M(s)$ is primitive, $\la(M(s))$ is an eigenvalue of multiplicity~$1$, so by Perron-Frobenius, it varies continuously with the coefficients of $M(s)$, and hence with $s$ itself. Thus, $s=d_{\Ga}$ implies $\la(M(s))=1$.

Observe also that the coefficients of $M(s)$ are decreasing in $s$ as they are either zero or of the form $e^{- u s}$ with $u>0$. This implies $\lambda(M(s))$ is a decreasing function in $s$, and shows that~${s=d_{\Ga}}$ if {and only if} ${\la(M(s))=1}$. 
\end{proof}

\begin{rem}
Corollary \ref{cor_spectral_formula} above provides a systematic approach for the computation of the Hausdorff dimension of the limit set $\La_{\Ga}$ of a Schottky group $\Ga$ over a non-Archimedean field $k$.  We will use it in Section \ref{sect_examples}.
\end{rem}

\subsection{Bounds on the Hausdorff dimension}
Let $k$ be a  non-Archimedean field.
Let $\Ga \subseteq \PGL_2(k)$ be a Schottky group. Let  $(\g_1, \g_2, \dots, \g_g)$ be a basis of $\Ga$, and $D_{\g}, D_{\g}^{\circ}, \g \in \Ga \backslash \{\mathrm{id}\}$, an adapted Schottky figure in $\bP^{\pan}_k$. Set $\{x_{\g}\}:=\partial{D}_{\g}$ for any such $\g$. Recall that $l(\g)$ denotes the length of $\g$ in the alphabet $(\g_i^{\pm 1})_{i=1}^g$.

\begin{prop} \label{Hdimbounds}
Let $C_1:=\min\{\rho(x_a,x_b) : a, b \in \Ga, a \neq b,  l(a)=l(b)=1\}$ and $C_2:=\max\{\rho(x_a,x_b) : a, b \in \Ga, l(a)=l(b)=1\}.$ The Hausdorff dimension $\Hdim \La_{\Ga}$ of the limit set~$\La_{\Ga}$ of $\Ga$ satisfies:
$$\frac{\log(2g-1)}{C_2} \leqslant \Hdim \Lambda_{\Ga} \leqslant \frac{\log(2g-1)}{C_1}.$$
\end{prop}

\begin{proof} If $g=1$, then the limit set $\La_{\Ga}$ consists of two points, and all the expresions in the inequalities of the statement equal $0$.  Assume now that $g>1$. We note that
\begin{equation*}
    e^{-s C_2 }M(0) \leqslant M(s) \leqslant e^{-s C_1} M(0),
\end{equation*}
where $A\leqslant B$ for two matrices $A, B$ means that the entries of $B-A$ are all non-negative. As the spectral radius is an increasing function, and the spectral radius of $M(0)$ is $2g-1$ (\emph{cf.} Lemma \ref{lem_characteristic_zero}), we obtain that  $   e^{-s C_2} (2g -1) \leqslant \lambda(M(s)) \leqslant e^{-s C_1} (2g-1)$.
Taking $s= d_{\Ga}$ one has $\lambda(M(s))=1$ by Corollary \ref{cor_spectral_formula}, which yields the result.
\end{proof}

\begin{rem} \label{rem_hdim_loccomp} Let $(k, |\cdot|)$ be a locally compact field. Let us denote by $\widetilde{k}$ its residue field, and~$q$ the cardinality of $\widetilde{k}$.
Finally, let $\pi \in k$ be a uniformizer. Then, the Hausdorff dimension of~$\bP^{1}(k)$ is $\dfrac{ \log q}{\log |\pi|^{-1}}$. To see this, note that the unit disk in $k$ can be covered by $q$ disks of radii~$|\pi|$. As a consequence, given a Schottky group $\Ga$ over $k$, the Hausdorff dimension $d_{\Ga}$ of its limit set is always bounded by $\dfrac{ \log q}{\log |\pi|^{-1}}.$ 
\end{rem}

\begin{ex}
Assume that the Schottky group $\Ga$ is such that for any words $a, b$ of length one, $\rho(x_a, x_b)$ is a constant $C$. Then by Proposition \ref{Hdimbounds}, the Hausdorff dimension of $\La_{\Ga}$ is~${d_{\Ga}=\frac{\log(2g-1)}{C}}.$
\end{ex}

\begin{ex}
Let $k=\C_p.$ Assume $p\geqslant 5$. Let $\g_1, \g_2 \in \PGL_2(\C_p) $ be the transformations determined by the Koebe coordinates $(1,2,p^{\alpha})$, respectively $(3,4,2p^{\alpha})$, where $\alpha \in \Q_{>0}$. 
Let us define the closed disks in $\bP^{\pan}_{\C_p}$:
$$D_{\g_1}:=D\left(\frac{1-2p^{\alpha}}{1-p^{\alpha}}, |p|^{\alpha/2}\right)=D\left(1, |p|^{\alpha/2}\right), \  \  D_{\g_1^{-1}}:=D\left(\frac{2-p^{\alpha}}{1-p^{\alpha}}, |p|^{\alpha/2}\right)=D(2, |p|^{\alpha/2}),$$

$$D_{\g_2}:=D\left(\frac{3-8p^{\alpha}}{1-2p^{\alpha}}, |p|^{\alpha/2}\right)=D(3, |p|^{\alpha/2}), \   \  D_{\g_2^{-1}}:=D\left(\frac{4-6p^{\alpha}}{1-2p^{\alpha}}, |p|^{\alpha/2}\right)=D(4, |p|^{\alpha/2}).$$

Also, let $D_{\g}^{\circ}$ denote the open maximal disk contained in $D_{\g}$ with same centre and radius as above for $\g \in \{\g_1^{\pm 1}, \g_2^{\pm 1}\}.$ By \cite[Lemma II.3.30]{poineau_turcheti_2}, $\g_i(D_{\g_i^{-1}}^c)=D_{\g_i}^{\circ}$ and $\g_i^{-1}(D_{\g_i}^c)=D_{\g_i^{-1}}^{\circ}$ for $i=1,2.$ 

As these four disks are mutually disjoint, they form a Schottky figure for $(\g_1, \g_2).$ Moreover, $\rho(x_a, x_b)=\log \frac{1}{|p|^{\alpha}}=\log p^{\alpha}$ for any $a,b \in \{\g_1^{\pm 1}, \g_2^{\pm 1}\}$ such that $a \neq b$. Hence, the Hausdorff dimension of the limit set of the Schottky group generated by $\g_1, \g_2$ is $\frac{\log 3}{\alpha \log p}$. 

We remark that as $\alpha$ tends to $0$, the Hausdorff dimension tends to $+\infty$, and as $\alpha$ tends to~$+\infty$, the Hausdorff dimension tends to $0$. 
Let us also note that if $\eta \in \bP^{\pan}_{\C_p}$ denotes the Gauss point, then  
{%
    $\mathcal{P}_{\eta}(s) =1 + \dfrac{4|p|^{s \alpha}}{1 - 3 |p|^{s \alpha}}=\dfrac{p^{s \alpha} +1}{p^{s\alpha}-3}.$}
\end{ex}

\begin{rem}
One can modify the example above to any field $k$ that is not discretely valued. The key point is that then the value group of $k$ is dense in~$\R.$
\end{rem}

\section{Non-Archimedean Poincaré series} \label{sect_4}

In Theorem \ref{thm_meromorphic_extension}, we proved that given a Schottky group $\Ga$ over a non-Archimedean field $k$, its associated Poincaré series could be given through a matrix series: for $x, y \in \Hy_k$, when there is convergence, 

{\centering
  $ \displaystyle
    \begin{aligned}
\Po_{x,y}(s)-e^{-s\rho(x,y)}=\sum_{n=0}^{+\infty} X_x(s) M(s)^n Y_y(s)^{t}=X_x(s)(I-M(s))^{-1}Y_y(s)^{t},
\end{aligned}
  $ 
\par}

\smallskip

\noindent for some well-chosen real vectors $X_x(s), Y_y(s)$, and a matrix $M(s)$ which does not depend on $x$ and $y$. 

In this section we study the meromorphic extension of the Poincaré series through the matrix~$(I-M(s))^{-1}$ and its eigenvalues, and show that $\Po_{x,y}(s)$ is analytic and admits a special value at~$0$. This value is given in terms of the Euler characteristic of the Mumford curve obtained by taking the quotient with respect to the group action of $\Ga$ on $\bP^{\pan}_k$. See also  \cite{fried_analytic_torsion} for the Ruelle zeta function, and \cite{dang_riviere_poincare,benard_chaubet_dang}, where results of similar nature are obtained for Poincaré series in other contexts.  

\subsection{A combinatorial lemma} \label{sect_comb_lemma}

Let $k$ be a non-Archimedean field, and $\Ga \subseteq \PGL_2(k)$ a Schottky group of rank $g>1$. Let~$\mathcal{D}$ be an associated Schottky figure in $\bP^{\pan}_k$. Let $M_{\mathcal{D}}(s)$ be the $2g \times 2g$ matrix constructed in Definition \ref{matrixM}. For simplicity of notation, we will simply denote it by $M(s)$ here.  
In this section, we study the behavior of the eigenvalues of the matrix $M(s)$ near zero. 

By Lemma \ref{lem_characteristic_zero}, the characteristic polynomial of the symmetric matrix $M(0)$ is given by~$(X-(2g-1))(X+1)^{g-1} (X-1)^g$. As we are interested in continuations of $\Po_{x,y}(s)$, we will mostly focus on the eigenvalue $1$ of $M(0)$, since it creates an apparent pole for $(I - M(s))^{-1}$. Our goal is to control the order of this pole and show that multiplication with the vectors $X_x(s),Y_y(s)$ compensates for it. 
By \emph{loc.cit.}, there exists an orthogonal matrix $P \in M_{2g \times 2g}(\R)$ such that:   
\begin{equation} \label{M(0)}
P^t M(0) P=\mathrm{Diag}(2g-1, -1 , \ldots, -1, +1 , \ldots, +1),
\end{equation}
where the diagonal matrix on the right contains $g-1$ entries equal to $-1$ and $g$ entries equal to $1$. 

We consider a change of basis induced by $P,$ and write the result as a $2\times 2$ block matrix 
\begin{equation} \label{eq_tildeM}
\tilde{M}(s) := P^t M(s) P = \left ( \begin{array}{ll}
A(s) & B(s)\\
C(s) & D(s)
\end{array} \right ),
\end{equation}
where $A(s),B(s),C(s),D(s)$ are of size $g\times g$ with real analytic coefficients in $s$. By~\eqref{M(0)}, $B(0)=C(0)=I-D(0)$ are all the zero matrix. 
Moreover, we can assume that $A(s)$ is invertible if we restrict to a neighborhood of $s=0$. 

Since we are interested only in the eigenvalue $1$, we now restrict our attention to the matrix~$D(s)$. The key point here is to show that $\mathrm{ord}_{s=0}(I-D(s))=1$, \emph{i.e.} that $D'(0)$ is invertible. 

We denote by $(\g_1, \dots, \g_g)$ the basis of $\Ga$ to which we have associated the Schottky figure $\mathcal{D}$, and by $D_{\alpha}$, $\alpha \in \{\g_i, \g_i^{-1}: i=1,2, \dots, g\}$, the first generation Schottky disks of $\mathcal{D}$.

\begin{lem} \label{prop_D_symmetric} The matrix $D(s)$ is symmetric for $s \in \C$. 
Moreover, the $(i,j)$ entry of $D(s)$ is given by:
\begin{equation*}
\delta_{ij}(s)=
\frac{1}{2}(-e^{-s\rho(\partial{D}_{\g_i}, \partial{D}_{\g_j})} - e^{-s \rho(\partial{D}_{\g_i^{-1}}, \partial{D}_{\g_j^{-1}})} + e^{-s\rho(\partial{D}_{\g_i}, \partial{D}_{\g_j^{-1}})} + e^{-s\rho(\partial{D}_{\g_i^{-1}}, \partial{D}_{\g_j})}),
\end{equation*}
where $\rho$ is the interval-length metric on $\Hy_k$ from Section~\ref{section_hyperbolic}.
 \end{lem}

\begin{proof}
As in  Lemma \ref{lem_characteristic_zero}, let $v_l: = (e_{2l-1} - e_{2l})/\sqrt{2}$ for all $l=1,2,\dots, g$. Note that they form an orthonormal basis of eigenvectors of $M(0)$ for the eigenvalue $1$.  The matrix $D(s) = (\delta_{ij}(s))_{ij}$ is then determined by $$\delta_{ij}(s)=\langle v_i, M(s) v_j \rangle= \frac{1}{2}(m_{2i,2j}(s)+m_{2i-1, 2j-1}(s) -m_{2i-1, 2j}(s)-m_{2i, 2j-1}(s)),$$ 
which yields the expression in the statement. 
By Remark \ref{propertiesM} 1., we obtain the equality~${\delta_{ij}(s)= \delta_{ji}(s)}$ for all $i,j$. 
\end{proof}
We remark that the $(i,j)$ entry of $-D'(0)$ is: $$-\delta_{ij}'(0)=\frac{1}{2}(-\rho(\partial{D}_{\g_i}, \partial{D}_{\g_j})-\rho(\partial{D}_{\g_i^{-1}}, \partial{D}_{\g_j^{-1}})+\rho(\partial{D}_{\g_i}, \partial{D}_{\g_j^{-1}})+\rho(\partial{D}_{\g_i^{-1}}, \partial{D}_{\g_j})).$$

\begin{cor} \label{birap}
For $i,j \in \{1,2,\dots, g\}$, set $L_{ij}:=[\partial{D_{\g_i}}, \partial{D_{\g_i}^{-1}}] \cap [\partial{D_{\g_j}}, \partial{D_{\g_j}^{-1}}] \subseteq \Hy_k$.
If~$L_{ij} \neq \emptyset$ and going from $\partial{D}_{\g_i}$ to $\partial{D}_{\g_i^{-1}}$ induces the same orientation on $L_{ij}$ as going from $\partial{D}_{\g_j}$ to $\partial{D}_{\g_{j}^{-1}}$, set $\varepsilon_{ij}:=1$. Otherwise, set $\varepsilon_{ij}=-1$. Then $-\delta_{ij}'(0)=\varepsilon_{ij}\rho(L_{ij}).$
\end{cor}

\begin{proof}
This is a direct calculation. See the figures below for a few of  examples of such a situation.
\end{proof}

\begin{center}
\begin{tikzpicture}[line cap=round,line join=round,>=triangle 45,x=1cm,y=1cm, scale=0.9]
\draw [line width=0.8pt] (-8.42,1.58)-- (-7.68,0.9);
\draw [line width=0.8pt] (-7.68,0.9)-- (-6.64,2.06);
\draw [line width=0.8pt] (-7.68,0.9)-- (-7.7,-0.54);
\draw [line width=0.8pt] (-7.7,-0.54)-- (-8.62,-0.68);
\draw [line width=0.8pt] (-7.7,-0.54)-- (-7.18,-0.72);
\draw [line width=0.8pt] (-4.06,1.56)-- (-3.06,0.98);
\draw [line width=0.8pt] (-3.06,0.98)-- (-2.62,1.54);
\draw [line width=0.8pt] (-3.06,0.98)-- (-3.12,0);
\draw [line width=0.8pt] (-3.12,0)-- (-2.72,-0.34);
\draw [line width=0.8pt] (-3.12,0)-- (-4.48,-0.28);
\draw [line width=0.8pt] (0.36,1.88)-- (1.2,1);
\draw [line width=0.8pt] (1.2,1)-- (2.06,1.78);
\draw [line width=0.8pt] (1.2,1)-- (1.2,-0.32);
\draw [line width=0.8pt] (1.2,-0.32)-- (0.56,-0.62);
\draw [line width=0.8pt] (1.2,-0.32)-- (1.92,-0.64);
\draw (-8.56,-1.18) node[anchor=north west] {$\varepsilon_{ij}=1$};
\draw (-4.02,-1.2) node[anchor=north west] {$\varepsilon_{ij}=-1$};
\draw (-0.1,-1.18) node[anchor=north west] {$L_{ij}=\emptyset, \varepsilon_{ij}=-1$};
\begin{scriptsize}
\draw [fill=black] (-8.42,1.58) circle (1pt);
\draw[color=black] (-8.4,1.81) node {$\partial{D}_{\g_i^{-1}} $};
\draw [fill=black] (-6.64,2.06) circle (1pt);
\draw[color=black] (-6.02,2.03) node {$\partial{D}_{\g_j}$};
\draw [fill=black] (-8.62,-0.68) circle (1pt);
\draw[color=black] (-9,-0.75) node {$\partial{D}_{\g_i}$};
\draw [fill=black] (-7.18,-0.72) circle (1pt);
\draw[color=black] (-6.5,-0.81) node {$\partial{D}_{\g_j^{-1}}$};
\draw [fill=black] (-4.06,1.56) circle (1pt);
\draw[color=black] (-4.26,1.83) node {$\partial{D}_{\g_i}$};
\draw [fill=black] (-2.62,1.54) circle (1pt);
\draw[color=black] (-2.1,1.80) node {$\partial{D}_{\g_j}$};
\draw [fill=black] (-2.72,-0.34) circle (1pt);
\draw[color=black] (-2.02,-0.61) node {$\partial{D}_{\g_j^{-1}}$};
\draw[color=black] (-7.36,0.39) node {$L_{ij}$};
\draw[color=black] (-2.7,0.67) node {$L_{ij}$};
\draw [fill=black] (-4.48,-0.28) circle (1pt);
\draw[color=black] (-4.44,-0.59) node {$\partial{D}_{\g_i^{-1}}$};
\draw [fill=black] (0.36,1.88) circle (1pt);
\draw[color=black] (0.26,2.21) node {$\partial{D}_{\g_i}$};
\draw [fill=black] (2.06,1.78) circle (1pt);
\draw[color=black] (2.7,1.85) node {$\partial{D}_{\g_i^{-1}}$};
\draw [fill=black] (0.56,-0.62) circle (1pt);
\draw[color=black] (0.46,-0.87) node {$\partial{D}_{\g_j}$};
\draw [fill=black] (1.92,-0.64) circle (1pt);
\draw[color=black] (2.64,-0.83) node {$\partial{D}_{\g_j^{-1}}$};
\end{scriptsize}
\end{tikzpicture}
\end{center}

\begin{rem}
By Lemma 2.5.2 of \cite{poineau_turcheti_universal}, we obtain that $-\delta_{ij}'(0)=|[a_i, b_i; a_j, b_j]|$ whenever $i \neq j$. Here $a_i, b_i \in \bP^1(k)$ denote the attractive, respectively repelling, points of $\g_i$, and $[a_i, b_i; a_j, b_j]$ denotes the respective cross-ratio. On the other hand, the diagonal term $-\delta_{ii}'(0)=\rho(\partial{D}_{\g_i}, \partial{D_{\g_i^{-1}}})$ is the \emph{translation length} of $\g_i$, or equivalently, it equals $-\log |y_i|,$ where~$y_i$ is the third Koebe coordinate of $\g_i$.     
\end{rem}

We now reinterpret $D'(0)$ in a way that will allow us to conclude it is invertible. Let $G \subseteq \Hy_k$ denote the convex hull of the set $\{\partial{D}_{\alpha} | \alpha=\g_i, \g_i^{-1}, i=1,2,\dots, g\}$. We remark that it is a finite metric graph with respect $\rho$. Let us denote by $\sigma_l$, $l=1,2,\dots, N$, its edges. Moreover, we fix an orientation on $G$, which is equivalent to fixing homeomorphisms ${\gamma_i \colon[0,1] \to \sigma_i}$.

Let $\mathcal{C}_1(G):= \langle [\sigma_1], \ldots, [\sigma_N] \rangle$ be the free $\R$-vector space generated by the oriented $\sigma_i$ for $i = 1, 2,  \ldots , N$.
 We consider the dual vector space $\mathcal{C}^1(G):= \Hom (\mathcal{C}_1(G), \R)$, which has a natural basis $(\mathbbm{1}_{\sigma_i})_i$, dual to $([\sigma_i])_i$. For $f = \sum_{i=1}^N a_i \mathbbm{1}_{\sigma_i} \in \mathcal{C}^1(G)$ and $c = \sum_{j=1}^N b_j [\sigma_j] \in \mathcal{C}_1(G)$, we define a pairing: 
\begin{equation}
\langle f, c\rangle := \sum_{l=1}^N a_l b_l \rho(\sigma_l).
\end{equation}

Let $\psi: \mathcal{C}_1(G) \rightarrow \mathcal{C}^1(G), c \mapsto f_c$, be the isomorphism induced by ${[\sigma_i] \mapsto \mathbbm{1}_{\sigma_i}}.$ Clearly, for $c=\sum_{i=1}^N \alpha_i [\sigma_i]$, one has $\langle f_c , c \rangle = \sum_{i=1}^N \alpha_i^2 \rho(\sigma_i),$
so the pairing is non-degenerate.

For $i=1,2,\dots, g$, we denote by $A_i$ the unique injective path in $\Hy_k$ connecting the boundaries of $D_{\g_i}$ and $D_{\g_i^{-1}}$. We note that it is a union of edges of the finite graph $G$ and oriented by design.

\begin{prop} \label{prop_invertible_Dprime} Let $[A_i]:=[\partial{D}_{\gamma_i}, \partial{D}_{\gamma_i^{-1}}]$ be the path in $G$ \emph{oriented} by going from $\partial{D}_{\g_i}$ to~$\partial{D}_{\g_i^{-1}}$, $i=1,2,\dots, g$. Then $\{[A_i]\}_{i=1}^g$ is linearly independent in $\mathcal{C}_1(G)$, and for any $i,j,$ $$-\delta_{ij}'(0)=\langle \psi([A_j]),  [A_i]  \rangle.$$
In particular, the matrix $D'(0)=(\delta_{ij}'(0))_{i,j}$ is invertible. 
\end{prop}

\begin{proof} 
As for $i=1,2,\dots, g$, we have  $\partial{D}_{\g_i} \in A_i \backslash \bigcup_{j \neq i} A_j$, the vectors $[A_i], i=1,2,\dots, g,$ are linearly independent in $\mathcal{C}_1(G)$.

Let $[A_i] = \sum_l \alpha_{il} [\sigma_l]$.Here $\alpha_{il}=0$ if $\sigma_{il} \not\subseteq A_i$, or $\alpha_{il}=1$ when the orientation on~$\sigma_{il}$ coincides with the orientation induced on it by  $A_i$, and $\alpha_{il}=-1$ otherwise. Similarly for $[A_j].$
The pairing then yields: 
$$
\langle  \psi([A_j]), [A_i] \rangle = \sum_{l} \alpha_{il} \alpha_{jl} \rho(\sigma_l) = \sum_{\sigma_l \subset A_i \cap A_j} \alpha_{il} \alpha_{jl} \rho(\sigma_l).
$$
This sum is $0$ if there are no such $l$, \emph{i.e.} if $A_i \cap A_j=\emptyset$. Otherwise, the product $\alpha_{il} \alpha_{jl}$ does not depend on $l$: if $A_i$ and $A_j$ induce the same orientation on $A_i \cap A_j$, then $\alpha_{il}\alpha_{jl}=1$ for all $l$ satisfying $\sigma_l \subset A_i \cap A_j$, and $\alpha_{il}\alpha_{jl}=-1$ otherwise. Hence, $\alpha_{il}\alpha_{jl}=\varepsilon_{ij}$ for all such $l$, with $\varepsilon_{ij}$ as in Corollary \ref{birap}. Consequently, $\langle  \psi([A_j]), [A_i] \rangle=\varepsilon_{ij}\rho(A_i \cap A_j)=-\delta_{ij}'(0)$.
Since the pairing is non-degenerate and the $[A_i]$ are linearly independent, we conclude that $D'(0)$ is an invertible matrix.  
\end{proof} 

\subsection{Analytic continuation at zero}

We use the notation introduced in the previous section. We introduced there a matrix $\tilde{M}(s)$ (\emph{cf.} \eqref{eq_tildeM}), obtained via a base change $P$ from the matrix $M(s), s \geqslant 0$, associated to a Schottky figure~$\mathcal{D}$ for the rank $g>1$ Schottky group $\Ga$ over the non-Archimedean field $k$. We hence have that 
$I-M(s)=P(I-\tilde{M}(s))P^t$, and so when their inverses exist, that $(I-M(s))^{-1}=P (I-\tilde{M}(s))^{-1} P^t$. We recall the block matrix notations of \eqref{eq_tildeM}, implying 
\begin{align} \label{I-tildem}
I - \tilde{M}(s)= \left  ( \begin{array}{ll}
I - A(s) & -B(s) \\
 -C(s) & I-D(s)
\end{array} \right ). 
\end{align}

Let $L(s): \R^{2g} \rightarrow \R^{2g}$ denote the linear 
morphism induced by $I-M(s)$ on $\R^{2g}$. Then, in the orthonormal basis $
\mathcal{B}:=(u, u_2, \dots, u_g, v_1, \dots, v_g)$ of eigenvectors of $M(0)$ (see Lemma \ref{lem_characteristic_zero}), the matrix of $L(s)$ is precisely $I-\tilde{M}(s)$. Similarly, when it exists, 
let $R(s): \R^{2g} \rightarrow \R^{2g}$ denote the linear 
morphism induced by $(I-M(s))^{-1}$. Its matrix in the basis $
\mathcal{B}$ is $(I-\tilde{M}(s))^{-1}$.

\begin{lem} \label{lem_schur}
There exist  matrices $\tilde{A}(s), \tilde{B}(s), \tilde{C}(s), \tilde{D}(s) \in M_{g\times g}(\R)$ with analytic entries in~$s$,   such that $\tilde{A}(s)$ and~$\tilde{D}(s)$ are invertible in a neighborhood of $s=0$, and
\begin{align*}
I - \tilde{M}(s)= \left  ( \begin{array}{ll}
\tilde{A}(s) & s\tilde{B}(s) \\
 s\tilde{C}(s) & s\tilde{D}(s)
\end{array} \right ). 
\end{align*}
Moreover, there is a neighborhood $N$ of $s=0$ such that $I-\tilde{M}(s)$ is invertible in $N \backslash \{0\}$, and 
\begin{equation} \label{schur}
\left  ( I-\tilde{M}(s) \right )^{-1} = \left ( \begin{array}{ll}
S^{-1} & -  S^{-1} \tilde{B} {\tilde{D}^{-1}} \\
 - {\tilde{D}^{-1}} \tilde{C} S^{-1} &  {\tilde{D}^{-1}} \tilde{C} S^{-1} \tilde{B} {\tilde{D}^{-1}} + \dfrac{\tilde{D}^{-1}}{s}
\end{array} \right ),
\end{equation}
where $S(s): = \tilde{A}(s) - s{ \tilde{B}(s) \tilde{D}(s)^{-1}\tilde{C}(s) }$ is the \emph{Schur complement}, which is invertible in a neighborhood of $s=0$.
\end{lem}

\begin{proof}
Set $\tilde{A}(s)=I-A(s)$, $-B(s)=s\tilde{B}(s)$, $-C(s)=s\tilde{C}(s)$ and $I-D(s)=s\tilde{D}(s)$. As $I-A(0)$ is invertible and $A(s)$ has analytic entries, $\tilde{A}(s)$ has analytic entries and is invertible in a neighborhood of $s=0$. Similarly, as $B(0)=C(0)=I-D(0)=0$ (see \eqref{eq_tildeM}) and all three matrices are analytic, we obtain that $\tilde{B}(s), \tilde{C}(s)$ and $\tilde{D}(s)$ are analytic in $s$. Using the Taylor expansion around $s=0$, we obtain that $I-D(s)=(I-D(0))-sD'(0)+o(s)$, implying $\tilde{D}(s)=-D'(0)+{o(1)}$. By Proposition \ref{prop_invertible_Dprime}, $D'(0)$ is an invertible matrix, so in a neighborhood of $s=0$, $\tilde{D}(s)$ is also invertible. The first equality is now a consequence of \eqref{I-tildem}.

The second equality is  the classical Schur complement formula (see \emph{e.g.} \cite[Theorem~5.1]{grigorchuk_nekrashevych}), and the invertibility of $S(s)$ near $0$ is due to the fact that $S(0)=\tilde{A}(0)$. 
\end{proof}

Let $E_1$ be the  $g$-dimensional subspace of $\R^{2g}$ spanned by $\mathcal{B}_1:=(u, u_2, \dots, u_g)$, and $E_2$ its orthogonal generated by $\mathcal{B}_2:=(v_1, v_2,\dots, v_g).$ Let $\pi_i: \R^{2g} \rightarrow E_i$ denote the orthogonal projection, $i=1,2.$

\begin{prop}\label{thm_resolvant_expansion}
For $i,j \in \{1,2\},$ let $R_{ij}(s): E_j \rightarrow E_i$ be the linear morphism induced by the~$(i,j)$ block of $(I-\tilde{M}(s))^{-1}$ in \eqref{schur} in the bases $\mathcal{B}_j, \mathcal{B}_i$. Then ${R(s)=\sum_{1 \leqslant i,j \leqslant 2} R_{ij}(s) \pi_j}$ and
\begin{enumerate}
\item $R_{11}(0)=R(0)_{|E_1}$,
\item $R_{11}(s), R_{12}(s)$ and $R_{21}(s)$ are analytic at $s=0$,
\item $R_{22}(s)$ is meromorphic with a simple pole at $s=0$.
\end{enumerate}
\end{prop}

\begin{proof}
The equality $R(s)=\sum_{1 \leqslant i,j \leqslant 2} R_{ij}(s) \pi_j$ can be directly checked on the basis $\mathcal{B}=\mathcal{B}_1 \cup \mathcal{B}_2$ of $\R^{2g}$. 
The first point is immediate seeing as $R_{11}(0)$ is given by the matrix $(I-A(0))^{-1}$, \emph{i.e.} the upper-right block of the matrix $(I-\tilde{M}(0))^{-1}$. 

Given the invertibility of $S$ and $\tilde{D}$ in a neighborhood of $s=0$, from \eqref{schur} we obtain that~$R_{11}(s), R_{12}(s)$ and $R_{21}(s)$ are analytic at $0$.
Similarly, as $\tilde{D}^{-1}\tilde{C}S^{-1}\tilde{B}\tilde{D}^{-1}+\dfrac{\tilde{D}^{-1}}{s}$ is meromorphic with simple pole at $0$, so is the linear map $R_{22}(s)$.
\end{proof}

We can now prove the main theorem of this section. Recall the equality obtained in Theorem~\ref{thm_meromorphic_extension}, where the vectors $X_x(s)$ and $Y_y(s)$ are analytic in $s$. 

\begin{thm} \label{analyticpoincare} Let $k$ be a non-Archimedean field and $\Ga \subseteq \PGL_2(k)$ a Schottky group of rank~${g>1}$. Let $\mathcal{D}$ be a Schottky figure for $\Ga$ in $\bP^{\pan}_k$.  Let $x, y \in \Hy_k.$
The associated Poincaré series $\Po_{x,y}(s), s \in \R_{\geqslant 0},$ can be meromorphically extended to $\C$ via $$\Po_{x,y}(s):=X_x(s)(I-M_{\mathcal{D}}(s))^{-1} Y_y(s)^{t}+e^{-s\rho(x,y)}.$$ Moreover, this continuation is \emph{analytic} at $s=0$, and 
$$\Po_{x,y}(0)=\frac{1}{1-g}.$$
\end{thm}

\begin{proof}
By Remark \ref{rem_meroext2}, we may assume that $x, y$ are not in the Schottky figure, so Theorem~\ref{thm_meromorphic_extension} is applicable and ensures the meromorphic continuation {over $\C$}. Let $X_x(s)=X_1(s)+X_2(s)$ and $Y_y(s)=Y_1(s)+Y_2(s)$ be the decompositions of $X_x(s), Y_y(s)$ in $E_1 \oplus E_2$, where $E_1, E_2$ are as in Proposition \ref{thm_resolvant_expansion}. As $X_x(0)=Y_y(0)=U$  is the vector with all entries equal to $1$, and $U \in E_1$, we get that ${X_2(0)=Y_2(0)=0 \in E_2}$.  Using the orthogonality of $E_1$ and $E_2$, and Proposition~\ref{thm_resolvant_expansion}, we obtain

\vspace{0.4cm}

{\centering
  $ \displaystyle
    \begin{aligned}
\cP_{x,y}(s)-e^{-s\rho(x,y)} & = \langle X_x(s), R(s)Y_y(s)^{t} \rangle = \sum_{1 \leqslant i,j \leqslant 2} \langle X_1(s) + X_2(s) , R_{ij}(s) Y_j(s) \rangle  \\
 & = \sum_{i,j} \langle X_i(s) , R_{ij}(s) Y_j(s) \rangle \\
\end{aligned}
  $ 
\par}

 \smallskip

By Proposition \ref{thm_resolvant_expansion},  $R_{11}(s), R_{12}
(s), R_{21}(s)$ are holomorphic at $0$. On the other hand, 
$R_{22}(s)$ is meromorphic with a simple pole at $0$, but since $X_2(s), Y_2(s)
$ are holomorphic and vanish at $s=0$, the product $\langle 
X_2(s), R_{22}(s) Y_2(s)\rangle$ is also holomorphic and vanishes 
at $s=0$. In addition to that, as $X_2(0)=Y_2(0)=0$, what remains 
is $\Po_{x,y}(0)-1=\langle X_1(0), R_{11}(0)Y_1(0) \rangle$. 

Finally, by Proposition \ref{thm_resolvant_expansion}, $R_{11}(0)=R(0)_{|E_1}$. Since $X_1(0)= Y_1(0)= U \in E_1$ is an eigenvector of $M(0)$ for the eigenvalue $2g-1$, it is also an eigenvector of $R(0)$ for the eigenvalue~$\frac{1}{1-(2g-1)},$ so
$$
\cP_{x,y}(0)-1 = \langle U , R_{11}(0) U \rangle = \dfrac{1}{1 - (2g-1)} \langle U, U \rangle = \dfrac{1}{1 - g} -1. 
$$
\end{proof}

\begin{rem}
We note that if $x=y$, then the ``reduced Poincaré series'' $\mathcal{P}'_x(s):=\sum_{\g \in \Ga \backslash \{\mathrm{id}\}} e^{-s\rho(x, \g x)}$ can be analytically extended to $0$ with value $$\Po'_x(0)=\frac{1}{1-g}-1.$$

When $x=y$, it makes sense to look at the reduced Poincaré series instead, seeing as that means we don't count the trivial loop $[x, \mathrm{id} \cdot x]$ in the quotient Mumford curve $(\bP^{\pan}_k \backslash \La_{\Ga})/\Ga.$
\end{rem}

\begin{rem}
The Mumford curve $(\bP^{\pan}_k \backslash \La_{\Ga})/\Ga$ can be retracted by deformation to a finite metric graph of genus $g$ (which is the quotient by $\Ga$ of the graph $G$ from Section \ref{sect_comb_lemma}). The analytic continuation of the (reduced) Poincaré series is then reminiscent of that of N. Anantharaman presented at her Collège de France course ``Spectres de graphes et de surfaces''.
\end{rem}

\begin{rem} Heuristically, as $M(0)U=(2g-1)U$ and $X(0)^{t}=Y(0)=U$, we get that $$\Po_{x,y}(0)-1=X_x(0)(I-M(0))^{-1}Y_y(0)^{t}=\frac{U^{t}U}{2-2g}=\frac{2g}{2-2g}=\frac{1}{1-g}-1.$$
\end{rem}

\section{Moduli space of  Schottky groups over $\Z$} \label{sect_2}

In this section we recall some definitions and properties of the theory of Berkovich spaces over~$\Z$ (see \cite{poineau_berkovich_Z} and \cite{lemanissier_poineau}), as well as the construction of a moduli space $\eS_g$ defined over $\Z$ of Schottky groups of rank $g \geqslant 1$ (see \cite{poineau_turcheti_universal}).

\subsection{Berkovich spaces over $\Z$} \label{section_berkovich_Z}
{We start by fixing some terminology.
\begin{defi} \label{def_banachring}
Let $A \neq 0$ be a commutative ring with unity. Let $\|\cdot\|: A  \rightarrow \mathbb{R}_{\geqslant 0}$ be such that
\begin{enumerate}
    \item $\|0\|=0, \|1\|=1$,
    \item $\|xy\| \leqslant \|x\| \|y\|$ for all $x, y \in A$,
    \item $\|x-y\| \leqslant \|x\|+\|y\|$ for all $x, y \in A$.  
\end{enumerate}
When it is complete, $(A, \|\cdot\|)$ is said to be a \emph{Banach ring}.
\end{defi}}

{From now on, let $(A, \|\cdot\|)$ be a Banach ring.}

\begin{defi}[{\cite[1.5]{berkovich_spectral}}] Let $n \in \Z_{\geqslant 0}$. The \emph{affine analytic space $\mathbb{A}^{n, \an}_{A}$} over $(A, \|\cdot\|)$ is defined as:
\begin{enumerate}
\item setwise, the multiplicative semi-norms $|\cdot|$ on $A[T_1, T_2, \dots, T_n]=:A[\underline{T}]$ such that $|\cdot|_{|A} \leqslant \|\cdot\|$, with the convention that $A[\underline{T}]=A$ when $n=0$,
 \item the topology on $\mathbb{A}^{n, \an}_A$ is the coarsest such that $\varphi_P: \mathbb{A}^{n, \an}_A \rightarrow \R$, $|\cdot| \mapsto |P|$, is continuous for all $P \in A[\underline{T}].$
\end{enumerate}
When $n=0$, we denote by $\M(A)$ the space $\Af_A^{0, \an}$, and call it the \emph{Berkovich spectrum} of $A$. 
\end{defi}

\begin{rem} \label{rem_banachring}
The ring $(\Z, |\cdot|_{\infty}),$ where $|\cdot|_{\infty}$ denotes the usual absolute value, is a Banach ring. 
The morphism $(\mathbb{Z}, |\cdot|_{\infty}) \rightarrow (A, \|\cdot\|), m\mapsto m$, is bounded, \emph{i.e.} $\|m\| \leqslant |m|_{\infty}$. For any~$n \geqslant 0,$ this naturally induces a surjective continuous morphism ${\pi_A: \A_A^{n, \an} \rightarrow \A_{\Z}^{n, \an}}$, which we will call a \emph{projection}. We do \emph{not} make mention of the potentially (non-){analytic} nature of this morphism, we will merely use that it is continuous. 
\end{rem}

The topological space $\M(A)$ is non-empty, compact and Hausdorff (\cite[Theorem 1.2.1]{berkovich_spectral}). For $n \geqslant 0$, the topological space $\Af^{n, \an}_A$ is locally compact and Hausdorff (\emph{cf.} \mbox{\cite[Remark 1.5.2]{berkovich_spectral}}). 
There is a continuous structural morphism $\Af^{n, \an}_A \rightarrow \M(A)$ induced by the restriction $|\cdot| \mapsto |\cdot|_{|A}$ (it is an \emph{analytic} morphism by {\cite[\S~2.1.1]{lemanissier_poineau}}). 
Let us describe its fibers.

To a point $x \in \Af^{n, \an}_A$ (corresponding to a semi-norm $|\cdot|_x$ on $A[\underline{T}]$) one associates the \emph{completed residue field} $\mathcal{H}(x)$, defined as the completion of the field $\mathrm{Frac}(A[\underline{T}]/\mathrm{Ker}{|\cdot|_x})$ with respect to the quotient norm induced by $|\cdot|_x$. There is always a map $A[\underline{T}] \rightarrow \mathcal{H}(x), f \mapsto f(x)$, induced by the projection $A[\underline{T}] \rightarrow A[\underline{T}]/\mathrm{Ker}{|\cdot|_x}.$ Moreover, $|f|_x=|f(x)|_{\mathcal{H}(x)}$, where $|\cdot|_{\mathcal{H}(x)}$ denotes the norm on~$\mathcal{H}(x)$. We often simply write $|f(x)|$ instead of $|f(x)|_{\mathcal{H}(x)}$. 

\begin{ex} \label{ex_MZ} Let us consider the Banach ring $(\Z, |\cdot|_{\infty}).$ Using Ostrowski's Theorem, one obtains that $\M(\Z)$ consists of the following points:
\begin{enumerate}
\item $|\cdot|_{\infty}^{\varepsilon}$ for $\varepsilon \in (0,1]$, in which case $\Hr(|\cdot|_{\infty}^{\varepsilon})=(\R, |\cdot|_{\infty}^{\varepsilon}),$
\item $|\cdot|_{p}^{\alpha}$ for $\alpha \in (0, +\infty)$, where $p$ is a prime number and $|\cdot|_p$ denotes the $p$-adic norm, in which case $\Hr(|\cdot|_{p}^{\alpha})=(\Q_p, |\cdot|_{p}^{\alpha})$,
\item the trivial norm $|\cdot|_0$ satisfying $|n|_0=1$ for any $n \in \Z \backslash \{0\},$ in which case $\Hr(|\cdot|_0)=(\Q, |\cdot|_0),$
\item the semi-norms $|\cdot|_{p, \infty}$, given by $n \mapsto 0$ if $p|n$, and $n \mapsto 1$ otherwise, in which case $\Hr(|\cdot|_{p, \infty})=(\mathbb{F}_p, |\cdot|_0)$. 
\end{enumerate}
We remark that the set of points $\{|\cdot|_{\infty}^{\varepsilon}: \varepsilon \in (0, 1] \}$, \emph{i.e.} the \emph{Archimedean part} of $\M(\Z)$, is an open (\emph{cf.} \cite[\S~2.1]{poineau_turcheti_universal}).  The same is true for $\Af_{\Z}^{n, \an}$, \emph{i.e.} $\{x \in \Af_{\Z}^{n, \an}: \Hr(x)  \text{ is Archimedean}\}$ is an open.
\end{ex}

Let $\pi: \Af^{n, \an}_A \rightarrow \M(A)$ be the structural morphism obtained by restricting the semi-norms to $A$. For $x \in \M(A)$, the fiber $\pi^{-1}(x)$ is homeomorphic to the $\mathcal{H}(x)$-Berkovich analytic affine space $\Af^{n, \an}_{\Hr(x)}$.  Thus, $\Af_{\Hr(x)}^{n, \an}$ can be  embedded into $\Af_A^{n, \an}$.

\begin{ex} \label{ex_scale}
In the case of $\Z$, the fiber of $\pi: \Af_{\Z}^{n, \an} \rightarrow 
\M(\Z)$ over $|\cdot|_{\infty}^{\varepsilon}$ is identified to 
$\Af_{\R, |\cdot|_{\infty}^{\varepsilon}}^{n, \an}$, $\varepsilon \in (0,1]$. Here 
$\Af_{\R, |\cdot|_{\infty}^{\varepsilon}}^{n, \an}$ is homeomorphic to $\C^n/
\mathrm{conj.}$ via  $[\underline{z}]:=[(z_1, z_2, \dots, z_n)] \mapsto 
|\cdot|_{v_{z, \varepsilon}}$, where $|\cdot|_{v_{z, \varepsilon}}: P(T_1, 
T_2, \dots, T_n) \mapsto |P(z_1, z_2, \dots, z_n)|_{\infty}^{\varepsilon}.$ 
We remark that all Archimedean fibers of $\pi$ are homeomorphic via $|\cdot| \in \Af^{n, \an}_{\R, |\cdot|_{\infty}} \mapsto |\cdot|^{\varepsilon} \in \Af^{n, \an}_{\R, |\cdot|_{\infty}^{\varepsilon}}.$ 

For a prime $p$, the fiber of $\pi$ over $|\cdot|_p^{\alpha}$ is $\Af_{\Q_p, |\cdot|_{p}^{\alpha}}^{n,\an}$ for $\alpha \in (0, +\infty)$. For any $\beta \in (0, +\infty)$, the fibers $\pi^{-1}(|\cdot|_{p}^{\alpha})$
and $\pi^{-1}(|\cdot|_{p}^{\beta})$ are homeomorphic.  
Finally, if $x=|\cdot|_{p, \infty}$, then $\pi^{-1}(x) \cong \Af_{\mathbb{F}_p}^{n, \an}$, and if $x=|\cdot|_{0}$, then $\pi^{-1}(x) \cong \Af^{n, \an}_{(\Q, |\cdot|_0)}$.
\end{ex}

If $(k, |\cdot|_k)=\Hr(z)$ for some $z \in 
\M(\Z)$,  there is an embedding $\varphi_k: k^n \hookrightarrow \Af_{\Z}^{n, 
\an}$ via $\underline{a}:=(a_1, a_2, \dots, a_n) \mapsto |
\cdot|_{\underline{a}},$ where $|P|_{\underline{a}}:=|
P(\underline{a})|_k$ for $P \in \Z[T_1, T_2, \dots, T_n].$ 
One can similarly define a map $\varphi_{\C, |\cdot|_{\infty}^{\varepsilon}}$ for $\varepsilon \in (0,1]$, but it will no longer be injective (unless we identify conjugated elements of $\C^n$).

Similar descriptions of $\Af^{n, \an}_A$ can be given for other Banach rings $A$ such as $(\C, |\cdot|_\mathrm{hyb})$, where $|\cdot|_{\mathrm{hyb}}:=\max(|\cdot|_{\infty}, |\cdot|_{0})$, the integer rings of number fields, and complete discretely valued rings. If there is a bounded map of Banach rings $(\Z, |\cdot|_{\infty}) \rightarrow (A, \|\cdot\|)$, then one has $\Af_A^{n, \an}=\Af_{\Z}^{n, \an} \times_{\M(\Z)} \M(A)$ whenever this fiber product exists (see also Remark \ref{fiberproducts}).

For the sake of completion, we briefly recall the structural sheaf of analytic functions here, though we will not make explicit use of it.

\begin{defi}[{\cite[1.5]{berkovich_spectral}}]
For $U \subseteq \Af^{n, \an}_A$ open, let $${K(U):=\left\{\frac{P}{Q} | P, Q \in A[\underline{T}] \ | \ |Q|_x \neq 0, \ \forall x \in U \right\}}.$$ Let $\mathcal{O}(U)$ be the set of functions $f: U \rightarrow \prod_{x \in U} \mathcal{H}(x)$, such that 
\begin{enumerate}
\item $f(x) \in \mathcal{H}(x)$ for all $x \in U$,
\item for $x \in U$, there exists a neighborhood $V$ of $x$ such that $f$ is a uniform limit of elements of~$K(V).$
\end{enumerate}
An $A$-\emph{analytic space} is a locally ringed space which is locally isomorphic to $(V(\mathcal{I}),\mathcal{O}_U/\mathcal{I})$, with~$U$ an open of some $\Af^{n, \an}_A$, and $\mathcal{I}$ a coherent ideal sheaf of $\mathcal{O}_U$.
\end{defi}

\begin{rem} \label{fiberproducts} 
Let $(B, \|\cdot\|_B)$ be a Banach ring, and  $f: A \rightarrow B$ a \emph{bounded} morphism, \emph{i.e.} there exists $C>0$ such that $\|f(a)\|_B \leqslant C\|a\|_A$ for all $a \in A$. The $B$-analytic spaces for all such Banach rings $B$, together with a notion of \emph{analytic morphisms}, form a category (\emph{cf.}~{\mbox{\cite[\S~2.2]{lemanissier_poineau}}}). It is not clear that fiber products exist in this category for any $A$, though this is the case whenever $A$ is \emph{e.g.} a complete valued field, a hybrid field, the integer rings of number fields, complete discretely valued rings, \emph{etc.} (\emph{cf.}~{\mbox{\cite[\S~4.3]{lemanissier_poineau}}}). 
\end{rem}

\begin{defi}
Let $X$ be an $A$-analytic space. A point $x \in X$ is said to be \emph{Archimedean} if~$\mathcal{H}(x)$ is Archimedean, and \emph{non-Archimedean} otherwise. 
\end{defi}

Another important example of $A$-analytic spaces is the relative projective line. To define it, one can glue along opens of two copies of $\Af_A^{1, \an}$.

\begin{ex} \label{relproj} We fix a coordinate function $T$ on $\Af_A^{1, \an}$ so that its elements are multiplicative semi-norms on $A[T]$.  Let $U:=\{x \in \Af_A^{1, \an} | |T|_x \neq 0\}$. Similarly, let us take a copy of $\Af_A^{1, \an}$ where we denote by $S$ the coordinate function, and let $V:=\{x \in \Af_A^{1, \an}| |S|_x \neq 0 \}$. We glue these two copies of the affine line via $A[T, \frac{1}{T}] \rightarrow A[S, \frac{1}{S}], T \mapsto \frac{1}{S}$, which induces an isomorphism $U \cong V.$ The resulting space is the relative analytic projective line $\bP_A^{\pan}$ over $A$.

There is a structural morphism $\pi: \bP_A^{\pan} \rightarrow \M(A)$, induced by the structural morphism of~$\Af^{1, \an}_A$. For any $x \in \M(A)$, the fiber $\pi^{-1}(x)$ can be identified to $\bP^{\pan}_{\Hr(x)}.$
\end{ex}

For $X:=\Af_{\Z}^{n, \an}$ one can construct $\bP_X^{\pan}:=\bP_{\Z}^{\pan} \times_{\M(\Z)} X$--the relative analytic projective line over $X$. For an open $U \subseteq X$, one can also define $\bP^{\pan}_U:=\bP^{\pan}_{\Z} \times_{\M(\Z)} U$--the relative projective line over $U$ (\emph{cf.} Remark \ref{fiberproducts}).  The fiber of the projection $\bP_X^{\pan} \rightarrow X$ over a point $x \in X$ can again be identified to~$\bP^{\pan}_{\Hr(x)}$.

\medskip

{\emph{Möbius transformations.}} There is an action of $\PGL_2(\mathcal{O}(U))$ on $\bP^{\pan}_U$ via Möbius transformations (\emph{cf.}  \cite[\S~2.2]{poineau_turcheti_universal}).  Given $M:=\begin{pmatrix}
a & b \\ c & d\\
\end{pmatrix}  \in \PGL_2(\mathcal{O}(U))$, the action is induced via the morphism $T \mapsto \frac{aT+b}{cT+d}$, where $T$ denotes a fixed coordinate function on $\bP^{\pan}_U$. 
   
   For any $x \in U$, the action of $\PGL_2(\mathcal{O}(U))$ on $\bP^{\pan}_U$ restricts to the action of $\PGL_2(\mathcal{H}(x))$ on the fiber $\bP^{\pan}_{\Hr(x)}$ via Möbius transformations.

\subsection{The modulus function over $\Z$}

We diverge momentarily from the goal of recalling the moduli space of Schottky groups over $\Z$ in order to study a useful function.

There is a notion of \emph{modulus} for annuli over $(\C, |\cdot|_{\infty})$,  denoted $\modu_{(\C, |\cdot|_{\infty})}(\cdot)$, sometimes also referred to as the \emph{extremal length function}  (\emph{cf.} \cite[Appendix B]{milnor_complex} or \cite[\S~6.3.1]{lyubich_book}), which  is~$\PGL_2(\C)$-invariant.
We start by providing an explicit computation for it,  inspired by~\cite{kisil}.

Let $\alpha, \beta \in \C$ and $R, r>0$ be such that the two disks $D(\alpha, R), D(\beta, r)$ \emph{cut out} an annulus $A$ in $\bP^1(\C)$. This means that $A=D(\alpha, R) \backslash D(\beta,r)$, $A=D(\beta, r) \backslash D(\alpha, R)$, or $A=D(\alpha, R)^c \backslash D(\beta, r)$. 

{
\begin{lem} \label{lem_complexmodulus} One has \begin{align*}
    \modu_{(\C, |\cdot|_{\infty})} A &=\mathrm{ch}^{-1} \left(  \left | \dfrac{|\alpha - \beta|_{\infty}^2 - R^2 - r^2}{2 Rr}\right |  \right ),
\end{align*}
where $\mathrm{ch}^{-1}$ denotes the inverse of the hyperbolic cosine 
function.
\end{lem}

\begin{proof} 
   A disk $D$ in $\bP^{1}(\C)$ is given by an inequality $ \{ z \ | \ \bar X^t M_D X \leqslant 0 \}$ where $X^t: = (z,1) \in \C^2$, and where $M_D$ is a hermitian matrix of the form $ \left ( \begin{array}{cc}
   a  & b  \\
  \bar b   &  c
\end{array} \right )$ (see \emph{e.g.} \cite[Lemma F.1]{kirillov}). For $U \in \mathrm{SL}_2(\C)$, one has $U(D)= \{z \ | \ {\bar X^t} (\overline{ U^{-1}}^t M_D U^{-1}) X \leqslant 0 \}$. 
Let us take the pairing $$ \langle M , N \rangle :=  \dfrac{1}{2 \sqrt{\det(M) \det(N)}}  \mathrm{Tr}( C^t(M) N),$$
where $C^t(M)$ denotes the transpose of the cofactor matrix of $M$. It is naturally invariant by the action of $\mathrm{SL}_2(\C)$ given by $U \cdot P:=\overline{ U^{-1}}^t P U^{-1}$.

There is always a Möbius transformation mapping the annulus $A$ to  a concentric one, meaning that the corresponding disks are concentric, (see \cite[\S~3.4]{mumford_series_wright}), and then by a translation and scaling, to $D(0,1) \backslash D(0, \lambda)$ for some $\lambda>0$. We recall that $\modu_{(\C, |\cdot|_{\infty})}(D(0,1) \setminus D(0, \lambda)) = -\log \lambda$. 
Let $M_0: = \left ( \begin{array}{cc}
     1 & 0  \\
     0& -1 
\end{array}\right)$ and  $N_0: = \left ( \begin{array}{cc}
     1 & 0  \\
     0& -\lambda^2 
\end{array} \right) $ be the matrices associated to $D(0,1)$ and~$D(0, \lambda)$, respectively, so that 
 $$\langle M_0, N_0 \rangle = -\dfrac{1}{2} (\lambda + \lambda^{-1}) = -\ch ((\modu_{(\C, |\cdot|_{\infty})}(D(0,1) \setminus D(0, \lambda))).$$  

As both sides are $\PGL_2(\C)$-invariant, $\modu_{(\C, |\cdot|_{\infty})} A=\ch^{-1}(- \langle M_0, N_0 \rangle)$, and $\langle M_0, N_0 \rangle= \langle M, N \rangle$, with $M, N$ the matrices determining the disks in $\bP^1(\C)$ that induce~$A$. When $A=D(\alpha, R) \backslash D(\beta,r),$ then $M:= \begin{pmatrix}
     1 & - {\alpha}  \\
   -  {\bar \alpha} & |\alpha|_{\infty}^2 - R^2 
\end{pmatrix}$, and $N:=\begin{pmatrix}
     1 & - {\beta}  \\
  -{\bar \beta} & |\beta|_{\infty}^2 - r^2 
\end{pmatrix}$ (the same result will be obtained when switching the roles of the two disks). When $A=D(\alpha, R)^c \backslash D(\beta, r)$, then $M':=\begin{pmatrix}
    -1 &   {\alpha}  \\
   {\bar  \alpha} & -|\alpha|_{\infty}^2 + R^2 
\end{pmatrix}$ and $N':=N$ are the matrices in question. 
The pairing yields
 \begin{align*}
     \ch (\modu_{(\C, |\cdot|_\infty)} A) &= - \langle M, N\rangle = -\dfrac{1}{2} \dfrac{|\alpha - \beta|_{\infty}^2 - R^2 - r^2}{Rr}, \\
     \ch (\modu_{(\C, |\cdot|_\infty)} A) & = - \langle M',N \rangle = - \dfrac{1}{2} \dfrac{R^2 + r^2 - |\alpha - \beta|_{\infty}^2}{R r}, 
 \end{align*}
 and one concludes by observing that $|\alpha - \beta|_\infty^2 < \max(R^2, r^2) < R^2 + r^2 $ in the first case, while $|\alpha - \beta|^2_\infty > (R+r)^2> R^2 + r^2$ in the second case. 
\end{proof}

\begin{defi} \label{defi_modulus}
Let $(k, |\cdot|_k)$ be a complete valued field. Let $\alpha, \beta \in k$, and $R,r>0$ be such that the disks $D(\alpha, R), D(\beta, r) \subseteq \Af^{\pan}_k$ cut out an annulus $A \subseteq \bP^{\pan}_k$.  When $|\cdot|_k = |\cdot|_\infty^{\varepsilon}$ for some $\varepsilon \in (0,1]$, let $A_\varepsilon$ be the image of $A$ via the map $ \bP^{\pan}_{(k, |\cdot|_k)}\to \bP^{\pan}_{(k, |\cdot|_\infty)} $. The \emph{modulus $\modu_{(k, |\cdot|_k)} A$ of the annulus $A$} is defined as:
\begin{equation}
  \modu_{(k, |\cdot|_k)} A:=  \begin{cases}
\varepsilon \modu_{(\C, |\cdot|_{\infty})}A_{\varepsilon} & \text{if } |\cdot|_k=|\cdot|_{\infty}^{\varepsilon},\\
\left| \log \left( \frac{|\alpha - \beta|_k}{R} \right) \right | + \left| \log \left(\frac{|\alpha - \beta|_k}{r} \right) \right |  & \text{if $k$ is non-Archimedean}.
    \end{cases}
\end{equation}
\end{defi}

\begin{rem}
When $(k, |\cdot|_k)=(\R, |\cdot|_{\infty}^{\varepsilon})$ for some $\varepsilon \in (0,1]$, the annulus $A \subseteq \bP^{\pan}_{(\R, |\cdot|_{\infty}^{\varepsilon})}$ induces uniquely an annulus in $\bP^{\pan}_{(\C, |\cdot|_{\infty}^{\varepsilon})}$ with the centers $\alpha, \beta$ on the real axis. In this case, $\modu_{(\R, |\cdot|_{\infty}^{\varepsilon})}A$ from Definition \ref{defi_modulus} is to be taken  as the modulus of the induced annulus over $(\C, |\cdot|_{\infty}^{\varepsilon})$.  
\end{rem}

We note that in the non-Archimedean case, one has $\modu_{(k, |\cdot|_k)}A=\rho_k(\eta_{\alpha,R}, \eta_{\beta, r})$, where $\rho_k$ is the interval length metric in $\Hy_k$ (\emph{cf.} Section \ref{section_hyperbolic}). Consequently, the modulus is a $\PGL_2(k)$-invariant function for any complete valued field $k$, Archimedean or not.

\begin{prop} \label{prop_modulus_extension} Let $n \geqslant 1$ and  $V \subseteq \mathbb{A}^{n, \an}_{\Z}$ an open. Let $\alpha, \beta \in \mathcal{O}(V),$ and $r,R \colon V \to \R_{>0}$ two continuous functions such that for any $x \in V$, the disks $D(\alpha(x), R(x)), D(\beta(x), r(x))$ cut out an annulus $A_x$ in $\bP^{\pan}_{\Hr(x)}$.  Then, the function $\modu_A : V \rightarrow \R_{\geqslant 0}$, $x \mapsto \modu_{\mathcal{H}(x)} A_x$, is continuous.
\end{prop}}

{
\begin{proof} Set  $A_x =  D(\alpha(x), R(x))\setminus D(\beta(x), r(x))$ (the case where the roles of the disks are switched being symmetrical) or $A_x= D(\alpha(x),R(x))^c \setminus D(\beta(x), r(x))$.
Let $x \in V$ be Archimedean. Let $y_x \in p^{-1}(|\cdot|_\infty)$ be the unique point such that $|\cdot|_{y_x}^{\varepsilon(x)} = |\cdot|_x$ for some $\varepsilon(x) \in (0,1]$ (\emph{cf.} Example~\ref{ex_scale}), where $p: \Af_{\Z}^{n, \an} \rightarrow \M(\Z)$ is the projection morphism (\emph{cf.} Example \ref{ex_MZ}).
Let~$A_{y_x}$ be the annulus in $\bP^{\pan}_{\Hr(y_x)}$ induced by $A_x$, meaning  $A_{y_x}=D(\alpha(x), R(x)^{\frac{1}{\varepsilon(x)}}) \backslash D(\beta(x), r(x)^{\frac{1}{\varepsilon(x)}})$ or $A_{y_x}=D(\alpha(x), R(x)^{\frac{1}{\varepsilon(x)}})^c \backslash D(\beta(x), r(x)^{\frac{1}{\varepsilon(x)}})$. By Definition \ref{defi_modulus}, one has that $\modu_A(x)=\modu_{\Hr(x)}(A_x)=\varepsilon(x)\modu_{\Hr(y_x)}(A_{y_x})$.

For $x \in V$ non-Archimedean, we recall that $$\modu(A_x)= 
 \left|\log \left(\dfrac{|\alpha(x) - \beta(x)|_x}{r(x)} \right) \right | + \left|\log  \left(\dfrac{|\alpha(x) - \beta(x)|_x}{R(x)} \right) \right|=\rho_x(\eta_{\alpha(x), R(x)},\eta_{\beta(x), r(x)}),$$ where $\rho_x$ denotes the length metric in $\Hy_{\Hr(x)}$.

Hence, we observe that the function $\modu$ is continuous over purely Archimedean or non-Archimedean opens of $V$. Recall that the Archimedean part of $V$ is open. Hence, it only remains to show that given $x_0 \in V$ non-Archimedean and $x \in V$ Archimedean, when $x \rightarrow x_0$, one gets $\modu_A(x) \rightarrow \modu_A({x_0})$.

Consider the functions $U_1(x): = |\alpha(x) - \beta(x)|_x/R(x), U_2(x):= |\alpha(x) - \beta(x)|_x/r(x)$, and $v(x): = |\log U_1(x)| + |\log U_2(x) |$. By Lemma \ref{lem_complexmodulus}, we have $\modu_A(x) = \varepsilon(x) \ch^{-1} (u(x))$, where 
\begin{align*}
    u(x) &:=  \dfrac{1}{2} \left | \dfrac{|\alpha(x) - \beta(x)|_x^{2/\varepsilon(x)} - R(x)^{2/\varepsilon(x)} -r(x)^{2/\varepsilon(x)} }{(R(x)r(x))^{1/\varepsilon(x)}} \right |, \\ 
    &= \dfrac{1}{2} \left | (U_1(x)U_2(x))^{1/\varepsilon(x)} - \left(\dfrac{U_2(x)}{U_1(x)}\right)^{1/\varepsilon(x)} - \left(\dfrac{U_1(x)}{U_2(x)}\right)^{1/\varepsilon(x)} \right |.
\end{align*}
When $A_x = D(\alpha(x), R(x)) \setminus D(\beta(x), r(x))$, then $U_1(x)<1 < U_2(x)$, so one obtains $$\left(\dfrac{U_1(x)}{U_2(x)}\right)^{1/\varepsilon(x)} <<(U_1(x)U_2(x))^{1/\varepsilon(x)} << \left(\dfrac{U_2(x)}{U_1(x)}\right)^{1/\varepsilon(x)} = \exp \left(\dfrac{v(x)}{\varepsilon(x)}  \right),$$ as $x\rightarrow x_0$ (which implies $\varepsilon(x) \rightarrow 0$), where $f(x)<< g(x)$ means $\lim_{x\rightarrow x_0}f(x)/g(x)=0$. 

When $A_x= D(\alpha(x), R(x))^c \setminus D(\beta(x), r(x)) $, the fact that $|\alpha(x) -\beta(x)|_x> R(x) + r(x)$ implies $U_i(x)> 1$ for $i=1,2$, hence $$ \max \left( \left(\dfrac{U_2(x)}{U_1(x)}\right) , \left(\dfrac{U_1(x)}{U_2(x)}\right)\right)^{1/\varepsilon(x)} << (U_1(x)U_2(x))^{1/\varepsilon(x)} = \exp  \left(\dfrac{v(x)}{\varepsilon(x)}  \right),$$ as $x\rightarrow x_0$. 
In both situations, the function $u(x)$ diverges to $+\infty$ with  $u(x) \underset{x\rightarrow x_0}{\sim} e^{v(x)/\varepsilon(x)} / 2$.  
Using the equivalence $\ch^{-1}(t) \underset{t\rightarrow +\infty}{\sim} \log (t)$, we obtain that 
\begin{equation*}
  \modu_A(x) = \varepsilon(x) \ch^{-1}(u(x))  \underset{x\rightarrow x_0}{\sim} \varepsilon(x) \log  \left( \dfrac{e^{v(x)/\varepsilon(x)}}{2}\right) 
  \underset{x\rightarrow x_0}{\sim}  v(x). 
\end{equation*}
Since $v(x_0) = \modu_A(x_0)$, we conclude that
 that $\lim_{x\rightarrow x_0} \modu_A(x) = \modu_A(x_0) $, as required. 
\end{proof}
}

\subsection{Moduli space of Schottky groups} \label{section_modulispace}
Let $g \geqslant 1$. The following construction of a Berkovich moduli space over $\Z$ for rank $g$ Schottky groups is due to Poineau and Turchetti, \emph{cf.} \cite[\S~4]{poineau_turcheti_universal}.  

 If $g=1,$ set $\eS_1:=\{x \in \Af^{1, \an}_{\Z}| 0<|Y|_x<1\}$, where $Y$ is a fixed coordinate function on~$\Af_{\Z}^{1, \an}.$ Then to any $x \in \eS_1$, one associates a Schottky group over $\Hr(x)$ generated by a single transformation which is determined by the Koebe coordinates $(0, \infty, Y(x)).$ 

Assume now that $g \geqslant 2$. Let us choose coordinates $X_3, X_4, \dots, X_g,$ $X_2', X_3', \dots, X_g', Y_1, \dots, Y_g$ for the space $\Af_{\Z}^{3g-3, \an}$. 

\begin{defi}
Let $U_g$ be the open subset of $\Af^{3g-3, \an}_{\Z}$ consisting of the points $x$ such that $$\begin{cases}
|Z-Z'|_x \neq 0, & \ \text{for any two {distinct}}  \ Z, Z' \in \{X_i, X_j': 3 \leqslant i \leqslant g, 2\leqslant j \leqslant g\},\\
0<|Y_l|_x<1, & \text{for any} \ l=1,2,\dots, g.
\end{cases}
$$
\end{defi}

We remark that the first condition is equivalent to $Z(x) \neq Z'(x)$ in $\mathcal{H}(x)$.
Set ${X_1:=0}$, ${X_2:=1}$ and $X_1':=\infty$. For $i=1,2,\dots, g,$
let us denote by $M_i \in \PGL_2(\mathcal{O}(U_g))$ the transformation with Koebe coordinates $(X_i, X_i', Y_i)$ (\emph{cf.} \cite[Proposition 3.4.3]{poineau_turcheti_universal}). Let $\Ga$ be the subgroup of $\PGL_2(\mathcal{O}(U_g))$ generated by $(M_1, \dots, M_g).$

We note that for any $x \in U_g$, one has a transformation $M_i(x) \in \PGL_2(\Hr(x))$ determined by the Koebe coordinates $(X_i(x), X_i'(x), Y_i(x))$, $i=1,2,\dots, g$.

\begin{defi} \label{def_modspace}
Let $\eS_g$ be the subset of $U_g$ consisting of $x \in U_g$ such that the subgroup $\Ga_x:=\langle M_1(x), M_2(x), \dots, M_g(x) \rangle$ of $\PGL_2(\Hr(x))$ is a Schottky group.  We denote its limit set in~$\bP^{1}(\Hr(x))$ by $\La_x$. 
\end{defi}

\begin{rem}
The trivial valuation on $\Q[\underline{T}]$ determines a point $x_0$ in $\Af_{\Z}^{m,\an}$. However, $x_0 \not \in \eS_g$, as Schottky groups don't exist over trivially valued fields (seeing as loxodromic elements don't exist over such fields). In other words, there is no uniformization theory for curves over trivially valued fields.  
\end{rem}

One can make the same construction of a space $\eS_{g,A}$ starting with an arbitrary Banach ring~${(A, \|\cdot\|)}$. For $x \in \eS_{g,A}$, we continue denoting by $\Ga_x$ the Schottky group over $\Hr(x)$ generated by the matrices $M_i(x)$, constructed as above. 
When the fiber products in the corresponding categories exist, one has $\eS_{g,A}=\eS_g \times_{\M(\Z)} \M(A)$. This will always be the case when $A$ is a complete valued field, a hybrid field, a complete discretely valued ring, the ring of integers of a number field, \emph{etc}. (\emph{cf.} also \cite[Remark~4.2.3]{poineau_turcheti_universal}). For the sake of completeness, we include the following, which holds for a general Banach ring (regardless of the existence of the mentioned fiber products).
\begin{lem} \label{lem_basechangemoduli}{
If $g\geqslant 2$, set $m=3g-3$, and if $g=1$, set $m=1$. Let ${\pi_A: \A_A^{m, \an} \rightarrow \A_{\Z}^{m, \an}}$ denote the continuous projection from Remark \ref{rem_banachring}. Then $\pi_A^{-1}(\eS_g)=\eS_{g, A}$.}
\end{lem}

\begin{proof} For $g=1$, this is immediate from the definition of $\pi_A$ and that of $\eS_{1}$.

{Assume $g \geqslant 2.$ Let us fix the coordinate functions $X_i, X_j', Y_l$, $i=3,4,\dots, g, j=2,3,\dots, g,$ $l=1,2,\dots, g$, on $\A_A^{3g-3, \an}$ and $\A_{\Z}^{3g-3, \an}$. Set $X_1=0, X_2=1, X_1'=\infty$.} 

Let $x \in \A_{A}^{3g-3, \an}$, and set $y=\pi_A(x)$. This induces a complete valued field extension~${\Hr(x) /\Hr(y)}$. 
{For any $1 \leqslant l \leqslant g$,  and any two distinct elements $Z,Z' \in \{ X_i, X_j': 3 \leqslant i \leqslant g, 2 \leqslant j \leqslant g \}$, we note that $|Z-Z'|_x \neq 0$, $0 < |Y_l|_x < 1$, if and only if $|Z-Z'|_y \neq 0$, $0 < |Y_l|_y < 1$.}
Let us assume that the points $x, y$ satisfy these conditions. Using Koebe coordinates, let $M_i^{A}(x):=M(X_i(x), X_i'(x), Y_i(x)) \in \PGL_2(\Hr(x))$ and $M_i(y):=M(X_i(y), X_i'(y), Y_i(y)) \in \PGL_2(\Hr(y))$. Let $\Ga_x,$ respectively $\Ga_y$, be the group generated by the $M_i^A(x)$, respectively $M_i(y)$, $i=1,2,\dots, g$. Then $\Ga_y$ is the preimage of $\Ga_x$ via $\PGL_2(\Hr(y)) \subseteq \PGL_2(\Hr(x))$. By Lemma \ref{lem_schottky_field_extension}, $\Ga_x$ is a Schottky group if and only if $\Ga_y$ is a Schottky group, meaning $x \in \eS_{g, A}$ if and only if $y \in \eS_g$.
\end{proof}

\begin{thm}[{\cite[Remark 4.2.6]{poineau_turcheti_universal}}]
Let $\Ga$ be a rank $g$ Schottky group with a fixed basis~$\mathcal{B}$ over a complete valued field $(k, |\cdot|)$. Then, up to conjugation in $\PGL_2(k)$, there exists a unique~${x \in \eS_{g,k}}$ such that $(\Ga, \mathcal{B})=(\Ga_x, \mathcal{B}_x)$, with $\mathcal{B}_x:=(M_1(x), \dots, M_g(x))$.
\end{thm}

We note that if $k=\Hr(z)$ for some $z \in \M(\Z)$, then $\eS_{g, k}$ is the fiber over $z$ of the structural morphism $\pi: \eS_g \rightarrow \M(\Z)$. For $k=(\C, |\cdot|_{\infty}^{\varepsilon})$, one obtains $\eS_{g, (\C, |\cdot|_{\infty}^{\varepsilon})}$ by pulling back $\eS_{g, (\R, |\cdot|_{\infty}^{\varepsilon})}$ via $\Af_{(\C, |\cdot|_{\infty}^{\varepsilon})}^{3g-3, \an} \rightarrow \Af_{(\R, |\cdot|_{\infty}^{\varepsilon})}^{3g-3, \an}.$ 

\begin{prop}[{\cite[\S~2.3]{poineau_turcheti_universal}}]
Let $\varepsilon \in (0,1]$. The map $\eS_{g,(\R, |\cdot|_{\infty})} \rightarrow \eS_{g,(\R, |\cdot|_{\infty}^{\varepsilon})}, |\cdot|_x \mapsto |\cdot|_x^{\varepsilon}$, is a homeomorphism.  The same remains true when replacing $\R$ by $\C$, or when taking $(\Q_p, |\cdot|_p^{\alpha})$ with $\alpha \in (0, +\infty)$ instead. 
\end{prop}

The analytic space $\eS_g$ possesses nice properties (\cite[Theorems 4.3.4 and 5.2.1]{poineau_turcheti_universal}):

\begin{thm} \label{opennes}
The moduli space $\eS_g$ of {rank} $g$ Schottky groups is open  and path-connected in~$\Af_{\Z}^{m, \an}$, {where $m=1$ if $g=1$, and $m=3g-3$ otherwise.} 
\end{thm}

\begin{defi} \label{def_relativefigure}
 Let $V \subseteq \eS_g$ be an open. Let $\Ga_V$ denote the subgroup of $\PGL_2(\mathcal{O}(V))$ generated by $(M_1, M_2, \dots, M_g)$, meaning it is the restriction of $\Ga$ to $V.$ For a basis $(N_1, N_2, \dots, N_g),$ let $D_{N_i}, D_{N_i^{-1}}, i=1,2,\dots, g$, be closed subsets of $\bP^{\pan}_V$ such that  for all $i$, 
$$N_i(D_{N_i^{-1}}^c)=D_{N_i}^{\circ} \ \text{and} \ N_i^{-1}(D_{N_i}^c)=D_{N_i^{-1}}^{\circ},$$
where $D_{N_i}^{\circ} \subseteq D_{N_i}$ and $D_{N_i^{-1}}^{\circ} \subseteq D_{N_i^{-1}}$ are opens. 

If for any $x \in V$, the restrictions $D_{P,x}, P \in \{N_i^{\pm 1}\}_i$, to the fiber $\bP^{\pan}_{\Hr(x)}$ form a Schottky figure for $\Ga_x$ in $\bP^{\pan}_{\Hr(x)}$, then $D_P, D_P^{\circ}, P \in \{N_i^{\pm 1}\}_i$, is said to be a \emph{relative Schottky figure for $\Ga_V$ adapted to $(N_1, \dots, N_g)$}.
\end{defi}

Given an automorphism $\tau$ of the free group $F_g$ of $g$ generators, it induces an automorphism of the subgroup $\Ga$ of $\PGL_2(\mathcal{O}(\eS_g))$ generated by the matrices $M_1, M_2, \dots, M_g$. We denote by $\tau \cdot (M_1, M_2, \dots, M_g)$ the image of the basis $(M_1, M_2, \dots, M_g)$ via this automorphism.

\begin{thm}[{\cite[Corollary 4.3.3]{poineau_turcheti_universal}}] \label{Gerritzenrelative}
Let $x \in \eS_g$ be a non-Archimedean point. There exists an open neighborhood $V$ of $x$ in $\eS_g$, an automorphism $\tau$ of the free group in $g$ generators, and a Schottky figure in $\bP_V^{\pan}$ for the basis $\tau\cdot (M_1, \dots, M_g)$ of $\Ga_V$. 
\end{thm}

Under the hypothesis $\infty \not \in \La_x$, the Schottky figure from Theorem \ref{Gerritzenrelative} is constructed  explicitly by using \emph{relative twisted Ford disks} in \cite[Theorem 4.3.2]{poineau_turcheti_universal}. The following consequence will be needed later on. 

Given a complete valued field $k$, for $a \in k$ and $r>0$, let us denote by $D(a,r)$ (respectively~$D^{\circ}(a,r)$) the closed (respectively open) disk centered at $a$ and of radius $r$ in $\bP^{\pan}_k$.

\begin{prop} \label{prop_coordinate_change}
Let $x \in \eS_g$ be a non-Archimedean point. There exists 
a connected neighborhood $V$ of~$x$ and a basis $\mathcal{B}$ of $\Ga_V \subseteq PGL_2(\mathcal{O}(V))$ for which there exists a relative Schottky figure $D_{\g}, l(\g)=1$, in $\bP^{\pan}_V$ and a choice of coordinate function on it such that 
\begin{enumerate}
\item[(i)] if $\La_x \neq \bP^{1}(\Hr(x))$, then there exists $c_{\g} \in \mathcal{O}(V)$, and a continuous function $r_{\g}: V \rightarrow \R_{> 0},$ such that $D_{\g, y}=D(c_{\g}(y), r_{\g}(y)) \subseteq \bP_{\Hr(y)}^{\pan}$ for any $y \in V$ and $\g \in \Ga_V$ with $l(\g)=1$; moreover, we may assume that $D_{\g,y}\subseteq D^{\circ}(0,1) \subseteq \bP^{\pan}_{\Hr(y)}$ for any $y \in V;$
\item[(ii)] if $\La_x=\bP^1(\Hr(x))$, then we may assume that for any $y \in V$, the point $y$ is non-Archimedean and the Gauss point $\eta_y \in \bP^{\pan}_{\Hr(y)}$ is not contained in the Schottky figure. 
\end{enumerate}
\end{prop}

\begin{proof}
 Let us prove assertion (i).
Assume $\La_x \neq \bP^1(\Hr(x))$. If $\infty \in \La_x$, since $\kappa(x)$ is dense in $\Hr(x)$, there exists $a \in \kappa(x)$ such that $a \not \in \La_x$. 
By lifting $a$ to $\mathcal{O}_x$ and shrinking $V$, one may assume $a \in \mathcal{O}(V)$. Hence, a change of coordinate function in $\bP^{\pan}_V$ is possible such that $a \in \bP^1({\Hr(y)})$ becomes $\infty$ for all $y \in V$. In particular, we now have that $\infty \not \in \La_x$. By \cite[Theorem~4.3.2]{poineau_turcheti_universal}, there exists a neighborhood $V$ of~$x$, a basis $\mathcal{B}'$ of $\Ga_V$ in $\PGL_2(\mathcal{O}(V))$ and a Schottky figure $D_{\g}, l(\g)=1$, in $\bP^{\pan}_V$ adapted to it consisting of  relative twisted Ford disks. In particular, $\infty \not \in \bigcup_{l(\g)=1} D_{\g,y}$ for any $y \in V$. There exists $c'_{\g}\in \mathcal{O}(V)$ and a continuous function $r'_{\g}:V \rightarrow \R_{>0}$ such that $D_{\g, y}=D(c'_{\g}(y), r'_{\g}(y))$ for $y \in V$ (\emph{cf.} the definition of relative twisted Ford disks \cite[Definition 4.3.1]{poineau_turcheti_universal}). 

Let $b \in \kappa(x)$ be such that $\bigcup_{l(\g)=1} D_{\g, x}$ is contained in the open disk $ D^{\circ}(0,|b(x)|) \subseteq \bP^{\pan}_{\Hr(x)}$. 
As above, we may assume that $b\in  \mathcal{O}(V)$. As $|b(x)| \neq 0$, we may also assume that $|b(y)| \neq 0$ for any $y \in V$, and hence that $b \in \mathcal{O}(V)^{\times}$. Changing the coordinate function on~$\bP^{\pan}_V$ via $T \mapsto \frac{T}{b}$,  the disk $D^{\circ}(0, |b(x)|)$ is sent to $D^{\circ}(0,1) \subseteq \bP^{\pan}_{\Hr(x)}$, implying  $D_{\g, x} \subsetneq D^{\circ}(0,1)$ for all $\g \in \Ga_x \backslash \{\mathrm{id}\}$. As the Schottky figure is continuous over $V$, by shrinking $V$ if necessary, we may assume that $D_{\g, y} \subsetneq D^{\circ}(0,1) \subseteq \bP^{\pan}_{\Hr(y)}$ for all $\g \in \Ga_y \backslash \{\mathrm{id}\}$ and all $y \in V$. Finally, we note that the change of variable sends $D(c'_{\g}(y), r'_{\g}(y))$ to $D(c'_{\g}(y)/b(y), r'_{\g}(y)/|b(y)|)$ in $\bP^{\pan}_{\Hr(y)},$ so we may conclude by taking $c_{\g}:=c'_{\g}/b \in \mathcal{O}(V)$ and $r_{\g}(y):=r'_{\g}(y)/|b(y)|$ for $y \in V$.
 
Let us prove assertion (ii). Assume $\La_x=\bP^{1}(\Hr(x))$. Then $\Hr(x)$ is locally compact, meaning~$x$ has a neighborhood containing only non-Archimedean points. Following the proof of \cite[Corollary 4.3.3]{poineau_turcheti_2} (first bullet point), we make a base change $p_K: \eS_{g,\mathcal{O}_K} \rightarrow \eS_g$ over the integer ring $\mathcal{O}_K$ of a suitably chosen number field $K$ and take $x_K$ over $x$. The choice of $K$ guarantees that $\La_{x_K} \neq \bP^1(\Hr(x_K))$. Moreover, there exists a neighborhood $W$ of $x_K$ in $\eS_{g, \mathcal{O}_K}$ and a basis $\mathcal{B}'$ of $\Ga_{W} \subseteq \PGL_2(\mathcal{O}(W))$ for which there exists an adapted Schottky figure in $\bP^{\pan}_{W}$.  This Schottky figure is given by $\{D_{\g}: \g \in \Ga_W, l(\g)=1\}=\{A^{-1}D_{i}, i=1,2,\dots, 2g\}$, where $A \in \PGL_2(\mathcal{O}(W))$ and $D_i$ is a twisted Ford disk over $W$ for all $i$. 

Set $V=p_K(W)$. It is a neighborhood of $x$  and by \emph{loc.cit.}, the family $\{p_K(D_{\g}): \g \in \Ga_V, l(\g)=1\}$ forms a Schottky figure for $\Ga_V \subseteq \PGL_2(\mathcal{O}(V))$ adapted to the  basis $\mathcal{B}$ induced  by~$\mathcal{B}'$.  

For a fixed coordinate function on $\bP^{\pan}_V$, let $\eta_x \in \bP^{\pan}_{\Hr(x)}$ and $\eta_{x_K} \in \bP^{\pan}_{\Hr(x_K)}$ be the Gauss points; then $\eta_{x_K}$ projects to $\eta_x$. We recall that $F:=\bP^{\pan}_{\Hr(x_K)} \backslash (\bigcup_{l(\g)=1} D_{\g,x_K}^{\circ})$ contains a fundamental domain of the action of $\Ga_{x_K}$ (\emph{cf.} Lemma \ref{domain}). Hence, there exists $\alpha \in \Ga_{x_K}$ such that $\alpha \eta_{x_K} \in F$.  If $\alpha \eta_{x_K} \in \bigcup_{l(\g)=1}{D_{\g, x_K}}$, \emph{i.e.} if there exists $i_0 \in \{1,2,\dots, 2g\}$ such that  $A\alpha \eta_{x_K} \in D_{i_0}$, it suffices to slightly disturb the radius of the relative twisted Ford disk $D_{i_0}$ to ensure that $A \alpha \eta_{x_K} \not \in \bigcup_{i=1}^{2g} D_i.$ As this is an open condition, by shrinking $W$ if necessary, we may assume that $A \alpha \eta_{y} \not \in \bigcup_{i=1}^{2g} D_i$ for $y \in W$, where $\eta_{y} \in \bP_{\Hr(y)}^{\pan}$ is the Gauss point. Thus, $\alpha\eta_{y} \not \in \bigcup_{l(\g)=1} D_{\g}$ for $y \in W$, and hence $\alpha\eta_z \not \in \bigcup_{l(\g)=1} p_K(D_{\g})$ for $z\in V$, with $\eta_z \in \bP_{\Hr(z)}^{\pan}$ the Gauss point. As $\alpha \in \Ga_V$, we may conclude by taking the change of variable  $T \mapsto \alpha T$ on~$\bP^{\pan}_V$. 
\end{proof}

We conclude with the following result we will need later on. 

\begin{lem} \label{lem_NA_center_radius}
Let $x \in \eS_g$ be a non-Archimedean point which has a neighborhood containing only non-Archimedean points. There is a neighborhood $V$ of $x$ containing only non-Archimedean points and a Schottky figure $D_{\g}, \g \in \Ga_V \backslash\{\mathrm{id}\}$, for $\Ga_V$ in $\bP^{\pan}_V$ as well as a coordinate function satisfying the properties of Proposition \ref{prop_coordinate_change} and such that:
\begin{enumerate}
    \item For $a, b \in \Ga_V$ with $l(a)=l(b)=1$, the function $h_{ab}: V \rightarrow \mathbb{R}_{\geqslant 0},$ $y \mapsto \rho_y(\partial{D_{a,y}}, \partial{D_{b,y}})$, is continuous. Here $\rho_y$ denotes the path-distance metric in $\Hy_{\Hr(y)} \subseteq \bP^{\pan}_{\Hr(y)}$ and $D_{\g,y}:=D_{\g} \cap \bP^{\pan}_{\Hr(y)}$ for $\g \in \Ga \backslash \{\mathrm{id}\}$.
    \item The function $R_{\g}: V \rightarrow \R_{\geqslant 0}, y \mapsto \rho_y(\eta_y, \partial{D_{\g, y}}),$ is continuous for any $\g \in \Ga_V$ with~${l(\g)=1}$. Here $\eta_y \in \bP^{\pan}_{\Hr(y)}$ denotes the Gauss point. 
\end{enumerate}
\end{lem}

\begin{proof}
As in the proof of Proposition \ref{prop_coordinate_change}, there exists a Schottky figure $D_{\g}, l(\g)=1$, for~$\Ga_V$ and a coordinate function on $\bP_V^{\pan}$ constructed as follows. There exists a base change ${\pi: \eS_{g, \mathcal{O}_K} \rightarrow \eS_g}$ to the integer ring  $\mathcal{O}_K$ of a number field $K$, a neighborhood $W$ of $x_K \in \pi^{-1}(x)$, and a transformation $A \in \PGL_2(\mathcal{O}(\eS_{g, \mathcal{O}_K}))$ such that by shrinking $V$ to have $\pi(W)=V$, we may assume that $D_{\g}=p(A \cdot B_{\g})$, with $B_{\g}$ a relative twisted Ford disk in $\bP^{\pan}_W$ for $\g \in \Ga_W$, $l(\g)=1$, and $p: \bP^{\pan}_W \rightarrow \bP^{\pan}_V$ the projection. Note that if $\bP^{1}(\Hr(x)) \neq \La_x$, then one may assume~${K=\mathbb{Q}}$ (\emph{cf}. \cite[Theorem~4.3.2]{poineau_turcheti_universal}). By Proposition \ref{prop_coordinate_change}, we may moreover assume that for any~${y \in V}$, the Gauss point $\eta_y \in \bP^{\pan}_{\Hr(y)}$ is not contained in this Schottky figure.

Let $y \in V$ and $z \in W$ such that $\pi(z)=y$. As $\Hr(z)/\Hr(y)$ is a finite field extension, and the boundary points $\partial{D}_{a,y}, \partial{D}_{b,y}$ are of type 2 or 3, we obtain that  $\rho_y(\partial{D}_{a,y}, \partial{D}_{b,y})=\rho_z(\partial(A \cdot {B}_{a,z}), \partial(A \cdot{B}_{b,z}))$. Moreover, as $\rho_z$ is $\PGL_2(\Hr(z))$-invariant, $\rho_y(\partial{D}_{a,y}, \partial{D}_{b,y})=\rho_z(\partial {B}_{a,z}, \partial{B}_{b,z}).$ Seeing as $B_a$ and $B_b$ are relative twisted Ford disks, the function $\rho_z(\partial {B}_{a,z}, \partial{B}_{b,z})$ is continuous in $z \in W$, implying $h_{ab}:y \mapsto \rho_y(\partial{D}_{a,y}, \partial{D}_{b,y})$ is continuous in $V$. 

Similarly, we have that $\rho_y(\eta_y, \partial{D_{\g, y}})=\rho_z(\eta_z, A \cdot \partial{B_{\g, z}})=\rho_z(A^{-1}\eta_{z}, \partial{B_{\g,z}})$. By the proof of \cite[Corollary 4.3.3]{poineau_turcheti_universal} (first bullet point), the transformation $A$ has a representative of the form~$\begin{pmatrix}
    0 & 1\\ 1 & \theta\\
\end{pmatrix}$ with $\theta \in \mathcal{O}_K$, implying $A \eta_z=\eta_z$ \emph{for any} $z \in W$. Hence, $\rho_y(\eta_y, \partial{D_{\g, y}})=\rho_z(\eta_z, \partial{B_{\g, z}})$. As $B_{\g}$ is a relative twisted Ford disk, there exists $c_{\g} \in \mathcal{O}(W)$ and $r_{\g}: W \rightarrow \R_{\geqslant 0}$ continuous such that $c_{\g}(z)$ is a center of $B_{\g,z}$ and $r_{\g}(z)$ its radius for any $z \in W$. 
Since $\rho_z(\eta_z, \partial{B_{\g, z}})=|\log |c_{\g}(z)|| + |\log \frac{|c_{\g}(z)|}{r_{\g}(z)}|$ is continuous in $z \in W$, the function $R_{\g}$ is continuous on $V$.   
\end{proof}

\section{Uniform behavior of the Hausdorff dimension}  \label{section_unifoorm}

In Section \ref{section_modulispace}, we recalled the construction of a moduli space $\eS_g$ over $\Z$ of Schottky groups of rank $g$, where to each point $x \in \eS_g$ one associates a Schottky group $\Ga_x$ over $\Hr(x)$ with limit set $\La_x$. 

The goal of this section is to prove the continuity of the Hausdorff dimension of the limit sets $\La_x$ over $\eS_g$. In order to do this, we first develop certain tools. The structure of the section is as follows: we start by proving some uniform distortion estimates, and then obtain a uniform bound for the Hausdorff dimension near a non-Archimedean point of $\eS_g$; finally, the core idea for the proof is an approximation of McMullen over $\C$ which we extend uniformly and combine with our previous estimates. We then deduce the main theorem. 

\subsection{Distortion estimates}

We recall some known estimates that allow us to control derivatives on the Archimedean part of the moduli space $\eS_g$. They also remain true in the non-Archimedean case. We assume $g \geqslant 2.$

\begin{lem} \label{lem_distortion_bis} Let $\gamma : \overline{D} \rightarrow \C$ be induced by a Möbius transformation, where $\overline{D}$ is a closed disk in $\C$. Set {$C_{\gamma}:= \sup_{x \in \overline{D}} |\gamma''(x)|_\infty/ |\gamma'(x)|_\infty$}. 
Let $p,q \in \overline{D} $ and set $r:= |p-q|_\infty $. Then
\begin{equation}
  e^{-C_\gamma r} \leqslant \dfrac{|\gamma'(p)|_\infty}{|\gamma'(q)|_\infty} \leqslant e^{C_\gamma r}.  
\end{equation}
\end{lem}

\begin{proof} For $t \in [0,1]$, set $u(t)= (1-t) p + t q$. We note that $u(t) \in \overline{D}$ for all $t \in [0,1]$. Consider the smooth function $g \colon t \mapsto \log( |\gamma'(u(t))|)=\frac{1}{2} \log (\gamma'(u(t)) \bar{\g'}(u(t))$ defined for $t \in [0,1]$. By the triangular inequality: 
{  $$|g'(t)|_\infty = |q-p|_\infty \dfrac{|\gamma''(u(t)) \bar \gamma'(u(t)) + \gamma'(u(t)) \bar \gamma''(u(t))|_\infty}{2 {|\gamma'(u(t))|^2_\infty}} {=} |q-p|_\infty \dfrac{|\gamma''(u(t))|_\infty}{|\gamma'(u(t))|_\infty} \leqslant C_{\g}|q-p|_\infty.$$}
    We thus get $
       |\log |\gamma'(q)|_\infty - \log |\gamma'(p)|_\infty |= |g(1) - g(0)|  \leqslant \sup_{t\in [0,1]} |g'(t)| \leqslant C_{\g} r. 
    $
    This yields $|\log (|\gamma'(p)|_\infty/ |\gamma'(q)|_\infty)| \leqslant C_{\g} r$, and we conclude by taking the exponential. 
\end{proof}

{
\begin{prop}
    \label{prop_distortion_A} Let  $\varepsilon, R\in (0,1)$. Set $k:= (\C , |\cdot|_\infty^{\varepsilon})$. Let $\overline{D}$ be a closed disk in $k$ of radius $R$. Let $\gamma \colon \overline{D} \to k$ be induced by a Möbius transformation. Let $C>0$ be such that $\sup_{x \in \overline{D}} |\g''(x)|_k/|\g'(x)|_k \leqslant C$. Then, for any $p,q \in \overline{D}$,  
    \begin{equation*}
         \exp{(- 2\varepsilon C^{1/\varepsilon}  R^{1/\varepsilon})} \leqslant \dfrac{|\gamma'(p)|_k}{|\gamma'(q)|_k} \leqslant \exp{(2\varepsilon C^{1/\varepsilon}  R^{1/\varepsilon})}.
    \end{equation*}
\end{prop}}

\begin{proof}
Let $\overline{D}'$ denote the image of the disk $\overline{D}$ via the natural isomorphism $(\C, |\cdot|_{\infty}^{\varepsilon}) \rightarrow (\C, |\cdot|_{\infty}).$ Its radius is $R':=R^{1/\varepsilon}.$ For $p,q \in \overline{D}$, let $p',q'$ denote their respective images in $\overline{D}'$. Set $r:=|p-q|_k,$ and $r':=|p'-q'|_{\infty}$. We note that $r'=r^{1/\varepsilon}$ and that $r'\leqslant 2R'= 2R^{1/\varepsilon}$. Set $C_{\g} := \sup_{x\in \overline{D}'}|\gamma''(x)|_{\infty}/|\gamma'(x)|_\infty$. Then $C_{\g} \leqslant C^{1/\varepsilon}.$

By Lemma \ref{lem_distortion_bis}, 
 $e^{-C_{\g} r'} \leqslant       \dfrac{|\gamma'(p')|_\infty}{|\gamma'(q')|_\infty} \leqslant e^{C_{\g} r'}. 
  $
As $r'\leqslant 2R^{1/\varepsilon}$, by raising to the power $\varepsilon$ yields:  

\smallskip
    
{\centering
  $ \displaystyle
    \begin{aligned}
 \exp  (- 2\varepsilon C^{1/\varepsilon}  R^{1/\varepsilon} ) \leqslant \exp  (- \varepsilon C_{\g}  r') \leqslant \dfrac{|\gamma'(p)|_k}{|\gamma'(q)|_k} \leqslant \exp (\varepsilon C_{\g}  r') \leqslant \exp (2\varepsilon C^{1/\varepsilon}  R^{1/\varepsilon}).
\end{aligned}
  $ 
\par}
\end{proof}

Let $(k, |\cdot|)$ be a complete valued field. Let $\Ga \subseteq \PGL_2(k)$ be a Schottky group over $k$ endowed with a basis $\mathcal{B}= \{ \gamma_1, \ldots, \gamma_g \}$ of generators  with respect to which we have a Schottky figure in $\bP^{\pan}_k$ consisting of closed disks $D_{\beta}$, $\beta \in \Ga \backslash \{\mathrm{id}\}$. Recall from Remark \ref{rem_sphericalmetric} that, when $k$ is non-Archimedean,  $\mathbb{P}^1(k)$ is endowed with the spherical metric, which is $|\cdot|$ on the unit disk in $k$ and is invariant by the involution $T\mapsto 1/T$.  Let $r_{\beta}$ denote the radius of $D_\beta$ with respect to this metric.  

\begin{lem} \label{deriv} Assume $k$ is non-Archimedean and that $\eta \not \in \bigcup_{\g \in \Ga \backslash \{\mathrm{id}\}} D_{\g}$, where $\eta \in \bP^{\pan}_k$ denotes the Gauss point. Let $a\in \{ \gamma_1^{\pm 1}, \g_2^{\pm 1}, \ldots, \gamma_g^{\pm 1} \}$ and $ \beta  \in \Ga$ such that  $a\beta$ is a reduced word over the alphabet induced by~$\mathcal{B}$. For any $p \in D_{\beta} \cap \bP^1(k)$, one has
\begin{equation} \label{eq_derivative_NA}
|a'(p)| = \dfrac{r_{a\beta}}{r_\beta}= e^{\rho(\eta, \partial{D_{\beta}})-\rho(\eta, \partial{D}_{a\beta})}=e^{\rho(\eta, \partial D_\beta) - \rho(\eta, \partial D_{a})-\rho(\partial{D_{a^{-1}}}, \partial{D_{\beta}})}.
\end{equation} 
If $b$ is the first letter of $\beta$, then  $|a'(p)|=e^{\rho(\eta, \partial D_b) - \rho(\eta, \partial D_{a})-\rho(\partial{D_{a^{-1}}}, \partial{D_{b}})}.$
\end{lem}

\begin{proof} As we are taking the spherical metric on $\bP^1(k)$, we may assume that $p \in D_{\beta} \cap \bP^{1}(k)$ is contained in the open unit disk. 
By Lemma \ref{wedge}:
$$
|a'(p)| = \lim_{y \rightarrow p} \dfrac{|ap - ay|}{|p - y|} = \lim_{y \rightarrow p} e^{\rho(\eta, p \wedge_{\eta} y) - \rho(\eta, ap \wedge_{\eta} ay)}=\lim_{y \rightarrow p} e^{\rho(\eta, p \wedge_{\eta} y) - {\rho(a^{-1}\eta, p \wedge_{a^{-1}\eta} y)}}, 
$$
where $y \in \bP^1(k)$.  The third equality is due to the $\PGL_2(k)$-invariance of $\rho$. We note that $p \wedge_{a^{-1}\eta} y \in D_{\beta}$. 
Indeed, $a^{-1}\eta \in D_{a^{-1}}$, and $D_{a^{-1}} \cap D_{\beta}=\emptyset$ (because $a\beta$ is assumed to be reduced), so we have that $\eta, a^{-1}\eta \not \in D_{\beta}$. Since the injective path $[p,y]$ connecting $p$ and $y$ can be assumed to be contained in~$D_{\beta}$, one has $p \wedge_{\eta} y=p \wedge_{a^{-1}\eta} y \in D_{\beta}$ when $y \rightarrow p.$

 As $\rho(\eta, p \wedge_{\eta} y)=\rho(\eta, \partial{D_{\beta}})+\rho(\partial{D_{\beta}}, p \wedge_{\eta} y)$, and similarly, {$\rho(a^{-1}\eta, p \wedge_{a^{-1}\eta} y)=\rho(a^{-1}\eta, \partial{D_{\beta}})+ \rho(\partial{D_{\beta}}, p \wedge_{a^{-1}\eta} y),$} we  obtain $$\rho(\eta, p \wedge_{\eta} y) - \rho(a^{-1} \eta, p \wedge_{a^{-1} \eta} y) = \rho(\eta , \partial D_{\beta}) - \rho(a^{-1}\eta, \partial D_{\beta}).$$  The first two equalities of the statement are now due to the $\PGL_2(k)$-invariance of $\rho$, the fact that $a \partial D_\beta= \partial D_{a \beta}$, and Remark~\ref{rem_radius}. For the third equality, a similar reasoning yields $$\rho(a^{-1}\eta, \partial{D_{\beta}})=\rho(a^{-1}\eta, \partial{D_{a^{-1}}}) +\rho(\partial{D_{a^{-1}}}, \partial{D_{\beta}})=\rho(\eta, \partial{D_a})+\rho(\partial{D_{a^{-1}}}, \partial{D_{\beta}}).$$
To obtain the last part of the statement, we remark that $D_{\beta} \subseteq D_{b}$, and hence $p \in D_{b} \cap \bP^{1}(k)$, so the argument above can be applied to the letters $a$ and $b$. 
\end{proof}

\begin{lem} \label{deriv2} Assume $k$ is non-Archimedean and that $\eta \not \in \bigcup_{\g \in \Ga \backslash \{\mathrm{id}\}} D_{\g}$, where $\eta \in \bP^{\pan}_k$ denotes the Gauss point. Let $a\in \{ \gamma_1^{\pm 1}, \g_2^{\pm 1}, \ldots, \gamma_g^{\pm 1} \}$ and $ \beta  \in \Ga$ such that  $a\beta$ is a reduced word over the alphabet induced by~$\mathcal{B}$. 
There exist $(z_a, t_a) \in k^2 \backslash \{(0,0)\}$ such that for any $p \in D_{\beta} \cap \bP^1(k)$, one has
\begin{equation} \label{eq_secondderivative_NA}
|a''(p)|=\frac{|2z_a|}{|z_a c_{\beta}+t_a|} \frac{r_{a\beta}}{r_{\beta}}  < \frac{r_{a\beta}}{r_{a^{-1}}r_{\beta}}, 
\end{equation} 
where $c_{\beta} \in D_{\beta} \cap \bP^1(k)$ is an arbitrary point. 
\end{lem}

\begin{proof} By a change of coordinate function $T \mapsto 1/T$ (which preserves the Gauss point) if necessary, seeing as $\eta \not \in D_a$, we may assume that $D_a$ is contained in the closed unit disk. We note that the spherical metric is invariant by this base-change. 
Let
$\begin{pmatrix}
u & v\\ z & t \\
\end{pmatrix} \in \GL_2(k)$
be a representative of~$a$. Then for $p \in D_{\beta} \cap \bP^1(k)$, one has $\dfrac{|a''(p)|}{|a'(p)|}=\dfrac{|2z|}{|zp+t|}.$ If $z=0$, then $|a''(p)|=0$. Otherwise, let us fix a point $c_{\beta} \in \D_{\beta} \cap \bP^1(k).$ We note that $|p-c_{\beta}|\leqslant r_{\beta}$.

As $\infty \not \in D_a$ (since $D_a$ is contained in the unit disk), one has $a^{-1} \infty=-t/z \in D_{a^{-1}}.$ Seeing as $a\beta$ is reduced, $D_{\beta} \cap D_{a^{-1}}=\emptyset$, so $-t/z \not \in D_{\beta}$ and $p \not \in D_{a^{-1}}$. Consequently, $|c_{\beta}+t/z|>r_{\beta}$ and $|p+t/z|>r_{a^{-1}}$, and so $|p+t/z|=|p-c_{\beta} +c_{\beta} + t/z|=|c_{\beta}+t/z|$. Finally, $$\dfrac{|a''(p)|}{|a'(p)|}=\dfrac{|2z|}{|zp+t|}=\dfrac{|2z|}{|zc_{\beta}+t|} < \frac{1}{r_{a^{-1}}}.$$
We can now conclude by applying Lemma \ref{deriv}.
\end{proof}

\begin{rem} \label{sphericalm} 
In what follows, we will mostly be working  with a continuous family of Schottky groups and Schottky figures on a neighborhood of a non-Archimedean point $x \in \eS_g$. Due to Proposition \ref{prop_coordinate_change}, we may assume that at all Archimedean points  these Schottky figures are contained in the unit disk. Hence, for Archimedean points, it suffices to consider the usual absolute value on $\C$, whereas for non-Archimedean points $y$, one must take the spherical metric on $\bP^1(\Hr(y))$  (\emph{cf.} Remark \ref{rem_sphericalmetric}).

We note that in the Archimedean case, taking the metric induced by the absolute value instead of the spherical metric does not affect the Hausdorff dimension as the latter is a conformal invariant (see \cite[pg. 8]{mcmullen_hausdorff_3}).
\end{rem}

\subsection{Uniform distortion and uniform bounds}
Let $g \geqslant 2.$
We recall that given $x \in \eS_{g}$, we denote by $\Ga_x \subseteq \PGL_2(\Hr(x))$ the associated Schottky group and by $\La_x$ its limit set (\emph{cf.} Definition \ref{def_modspace}). Let us fix a basis of $\Ga_x$ for which we assume a Schottky figure exists. Then, given $\g \in \Ga_x \backslash \{\mathrm{id}\}$, we denote by $D_{\g, x}$ the associated closed Schottky disk in~$\bP^{\pan}_{\Hr(x)}$. Similarly, $l(\g)$ will denote the length of the reduced word corresponding to $\g$ in the alphabet determined by said basis.  We will denote by $\rho_x$ the interval length metric on~$\Hy_{\Hr(x)}$ (\emph{cf.} Section~\ref{section_hyperbolic}).

The goal of this section is to establish uniform upper bounds on the Hausdorff dimension of $\La_x$ near a non-Archimedean point $x \in \eS_g$.  We start by uniformly controlling the radii of the  Schottky disks.
\begin{lem} \label{lem_uniform_bound_locus} Let $x\in \eS_{g}$ be a non-Archimedean point. There exists a neighborhood $V\subseteq \eS_{g}$ of $x$, a basis of $\Ga_{V}$ in $\PGL_2(\mathcal{O}(V))$ for which a Schottky figure exists in $\bP^{\pan}_{V}$ satisfying the properties of Proposition \ref{prop_coordinate_change}, and constants {$R> 0, c \in (0,1)$ such that for any $y\in V$ and $\gamma \in \Ga_y \backslash \{\mathrm{id}\}$, the radius $r_{\gamma,y}$ of the disk $D_{\gamma,y}$ satisfies:
\begin{equation}
 r_{\gamma,y} \leqslant R c^{l(\gamma)-1}. 
\end{equation}}
\end{lem}
\begin{proof}
 Let us take a neighborhood $V$ of $x$, a basis of $\Ga_V$, and an adapted Schottky figure as in~Proposition~\ref{prop_coordinate_change}. 
{Set $R_y:=\max_{l(a)=1} r_{a,y}$ and $c_y:= \max_{l(a)=l(b)=1} e^{-\modu_{\mathcal{H}(y)}(D_{a^{-1},y} \backslash  D_{b,y})},$ where~$ab$ is a reduced word and $\modu_{\mathcal{H}(y)}(\cdot)$ denotes the modulus function from Definition \ref{defi_modulus}.  

If $\La_{x}=\bP^1(\Hr(x))$, we may assume $V$ contains only non-Archimedean points and, by Lemma~\ref{lem_NA_center_radius}, that both functions $R_y$ and $c_y$ are continuous. Moreover, we have that $r_{\g,y}=e^{-\rho_y(\eta_y, \partial{D}_{\g, y})}$ and if $\g=a_1\dots a_n$ is the reduced writing of $\g$, seeing as the Gauss point $\eta_{y}$ is not contained in the Schottky figure, by Lemma~\ref{lem_length_expansion} and Remark \ref{rem_coding} (ii): 
$$r_{\g,y}=e^{-\rho_y(\eta, \partial{D_{\g,y}})}=\exp(-\rho_y(\eta, \partial{D}_{a_1}) - \sum_{i=1}^{n-1} \rho_{y}(\partial{D}_{a_i^{-1}}, \partial{D}_{a_{i+1}})) \leqslant R_{y} c_{y}^{l(\g)-1},$$
for any $y \in V$.

If $\La_x \neq \bP^1(\Hr(x))$, then by \cite[Proposition~3.2.4]{poineau_turcheti_universal}, for any $y \in V$ and any $\gamma \in \Ga_y$, we have $r_{\gamma, y}\leqslant R_y c_y^{l(\gamma)-1}.$ By Proposition \ref{prop_coordinate_change} (i) and Proposition \ref{prop_modulus_extension}, the functions $R_{y}$ and $c_y$ are continuous in $y \in V$. %
   
In either case, let $c<1$ and $R>0$ be such that $c_x< c$ and $ R > R_x$. By shrinking $V$ if necessary, using the continuity of $R_y$ and $c_y$, for any $y\in V$, one has  $ r_{\gamma,y} \leqslant R c^{l(\gamma)-1} $.  
}\end{proof}

\begin{cor} \label{cor_local_bound_hausdorff} For any point $x \in \eS_{g}$ which is non-Archimedean, there exists a neighborhood~${V \subseteq \eS_{g}}$ of $x$ and a constant $A> 0$ such that for each $y \in V$, the Hausdorff dimension $\delta_y$ of the limit set $\La_y$ satisfies  $ \delta_y \leqslant A.$ 
\end{cor}

\begin{proof} 
Let $V \subseteq \eS_g$ be a neighborhood of $x$ satisfying the properties of  Lemma~\ref{lem_uniform_bound_locus}.  
 For $\sigma > 0$, let $N \in \N$ be such that for all $n \geqslant N$, all $y \in V$, and all $\gamma \in \Ga_y$ with $l(\g)=n$, one has $r_{\gamma,y} \leqslant R c^{l(\gamma)-1} < \sigma/2$. 
The disks $D_{\gamma,y }, l(\g)=n$,  form a covering of $\Lambda_y$, and $\diam(D_{\gamma,y})< \sigma$ for any $n \geqslant N$, so for any $s \geqslant 0$:
\begin{align*}
\sum_{l(\gamma)= n } (\diam D_{\gamma, y})^s & \leqslant 2^s \sum_{l(\gamma) = n} r_{\gamma,y}^s  \leqslant 2^s R^s 2g(2g-1)^{n-1} c^{(n-1)s},  
\end{align*}
seeing as there are $2g(2g-1)^{n-1}$ reduced words of length $n$.
Hence, for $n \geqslant N$ and $y\in V$, we get: 
$\mathcal{H}_{s, \sigma}(\Lambda_y) \leqslant R^s 2^{s+1}g ((2g-1)c^s)^{n-1}$ (\emph{cf.} Definition \ref{defi_Hdim}).
Let $s >  - \log(2g-1)/\log c.$ Taking the limit as $\sigma \rightarrow 0$ implies 
$n \rightarrow +\infty$, and the previous inequality yields $\lim_{\sigma \rightarrow 0} \mathcal{H}_{s, \sigma}(\Lambda_y) < +\infty.$ 
Consequently, we get that for all $y \in V$: 
$$
\delta_y  \leqslant  - \dfrac{ \log(2g-1)}{\log c}=:A. $$\end{proof}
We now obtain uniform distortion bounds which we will need later on.
\begin{lem} \label{rel_uniform_cont_second_deriv}
Let $x \in \eS_g$ be a non-Archimedean point such that $\La_x \neq \bP^{1}(\Hr(x))$. Let  $V \subseteq \eS_g$ be a neighborhood of~$x$ for which a Schottky figure $\{D_{\g}\}_{\g \in \Ga_V \backslash \{\mathrm{id}\}}$ adapted to a basis $\mathcal{B}_V$ of $\Ga_V$ exists in $\bP^{\pan}_V$ and satisfies the properties of Proposition \ref{prop_coordinate_change}. Let $a, \beta \in \Ga_V$ be length one words such that  $a \beta$ is reduced. The functions $f_{a\beta}: V \rightarrow \R_{\geqslant 0}, y \mapsto \sup_{p \in D_{\beta,y} \cap \bP^1(\Hr(y))}|a''(p)/a'(p)|_y$, and
$F_{a \beta}: V \rightarrow R_{\geqslant 0}, y \mapsto \inf_{p \in D_{\beta,y} \cap \bP^1(\Hr(y))}|a'(p)|_y$, are continuous.
\end{lem}

\begin{proof}  
Let $\begin{pmatrix}
u & v\\ z& t
\end{pmatrix}$ be a representative of $a$ in $\GL_2(\mathcal{O}(V))$. %
{Then $|a'(p)|_y=\dfrac{|ut-vz|_y}{|zp+t|_y^2}$ and
$|a''(p)/a'(p)|_y=\dfrac{|2z|_y}{|zp+t|_y}$}. Since $p \in D_{\beta,y},$ one has $a p \in D_{a\beta,y}$, and as we assume $D_{a \beta,y}$ is contained in the unit disk, $a p \neq \infty$, so $|zp+t|_y \neq 0$. 
 Hence, it suffices to show that $y \mapsto \sup_{p  \in D_{\beta , y} \cap \bP^1(\Hr(y))} \dfrac{1}{|zp+t|_y}$ and $y \mapsto \inf_{p \in D_{\beta , y} \cap \bP^1(\Hr(y))} \dfrac{1}{|zp+t|_y}$ are continuous.
By making the change of variable $\varphi: T \mapsto zT+t$ in $\bP^{\pan}_V$, we are reduced to proving the continuity of $g_{-}: y \mapsto \inf_{p \in D_{y} \cap \bP^1(\Hr(y))} |p|_y$ and $g_{+}:y \mapsto \sup_{p \in D_{y} \cap \bP^1(\Hr(y))} |p|_y$, where $D_y:=\varphi(D_{\beta,y})$. 
  Moreover, {note that} $0 \not \in D_y$  
  {(since $|zp+t|_y \neq 0$ for $p\in D_{\beta,y}$).}
  By Proposition \ref{prop_coordinate_change} (i), seeing as $D_{y}=\varphi(D_{\beta,y})$, there exists $\alpha \in \mathcal{O}(V)$ and $r: V \rightarrow \R_{\geqslant 0}$ continuous such that for $y\in V$, one has $D_y=D(\alpha(y), r(y))$.
 
We note that $g_{-}(y)$ corresponds to the distance of $0 \in \Hr(y)$ from the disk $D(\alpha(y), r(y))$. (Similarly, $g_{+}(y)$ corresponds to the largest distance between $0$ and a point of the disk $D_y$). Hence, when $\Hr(y)=(\C, |\cdot|_{\infty}^{\varepsilon(y)}),$ by making the computation over $(\C, |\cdot|_{\infty})$ and then scaling, 
one has \begin{equation} \label{mindist}
g_{\sigma}(y)=|\alpha(y)|_y \left(\sigma \left(\dfrac{r(y)}{|\alpha(y)|_y}\right)^{1/\varepsilon(y)} +1\right)^{\varepsilon(y)},
\end{equation}
{with $\sigma \in \{ +,-\}$.}
Since $0 \not \in D(\alpha(y), r(y)),$ we have $|\alpha(y)|_y>r(y).$

When $y$ is Archimedean such that $\Hr(y)=(\R, |\cdot|_{\infty}^{\varepsilon(y)})$, the disk $D(\alpha(y), r(y))$ induces uniquely a disk in $(\C, |\cdot|_{\infty}^{\varepsilon})$ of center $\alpha(y)$ in the real axis, and of radius $r(y).$ Moreover, this disk does not contain the origin. Thus, formula \eqref{mindist} is still applicable. Consequently, when intersecting~$V$ with the Archimedean part of $\eS_g$,  {the functions $g_{\pm}$ are}  continuous. 

When $y$ is non-Archimedean, then as $0 \not \in D(\alpha(y),r(y)),$ all elements of this disk have same norm in $\Hr(y)$. Otherwise, if $l_1, l_2 \in D(\alpha(y), r(y))$ are such that $|l_1|_y < |l_2|_y$, one obtains $|l_2|_y=|l_1-l_2|_y \leqslant r(y)$, and as $D(\alpha(y), r(y))=D(l_2, r(y))$, we get $0 \in D_y$, contradiction. Thus, $g_{\pm}(y)=|\alpha(y)|_y$. To conclude, we note that for Archimedean points $y \in V$, when $y\rightarrow x_0$ with~${x_0 \in V}$ non-Archimedean, then $\varepsilon(y) \rightarrow 0$, and by \eqref{mindist} and the continuity of $|\alpha(w)|_w$ on~$V$, we obtain that $g_{\pm}(y) \rightarrow g_{\pm}(x_0)$.
\end{proof}

\begin{cor}\label{cor_distortion_ball_bis}
Let $x \in \eS_g$ be a non-Archimedean point. There exist $N \geqslant 1$,  a neighborhood~${V \subseteq \eS_g}$ of $x$, a Schottky figure $\{D_{\g}\}_{\g \in \Ga_V \backslash \{\mathrm{id}\}}$ in  $\bP^{\pan}_{V}$  adapted to a basis~$\mathcal{B}_V$ of $\Ga_V$ and satisfying the properties of Proposition \ref{prop_coordinate_change} such that: whenever $a, \beta \in \Ga_V$ with $l(a)=1,$ $l(\beta) > N$, and $a\beta$ a reduced word, for $S_{a,\beta,y}:= \inf_{p \in  D_{\beta,y}} |a'(p)|_y, U_{a, \beta,y} := \sup_{p\in D_{\beta,y}} |a'(p)|_y $ when $y \in V$, one has:
\begin{equation*}
  \exp\left(-\frac{1}{4} r_{\beta,y}\right) \leqslant \dfrac{S_{a,\beta,y}}{U_{a, \beta,y}} \leqslant \exp\left(\frac{1}{4} r_{\beta,y}\right).
\end{equation*}
\end{cor}

\begin{proof}
We may assume that $V$ satisfies the properties of Lemma \ref{lem_uniform_bound_locus}. 
If~${y \in V}$ is non-Archimedean, then the conclusion of Proposition \ref{prop_coordinate_change} guarantees that the Gauss point is not in $\bigcup_{\g \in \Ga_y \setminus \{\mathrm{id}\}} D_{\g,y}$, hence by Lemma \ref{deriv},  $|a'(p)|_y$ is independent of $p$, and so the quotient $S_{a, \beta,y}/U_{a, \beta,y}$ is always~$1$, meaning the inequalities hold. Thus, if $\La_x=\bP^{1}(\Hr(x))$, then it suffices to take a purely non-Archimedean neighborhood of $x$ for the statement to hold.

Assume $\La_x \neq \bP^{1}(\Hr(x))$. 
Let $a, b\in \Ga_V$ be words of length $1$ such that $ab$ is reduced.
{By Lemma \ref{rel_uniform_cont_second_deriv}, the function $f_{ab}: V \rightarrow \R_{\geqslant 0}, y \mapsto \sup_{p \in D_{b,y} \cap \bP^1(\Hr(y))} |a''(p)/a'(p)|_y$, is continuous.} By Lemmas~\ref{deriv} and~\ref{deriv2},  $f_{ab}(x) < \dfrac{1}{r_{a^{-1},x} } \leqslant \dfrac{1}{\max_{l(w)=1} r_{w,x}}=:K$.
Let us consider $f: V \rightarrow \R_{\geqslant 0},$
$y \mapsto \max_{l(a)=l(b)=1} f_{ab}(y)$. As there are only finitely many words of length $1$, we may assume $f$ is 
continuous on $V$  by Lemma~\ref{rel_uniform_cont_second_deriv}.  
We note that $f(x) < K$.
 As~$f$ is continuous, by shrinking~$V$, we may assume that for all $y \in V$, one has $f(y)< K$.

For any $\beta \in \Ga_V$ such that $l(\beta) > 1$,  let $\beta=b \alpha$ be reduced with $l(b)=1$. We remark that for~${y \in V}$ and $a \in \Ga_V$ with $l(a)=1$ and $a\beta$ reduced, the inequalities $f_{a \beta}(y) \leqslant f_{a b}(y) \leqslant {f(y)}<K$ hold seeing as $D_{\beta,y} \subsetneq D_{b,y}$. 
By Lemma \ref{lem_uniform_bound_locus}, by shrinking $V$ if necessary, there exists $N \geqslant 1$ such that if $\beta \in \Ga_V$ satisfies $l(\beta)>N$, then $K \sqrt{r_{\beta,y}}< \frac{1}{2}$ for all $y \in V$.

If $y \in V$ is Archimedean such that $\Hr(y)=(\C, |\cdot|_{\infty}^{\varepsilon(y)}),$ then by Proposition \ref{prop_distortion_A}, we obtain 
\begin{equation} \label{caseC} \exp(-2\varepsilon(y) (K \sqrt{r_{\beta,y}})^{1/\varepsilon(y)}{r_{\beta,y}^{1/2\varepsilon(y)}} ) \leqslant  \dfrac{S_{a,\beta,y}}{U_{a, \beta,y}} \leqslant \exp(2\varepsilon(y) (K \sqrt{r_{\beta,y}})^{1/\varepsilon(y)}{r_{\beta,y}^{1/2\varepsilon(y)}} ). \end{equation}
If $\Hr(y)=(\R, |\cdot|_{\infty}^{\varepsilon(y)}),$ each disk $D_{\g, y}$ induces uniquely a disk in $\C$ centered on the real axis, with same radius and center. Hence, the inequalities from \eqref{caseC} apply to this case as well.
By shrinking $V$ if necessary, we may assume that $\varepsilon(y)<\frac{1}{2}$ for all $y \in V$ Archimedean. Seeing as~${K \sqrt{r_{\beta,y}}<1}$, one has $(K \sqrt{r_{\beta,y}})^{1/\varepsilon(y)} \leqslant (K \sqrt{r_{\beta,y}})^2<1/4$ and $r_{\beta,y}^{1/2\varepsilon(y)} \leqslant r_{\beta,y}$, so we may conclude by applying these inequalities to \eqref{caseC}. 
\end{proof}

\subsection{Approximation of the Hausdorff dimension} \label{subsect_approxi}

Let $(k, |\cdot|)$ be a complete valued field. Let $\Ga$ be a Schottky group over $k$ endowed with a basis~$\mathcal{B}$ with respect to which we have a Schottky figure. Let $g \geqslant 2$ denote the rank of $\Ga$, {and~$\La_{\Ga}$ its limit set. 
Given $\g \in \Ga \backslash \{\mathrm{id}\}$, we denote by $D_{\g}$ the associated closed Schottky disk. By Proposition~\ref{prop_coordinate_change}, when $k$ is non-Archimedean, we may assume that the choice of coordinate on $\bP^{\pan}_k$ is such that the Gauss point $\eta$ is not contained in $\bigcup_{\g \in \Ga \backslash \{\mathrm{id}\}} D_{\g}.$}   

We recall and extend to $k$ the definition of certain real matrices by McMullen in \cite{mcmullen_hausdorff_3}, used to construct the \emph{eigenvalue algorithm} for obtaining an approximation of the Hausdorff dimension of~$\Ga$ in the case $k=\C$.

Let $m \in \N$. Let $\alpha = a_1 \cdots a_m$ and $\beta = b_1 \cdots b_m$ be two reduced words in $\Ga$ of length $m \in \N$ over the alphabet induced by $\mathcal{B}.$ We pick any point $p_{\alpha \beta}$ in $D_\alpha \cap a_1 D_{\beta} \cap \bP^1(k)$ provided this set is not empty. 
For $s\geqslant 0$, let
 \begin{equation} \label{matrixT}
 t_{\alpha \beta, m}^{(s)} := \left  \lbrace \begin{array}{ll}
 0 & \text{if \ } D_\alpha \cap a_1D_{\beta} \cap \La_{\Ga} = \emptyset, \\
 
|{a_1^{-1}}'(p_{\alpha \beta})|^{-s} &\text{otherwise. \ }
\end{array}  \right .
 \end{equation}
 Let $T_{m}^{(s)}$ be the matrix with entries $t_{\alpha \beta, m}^{(s)}$ for all such words $\alpha$ and $\beta$ of length~$m$. We note that~$T_m^{(s)}$ is a real matrix of size $2g(2g-1)^{m-1} \times 2g(2g-1)^{m-1}$, where the entries are indexed by the reduced words of length~$m$ in $\Ga$.  We also note that the coefficients of $T_{m}^{(s)}$ depend on the choice of the points $p_{\alpha \beta}$, though we do not emphasize that in this section. 
To simplify the notation, we write $T_{m}$ for the matrix $T_{m}^{(1)}$.

We also define the real number
\begin{equation} \label{mcrit}
\delta_{m} := \inf \{ s> 0 \  | \  \lambda(T_{m}^{(s)})\geqslant 1 \},
\end{equation}
where $\lambda (T_{m}^{(s)})$ denotes the spectral radius of the matrix $T_{m}^{(s)}$. McMullen shows in \cite[Theorem~2.2]{mcmullen_hausdorff_3} that if $k$ is $\R$  or $\C$, then 
 $\lim_{m \rightarrow +\infty} \delta_m$ is the Hausdorff dimension of the limit set~$\La_{\Ga}$ of~$\Ga$.

\begin{rem} \label{rem_case_1}
We note that when $m=1$, if $\alpha=\beta^{-1}$, then $D_{\alpha} \cap \alpha D_{\alpha}^{-1}=D_{\alpha} \cap (D_{\alpha}^{\circ})^c$ (\emph{cf.} Remark~\ref{rem_coding}~(ii)), and hence $D_{\alpha} \cap \alpha D_{\alpha}^{-1} \cap \La_{\Ga}=\emptyset$. If $\alpha \neq \beta^{-1}$, then $\alpha D_{\beta}=D_{\alpha \beta} \subsetneq D_{\alpha},$ so $D_{\alpha} \cap D_{\alpha\beta} \cap \La_{\Ga} \neq \emptyset.$ Thus, $D_{\alpha} \cap \alpha D_{\beta} \cap \La_{\Ga}=\emptyset$ if and only if $\alpha\beta$ is reduced. 
\end{rem}

Let us determine exactly when the entries of the matrix $T_{m}^{(s)}$ are zero.

\begin{lem} \label{zeroterms} Let $\alpha=a_1a_2 \cdots a_m$ and $\beta=b_1 b_2 \cdots b_m$ be two reduced words in $\Ga$ of length~${m > 1}$. Then  $D_{\alpha} \cap a_1D_{\beta} \neq \emptyset$ if and only if  $D_{\alpha} \cap a_1D_{\beta} \cap \La_{\Ga}\neq \emptyset$ if and only if $b_i=a_{i+1}$ for all $i=1,2,\dots, m-1$.  
\end{lem}

\begin{proof} 
If $a_{i+1}=b_i$ for all $i=1,2,\dots, m-1$, then $\alpha=a_1a_2\cdots a_m$ and $b=a_2a_3\cdots a_mb_m$, so clearly $D_{\alpha}\cap a_1D_{\beta} \neq \emptyset$.  On the other hand, if $a_1D_{\beta} \cap D_{\alpha} \neq \emptyset$, then $D_{\beta} \cap D_{a_2a_3\cdots a_m} \neq \emptyset.$ As $D_{\beta} \subseteq D_{b_1}$ and $D_{a_2a_3\cdots a_m} \subseteq D_{a_2}$, if the intersection is non-empty then $a_2=b_1$. Consequently, $D_{b_2b_3 \cdots b_m} \cap D_{a_3a_4\cdots a_m} {\neq \emptyset}$.
Continuing with the same reasoning, we obtain that $b_i=a_{i+1}$ for all $i=1,2,\dots, m-1$. Thus $D_{\alpha} \cap a_1D_{\beta} \neq \emptyset$ if and only if  $b_i=a_{i+1}$ for all $i=1,2,\dots, m-1$.

If $D_{\alpha} \cap a_1 D_{\beta} \neq \emptyset$, then $a_1D_{\beta}=D_{a_1 a_2\cdots a_m b_m} \subseteq D_{\alpha},$ and so $D_{\alpha} \cap a_1 D_{\beta} \cap \La_{\Ga} \neq \emptyset$. The other direction is immediate.
\end{proof}

By Lemma \ref{zeroterms}, for any $m \in \N_{\geqslant 2}$, any $s \geqslant 0,$ and any $\alpha, \beta \in \Gamma$ reduced words of length $m$, 
\begin{equation} \label{nonzeroterms}
t_{\alpha \beta,m}^{(s)} \neq 0 \iff \alpha=a_1a_2\cdots a_m \ \text{and} \ \beta=a_2a_3\cdots a_m b,
\end{equation}
where these are the reduced forms of the transformations $\alpha, \beta$.

\begin{lem} \label{lem_mcmullen_matrix_coef_NA} If the field $k$ is non-Archimedean, then for any $m \geqslant 1$, any reduced words $\alpha = a_1 \cdots a_m$, $\beta = b_1 \cdots b_m$ of length $m$, and any $p_{\alpha \beta} \in D_{\alpha} \cap a_1D_{\beta} \cap \bP^{1}(k),$ one has: 
\begin{equation}
t_{\alpha \beta, m}^{(s)} = \left  \lbrace \begin{array}{ll}
 0 & \text{if \ } D_\alpha \cap a_1D_{\beta} \cap \La_{\Ga}= \emptyset, \\
 e^{s(\rho(\eta, \partial{D}_{\beta})-\rho(\eta, \partial{D}_{a_1\beta}))}=e^{\rho(\eta, \partial D_{b_1}) - \rho(\eta, \partial D_{a_1})-\rho(\partial{D_{a_1^{-1}}}, \partial{D_{b_1}})}  &\text{otherwise. \ }
\end{array}  \right .
\end{equation}
In particular, $t_{\alpha\beta, m}^{(s)}$ does not depend on the choice of $p_{\alpha\beta}$, but only on the points $\partial{D}_{a_1}, \partial{D_{b_1}}$. 
\end{lem}

\begin{proof}
If $D_\alpha \cap a_1D_{\beta} \cap \La_{\Ga} \neq \emptyset$, then for $p_{\alpha\beta} \in D_\alpha  \cap {a_1D_{\beta}}\cap \bP^1(k)$, and $y:= a_1^{-1} p_{\alpha\beta}$, one has
\begin{equation*}
|(a_1^{-1})' (p_{\alpha \beta})|^{-1} = \dfrac{1}{|(a_1^{-1})' (a_1(y))|} = |a_1'(y)|.
\end{equation*}
Since $y \in D_\beta$, we may conclude by Lemma \ref{deriv}. 
\end{proof}

\begin{lem} \label{calc} For $N \geqslant 1,$  
let $\alpha, \alpha_1, \dots, \alpha_{N-1}, \beta \in \Ga$ be reduced words of length $m \geqslant 1$. Then, $$t_{\alpha \alpha_1, m} t_{\alpha_1 \alpha_2,m}\cdots t_{\alpha_{N-1} \beta, m} \neq 0$$ if and only if there exist letters $a_1, a_2, \dots, a_{N+m}$ in $\Ga$ such that $\alpha_i=a_{i+1}a_{i+2} \cdots a_{m+i}$ for $i=0,1,\dots, N$, where $\alpha_0:=\alpha$, $\alpha_N:=\beta$ {and $a_1 \cdots a_{N+m}$ is a reduced word}.
In that case, if~$k$ is non-Archimedean, then: \begin{align*}
\log \prod_{i=0}^{N-1} t_{\alpha_{i} \alpha_{i+1}, m} &=\rho(\eta, \partial{D}_{a_{N+1}}) - \rho(\eta, \partial{D_{a_1}})- \sum_{i=1}^N \rho(\partial{D_{a_{i}^{-1}}}, \partial{D_{a_{i+1}}})\\
&= \rho(\eta, a_{N+1}\eta)-\rho(\eta, a_1\cdots a_{N+1} \eta).
\end{align*}

\end{lem}

\begin{proof}
The first part of the statement is immediate from \eqref{nonzeroterms} when $m \geqslant 2$, and Remark~\ref{rem_case_1} when $m=1$ (which is where the assumption that $a_1\cdots a_{N+m}$ is reduced is necessary).

Now assume $k$ is non-Archimedean. The second equality follows directly from Lemma~\ref{lem_calc_2}. 
For the first equality, set $S:=\log \prod_{i=0}^{N-1} t_{\alpha_i \alpha_{i+1},m}.$ By Lemma \ref{lem_mcmullen_matrix_coef_NA}, $S=\sum_{i=1}^N (\rho(\eta, \partial{D_{\alpha_i}})-\rho(\eta, \partial{D_{a_i\alpha_i}})),$ so  decomposing this sum:
\begin{align*} 
S= & \sum_{i=1}^{N}( \rho(\eta, \partial D_{a_{i+1}}) + \rho(\partial D_{a_{i+1}},\partial D_{\alpha_{i}})) - \sum_{i=1}^N \left (\rho(\eta, \partial D_{a_i} ) + \rho (\partial D_{a_i},\partial D_{a_i \alpha_{i}}) \right )   \\
= &  \rho(\eta, \partial{D_{a_{N+1}}}) -\rho(\eta, \partial{D_{a_1}}) + \sum_{i=1}^N (\rho(\partial D_{a_{i+1}}, \partial D_{\alpha_{i}})- \rho(\partial D_{a_i}, \partial D_{a_{i} \alpha_{i}})). 
\end{align*}
By the $\PGL_2(k)$-invariance of $\rho$, we have $ \rho(\partial D_{a_i}, \partial D_{a_{i} \alpha_{i}})=\rho(\partial{D_{a_i^{-1}}}, \partial{D_{\alpha_{i}}})$ (\emph{cf.} Remark~\ref{rem_coding}). As $D_{\alpha_{i}} \subseteq D_{a_{i+1}}$ and $D_{a_{i+1}} \cap D_{a_i^{-1}}=\emptyset,$ the boundary point $\partial{D}_{{a_{i+1}}}$ is contained in the injective path connecting $\partial{D_{a_{i}^{-1}}}$ and $\partial{D_{\alpha_i}}.$ Consequently, 
$\rho(\partial D_{a_{i+1}}, \partial D_{\alpha_{i}})- \rho(\partial{D_{a_i^{-1}}}, \partial{D_{\alpha_{i}}})=-\rho(\partial{D_{a_{i+1}}}, \partial{D}_{a_i^{-1}})$, as required for the first equality. 
\end{proof}

\begin{cor} \label{postcalcul}
Let $N \in \N_{\geqslant 2}$ and assume $m \geqslant N$. Then the entry $(\alpha, \beta)$ of $T_m^{N}$  equals $\prod_{i=0}^{N-1} t_{\alpha_i \alpha_{i+1}, m},$
where $\alpha_i:=a_{i+1}\cdots a_{m+i}$, $\alpha_0:=\alpha$, and $\alpha_{N}:=\beta$. {Consequently, the $(\alpha, \beta)$-entry of $(T_m^{(s)})^N$ is $\prod_{i=0}^{N-1} t_{\alpha_i \alpha_{i+1}, m}^{(s)}.$}
\end{cor}

\begin{proof}
By Lemma \ref{calc}, it suffices to show that the letters $a_1, a_2, \dots, a_{N+m}$ are all uniquely determined by $\alpha$ and $\beta$. Here $\alpha=a_1a_2\cdots a_m$ and $\beta=a_{N+1}\cdots a_{N+m}$, meaning the letters $a_1, \dots, a_N$ are all uniquely determined by $\alpha$ and the rest by $\beta$ seeing as $m \geqslant  N$.  
\end{proof}

\begin{lem} \label{lem_perron}  For $m \geqslant 1$ and $s\geqslant 0$, the matrix $T_m^{(s)}$ is a Perron matrix.
\end{lem}

\begin{proof}
Take any two elements $\alpha,\beta \in \Gamma$ of length $m$. Let $\alpha = a_1 \cdots a_m$ and $\beta = b_1 \cdots b_m$ be their reduced forms over the alphabet induced by the basis~$\mathcal{B}$. Let $b \in \mathcal{B} \cup \mathcal{B}^{-1}$ be such that~${b \neq a_m^{-1}}$ and $b \neq b_1^{-1}$. Since $g \geqslant 2$, there are at least $4$ length one words, so such an element $b$ exists. 

When $m=1$, we have $t_{\alpha b,m}^{(s)} t_{b \beta,m}^{(s)} \neq 0$ by Remark \ref{rem_case_1}. 
If $m>1$, set ${\alpha_1 = a_2 a_3 \cdots a_m b}$, ${\alpha_2 = a_3 a_4 \cdots a_m b b_1}$, $\dots,$ $\alpha_m =  b b_1 \cdots b_{m-1}$, $\alpha_{m+1} =  b_1 \cdots b_{m} = \beta$. 
By \eqref{nonzeroterms}, $t_{\alpha \alpha_1,m}^{(s)}, t_{\alpha_i \alpha_{i+1}, m}^{(s)}$ are all non-zero for $i\leqslant m$.

In both cases, as the entries of $T_m^{(s)}$ are all non-negative, the $(\alpha, \beta)$ entry of $(T_m^{(s)})^{m+1}$ is strictly positive. This shows that the $(m+1)$-th power of $T_m^{(s)}$ has all strictly positive entries, as required.  
\end{proof}

We recall the real number $\delta_m$ associated to the matrix $T_m^{(s)}$ defined in \eqref{mcrit}.

\begin{prop} \label{prop_mc_mullen_NA} Assume $k$ is non-Archimedean. Then for any ${m \in \N_{\geqslant 1}}$, the number $\delta_{m}$ is the Hausdorff dimension of the limit set of the Schottky group $\Ga$. 
\end{prop}

\begin{proof}
Let $N \in \N$. Let us consider the sum $U^{t}(T_m^{(s)})^NU$ of all the entries of the matrix $(T_{m}^{(s)})^N$, where $U$ is the vector with all entries $1$. By Lemma~\ref{calc},
 $$U^{t}(T_m^{(s)})^NU=\sum_{a_1, a_2, \dots, a_{N+m}} e^{s(\rho(\eta, a_{N+1}\eta)-\rho(\eta, a_1\cdots a_{N+1}\eta))},$$
where the sum is taken over length one words $a_i$ such that $a_1a_2\cdots a_{N+m}$ is a reduced word in~$\Ga$. As $a_{N+2}, \dots, a_{N+m}$ don't appear in the general term of the sum, we obtain that $$U^{t}(T_m^{(s)})^NU=(2g-1)^{m-1} \sum_{a_1, a_2, \dots, a_{N+1}} e^{s(\rho(\eta, a_{N+1}\eta)-\rho(\eta, a_1\cdots a_{N+1}\eta))},$$
where the sum is now taken over all $a_1, \dots, a_{N+1}$ such that $a_1\cdots a_{N+1}$ is a reduced word. 

Let $C_1, C_2>0$ be such that $\max_{a}  e^{\rho(\eta, a\eta)} \leqslant C_1$ and $\min_{a}  e^{\rho(\eta, a\eta)}\geqslant C_2$, where~$a$ is a length one word in $\Ga$. Then, 
$${C_2^s}{(2g-1)^{m-1}} \sum_{a_1, \dots, a_{N+1}} e^{-s\rho(\eta, a_1\cdots a_{N+1}\eta)} \leqslant U^{t}(T_m^{(s)})^NU \leqslant {C_1^s}{(2g-1)^{m-1}} \sum_{a_1, \dots, a_{N+1}} e^{-s\rho(\eta, a_1\cdots a_{N+1}\eta)},$$ 
which is equivalent to 
$${C_2^s}{(2g-1)^{m-1}} \sum_{l(\g)=N+1} e^{-s\rho(\eta, \g \eta)} \leqslant  U^{t}(T_m^{(s)})^NU \leqslant {C_1^s}{(2g-1)^{m-1}}  \sum_{l(\g)=N+1} e^{-s\rho(\eta, \g \eta)}.$$
Consequently,  $(U^t (T_{m}^{(s)})^N U)^{1/N}$ is equivalent to $(\sum_{l(\gamma)=N+1} e^{-s \rho(\eta, \gamma \eta)})^{1/N}$. 
Using the spectral radius formula, this shows that the critical exponent of the Poincaré series $\mathcal{P}_{\eta}(s)$ (Definition~\ref{poincareseries}) is precisely~$\delta_m$, and we can conclude by Corollary \ref{cor_critical_exponent_hausdorff}.   
\end{proof}

We end this section with an application of the Perron-Frobenius theorem and of the implicit function theorem, which gives information on the behavior of the ``critical exponent'' $\delta_m$ of $T_m$ (\emph{cf.} \eqref{mcrit}) with respect to its entries. The proof follows closely the estimates for the pressure of Markov operators in the thermodynamic formalism (see \cite[\S~20.2]{katok}).

For a matrix $P=(p_{ij}) \in M_{n \times n}(\R)$, let $\delta_P:=\inf\{s>0 \ | \ \lambda(P^{(s)}) \geqslant 1\}$, where $\lambda(\cdot)$ denotes the spectral radius and $P^{(s)}:=(p_{ij}^{s})$ for $s \in \R_{\geqslant 0}$.

\begin{prop} \label{prop_conseq_perron}
Let $\xi>1$. Let $E \subseteq M_{n \times n}(\R)$ consist of matrices with entries in $(0,\xi^{-1})$. Let $P \in E$ be such that $\delta_P>0$.   
There exists a neighborhood $U \subseteq E$ of $P$ such that the function $\delta: U \rightarrow \R_{\geqslant 0}, M \mapsto \delta_M$, is continuous.
\end{prop}

\begin{proof}
Consider the analytic map $\phi : E \times (\delta_P/2, +\infty) \rightarrow E, T \mapsto T^{(s)}$.  Let also $g : E \times (\delta_P/2 , +\infty) \to \R$, 
\begin{equation*}
    g(T,s):= \lambda(\phi(T,s))=\lambda(T^{(s)}).
\end{equation*}
Since each $T \in E$ is a Perron matrix, the spectral radius $\lambda(T^{(s)})$ has multiplicity one. By the proof of Theorem 5.16, pg. 119 of \cite{kato}, $g(T,s)$ is analytic in $(T,s)$.

Since the roots (with their multiplicities) of the characteristic polynomial of $T^{(s)}$ vary continuously in $(T,s)$ (see \emph{e.g.} \cite[Theorem 5.14]{kato}), we can choose an open neighborhood $W \subseteq E \times  (\delta_P/2, +\infty)$  of $(P,\delta_P)$ such that for any $(T,s)\in W$, the eigenvalues $\mu$ of $T^{(s)}$ distinct from the spectral radius $g(T,s)$ satisfy $|\mu|/ g(T,s) < \alpha < 1 $.

 The function $g$ is continuous in $W$ and by the spectral radius formula, for all $(T,s) \in W$, $g(T,s)=  \lim_{n\rightarrow +\infty} (U^t (T^{(s)})^n U)^{1/n},$
 where $U^t:= (1,\ldots , 1)$ is the unit vector. 
Set $f_n(T,s):=U^t (T^{(s)})^n U$ for $(T,s) \in W$. We show that $\log (f_n(T,s))/n$ converges uniformly toward $\log g(T,s)$ on $W$. 
 Using Perron-Frobenius, consider the projection $\pi_{T,s}$ onto the subspace associated with the largest eigenvalue of $T^{(s)}$. It is continuous (even analytic) in $(T,s)$ by \cite[Formula~1.17, p.68]{kato}, since the multiplicity of the largest eigenvalue is $1$. By shrinking $W$, we can thus assume that $(T,s) \mapsto \log U^t \pi_{T,s}(U)$ is bounded on $W$.  
 Since the other eigenvalues $\mu$ of $T^{(s)}$ satisfy  $|\mu| < \alpha g(T,s)$ for $(T,s) \in W$, we have that for all $n$:
 \begin{align*}
    \dfrac{1}{n} \log f_{n}(T,s) &= \dfrac{1}{n} \log \left(  g(T,s)^n U^t \pi_{T,s}(U) + O(\alpha^n g(T,s)^n) \right ) \\
     & = \log(g(T,s)) + \dfrac{1}{n} \log (U^t \pi_{T,s}(U) + O(\alpha^n)),
 \end{align*}
 where the terms in $O(\cdot)$ depend on the projection onto the eigenspaces not corresponding to~$g(T,s)$. 
 This shows that the convergence is uniform by our boundedness assumption on $W$. 
 
 By the uniform convergence  $\log \frac{f_n(T,s)}{n} \rightarrow \log g(T,s)$ on $W$, we obtain that  $\dfrac{\partial(\log\frac{f_n(T,s)}{n})}{\partial{s}}$ converges to $\dfrac{\partial{\log{g(T,s)}}}{\partial{s}}$ when $n \rightarrow +\infty$. Let $T=(t_{ij})_{i,j}$ so that $T^{(s)}=(t_{ij}^s)_{i,j}.$ As $t_{ij} \in (0, \xi^{-1}),$ we have
\begin{align*}
\dfrac{\partial_s{f_n(T,s)}}{nf_n(T,s)}=& \dfrac{\sum_{i_0, i_1, \dots, i_{n}} \log(\prod_{j=0}^{n-1} t_{i_ji_{j+1}})(\prod_{j=0}^{n-1} t_{i_ji_{j+1}})^s}{nf_n(T,s)} < \dfrac{f_n(T,s) \log \xi^{-n}} {nf_n(T,s) }=-\log{\xi}.
\end{align*}
By passing to the limit, we obtain that $\dfrac{\partial_s g(T,s)}{g(T,s)} = {\partial_s (\log  g(T,s))} \leqslant -\log \xi < 0$, so in particular $\partial_s g(T,s) \neq 0$ for any $(T,s) \in W$.

As all matrices in $E$ have positive entries, the spectral radius is a continuous strictly increasing function on their entries. Hence $\delta_T$ is the unique solution to $g(T,s)=1$ for any $T \in E$. 
By the implicit function theorem, there exists a neighborhood $U \subseteq E$ of $P$ and a unique continuous function $u: E \rightarrow \R$ such that $g(T, u(T))=1.$ Consequently, $u(T)=\delta_T$, and the function $\delta$ is continuous on $U$.
\end{proof}

\subsection{Uniform variation of approximations}

Let us now consider the matrices constructed in Section \ref{subsect_approxi} in families. For $g \geqslant 2,$ given $x \in \eS_g,$ $m \in \N,$ and $s\in \R_{\geqslant 0}$, we denote by $T_{m,x}^{(s)}$ the matrix constructed as in \eqref{matrixT} and associated to the Schottky group $\Ga_x$ over $\Hr(x)$ equipped with a Schottky figure. (We will usually make the basis precise from the very beginning, and even though the matrix $T_{m,x}^{(s)}$ depends on it, we omit this from the notation.) Similarly, let $(t_{\alpha \beta, m, x}^{(s)})_{\alpha, \beta}$ denote the entries of $T_{m,x}^{(s)}$, and $\delta_{m,x}$  the associated real number defined as in~\eqref{mcrit}. 
Recall also that $T_{m,x}:=T_{m,x}^{(1)}$.

The matrices $T_{m,x}^{(s)}$ depend on the points $(p_{\alpha\beta})_{\alpha, \beta}$ in the fixed Schottky figure. We now make a choice of these points so that certain properties are satisfied, and  which we will respect throughout the rest of the section.

\begin{lem} \label{lem_choice_p}
Let $x \in \eS_g$ {be a non-Archimedean point}. There exists a neighborhood $V \subseteq \eS_g$ of~$x$ and a basis $\mathcal{B}$ of $\Ga_V$  satisfying the conditions of Proposition \ref{prop_coordinate_change} such that for any $m \geqslant 1$: if $\alpha=a_1a_2 \cdots a_m, \beta=b_1b_2\cdots b_m \in \Ga_V$ are reduced words of length $m$, then for $y \in V$, there exists 
$p_{\alpha\beta,y} \in D_{\alpha,y} \cap a_1 D_{\beta,y} \cap \bP^{1}(\Hr(y))$ whenever $D_{\alpha,y} \cap a_1 D_{\beta,y} \cap \La_y \neq \emptyset$,
such that the entries of the associated $T_{m,y}$ vary continuously in $y$.
\end{lem}

\begin{proof}
Let $V$ be as in Proposition \ref{prop_coordinate_change}. By \eqref{nonzeroterms} and Remark~\ref{rem_case_1}, for any $y \in V$ and any choice of $p_{\alpha\beta,y}$, we have that  $t_{\alpha\beta, m, x}=0$ if and only if $t_{\alpha\beta,m,y}=0$ for all~${y \in V.}$ Thus, we may assume that the $(\alpha,\beta)$-entry of $T_{m,y}$ is non-zero for any $y \in V$.

If $\La_x=\bP^{1}(\Hr(x))$, we may assume that $V$ contains only non-Archimedean points. Let~${y \in V}$.
By {Lemma \ref{lem_mcmullen_matrix_coef_NA}}, the entry $t_{\alpha\beta, m,y}$ %
of $T_{m,y}$ does not depend on the choice of $p_{\alpha\beta,y},$ %
but only on $\rho_y(\eta_y, \partial{D_{a_1,y}})$ and $\rho_y(\eta_y, \partial{D_{b_1,y}}).$
By Lemma \ref{lem_NA_center_radius}, up to shrinking~$V$, we may assume $t_{\alpha\beta, m,y}$ is continuous in $y$.

Assume $\La_x \neq \bP^{1}(\mathcal{H}(x)).$ Let $\beta=\g b_m$, where $\g:=b_1\cdots b_{m-1}$ (with the convention that $\gamma=\mathrm{id}$ when $m=1$). Then $D_{\beta}=\g D_{b_m}$, and by Proposition \ref{prop_coordinate_change}, there exists $c \in \mathcal{O}(V)$ such that for any $ y\in V$, the point $c(y) \in \mathcal{H}(y)$ is a center of~$D_{b_m, y}$. We note that $\g \cdot c(y) \in D_{\beta,y}$ and $a_1  \g \cdot c(y) \in D_{a_1,y}$.

Let  $\begin{pmatrix}
    u & v\\ z& t\\ 
\end{pmatrix}, {\begin{pmatrix}
   e & f\\i & h \\
\end{pmatrix} } \in \GL_2(\mathcal{O}(V))$ be representatives of $\g$ and $a_1$, respectively. As $D_{b_m,y}$ and $D_{a_1,y}$ are contained in the unit disk, they do not contain the point $\infty \in \bP^1(\Hr(y))$, so $|zc(y)+t|_y \neq 0$ and $|i(\g \cdot c(y)) + h|_y \neq 0$ for all $y \in V.$ Thus, $zc+t, i(\g \cdot c) + h \in \mathcal{O}(V)^{\times}$. 
Set $d :=\g\cdot c= \dfrac{uc+v}{zc+t} \in \mathcal{O}(V)$ and $p_{\alpha\beta, y}:=a_1 d(y) \in {D_{\alpha, y} \cap} a_1D_{\beta,y} \cap \bP^{1}(\Hr(y))$. Then $$t_{\alpha\beta, m, y}=|a_1'(d(y))|_y=\dfrac{|e h- fi|_y}{|id+h|_y^2},$$
which is continuous in $y$.
\end{proof}

From now on, whenever we use the matrices $T_{m,y}^{(s)}$, it will implicitly be with respect to a choice of points $(p_{\alpha\beta,y})_{\alpha, \beta}$ satisfying the properties of Lemma \ref{lem_choice_p}.

\begin{prop} \label{prop_power_T_uniform} Let $x\in \eS_{g}$ be a non-Archimedean point. For $\xi>1$ and large enough~${N \in \N}$, there exists  a neighborhood $V \subseteq \eS_{g}$ of $x$ and a basis $\mathcal{B}$ of $\Ga_V$ satisfying the properties of Proposition~\ref{prop_coordinate_change},  and so that for any $m \in \N_{\geqslant 1}$, the entries of $T_{m,y}^N$ are  strictly bounded above by~$\xi^{-1}$ for all $y \in V$. 
\end{prop}

\begin{proof}
Let $V$ be as in Proposition~\ref{prop_coordinate_change}.  This induces a {basis} $\mathcal{B}$ and associated Schottky figures $(D_{\g, y})_{\g \in \Ga_y \backslash \{\mathrm{id}\}}$ for $\Ga_y$ in~$\bP^{\pan}_{\Hr(y)}$, $y \in V$. 

{We first establish the statement at the  point $x$.}
Let $m \geqslant 1$ and $N \geqslant 2$.  
Let $\alpha, \beta$ be two elements of $\Ga_x$ of length $m$ with reduced forms $\alpha=a_1a_2\cdots a_m$ and $\beta=a_{N+1}a_{N+2}\cdots a_{N+m}$. Let $S_{\alpha, \beta, x}$ denote the entry of $T_{m,x}^N$ corresponding to $(\alpha, \beta).$
Set $L:= \min \rho(\partial D_{u^{-1},x}, \partial D_{v,x})$, where the minimum is taken over $u, v\in \mathcal{B} \cup \mathcal{B}^{-1}$ such that $u^{-1} \neq v$. By Lemma \ref{calc}, {for any $m \geqslant 1$}: \begin{equation*}
S_{\alpha, \beta,x} \leqslant Ne^{\rho(\eta, \partial D_{a_{N+1}})-\rho(\eta, \partial D_{a_1}) - N L},
\end{equation*}
where the right-hand side converges to $0$ as $N\rightarrow +\infty$.
Given $\xi>1$, we choose $N$ large enough so that $S_{\alpha, \beta, x} < \xi^{-1}$. 
Hence, the entries of the matrix $T_{m,x}^{N}$ are all strictly bounded above by~$\xi^{-1}$ for any $m \geqslant 1$.

By Lemma \ref{lem_choice_p}, we may assume that $V$ is such that the matrices $T_{m,y}^N$ vary continuously in~${y \in V}$. As the entries of $T_{m,x}^N$ are strictly bounded above by $\xi^{-1}$, by appropriately shrinking~$V$ (\emph{i.e.} depending on $N$), one can ensure that the entries of $T_{m,y}^N$ are also all strictly bounded above by $\xi^{-1}$.\end{proof}

We now show that the approximating functions $\delta_{m,y}$ are continuous.

\begin{cor} \label{cor_continuity_deltam}  
Let $x \in \eS_g$ be a non-Archimedean point. There exists $N  \in \N$ such that for any $m > N$, there exists a neighborhood $V \subseteq \eS_g$  of $x$  and a basis $\mathcal{B}$ of $\Ga_V$ satisfying the properties of Proposition \ref{prop_coordinate_change}, for which $y \mapsto \delta_{m,y}$ is continuous on $V$.
\end{cor}

\begin{proof} 
By Lemma \ref{lem_perron} and Proposition \ref{prop_power_T_uniform}, for~$\xi>1$ and $N$ large enough, there exists a neighborhood $V'$ of $x$ such that the entries of $T_{m,y}^{N}$ are in $(0, \xi^{-1})$ for all $m \geqslant 1.$

By Corollary \ref{postcalcul}, for any $y \in \eS_g,$ if $m \geqslant N$, we have that $(T_{m,y}^N)^{(s)}=(T_{m,y}^{(s)})^N$, so
$\lambda(T_{m,y}^{(s)})=\lambda((T_{m,y}^N)^{(s)})$, where the entries of $(T_{m,y}^N)^{(s)}$ are those of $T_{m,y}^N$ raised to the power $s \in  \R_{\geqslant 0}$. Consequently, $\delta_{m,y}=\delta_{T_{m,y}^N}$, with $\delta_{T_{m,y}^N}$ defined as for Proposition \ref{prop_conseq_perron}. By Propositions~\ref{prop_mc_mullen_NA} and~\ref{Hdimbounds}, $\delta_{T_{m,x}^N}>0$.
 Hence, by Proposition~\ref{prop_conseq_perron}, there exists a neighborhood $U \subseteq M_{n \times n}(\R)$ of~$T_{m,x}^N$ on which the function $\delta$ is continuous, where $n:=2g(2g-1)^{m-1}.$ 
 
By Lemma \ref{lem_choice_p}, we may assume that the choice of the points $p_{\alpha\beta, y}$ is such that the entries of~$T_{m,y}$ are all continuous in $y \in V'$. Hence, the map $t: V' \rightarrow M_{n \times n}(\R), y \mapsto T_{m,y}^N$, is continuous. Consequently, there exists a neighborhood $V$ of $x$ such that $V \rightarrow \R_{\geqslant 0},$ ${y \mapsto \delta_{T_{m,y}^N}=\delta_{m,y}},$ is continuous.  
\end{proof}

In the following proof, we reproduce and extend an argument of McMullen for Archimedean places from \cite[Theorem 2.2]{mcmullen_hausdorff_3} and combine it with our uniform estimates. 
\begin{thm} \label{thm_uniform_convergence_mcmullen} For $x \in \eS_{g}$ which is non-Archimedean, there exists a neighborhood $V \subseteq \eS_{g}$ and a basis $\mathcal{B}$ for $\Ga_V$ such that the sequence of functions $(y \mapsto \delta_{m,y})_{m \in \N}$ converges uniformly toward the function $y \mapsto \delta_y$, where $\delta_y:=\Hdim \La_{y}$ and $\La_{y}$ is the limit set associated to $\Ga_y$. 
\end{thm}

\begin{proof}
Let $V \subseteq \eS_g$ be a neighborhood of $x$, and $\mathcal{B}$ a basis of $\Ga_V$, such that the statements of  Proposition~\ref{prop_coordinate_change}, Lemma~\ref{lem_uniform_bound_locus}, and Proposition~\ref{prop_power_T_uniform} hold. This means there exist constants $c \in (0,1)$ and $R>0$ such that $r_{\g, y} \leqslant Rc^{l(\g)-1}$ for any $\g \in \Ga_y$ and any $y \in V$; there also exists~${\xi>1}$ such that for a large enough $N$ and any $m \in \N_{\geqslant 1}$, the entries of $T_{m,y}^N$ are strictly bounded above by $\xi^{-1}$ for all $y \in V$. We remark here that for $\g \in \Ga_y$ of large enough length (uniform on all of $V$), $r_{\g,y} \leqslant Rc^{l(\g)-1}<1$ for all $y \in V$. If $y$ is non-Archimedean, then $\delta_{m,y}=\delta_y$ for all $m \in \N_{\geqslant 1}$ by Proposition \ref{prop_mc_mullen_NA}. 

Now assume $y$ is Archimedean. Since $V$ is chosen so that Proposition \ref{prop_coordinate_change} holds, we can assume that the Schottky figure is contained in the open unit disk.
We consider two matrices $W_{m,y}^{(s)}$ and $U_{m,y}
^{(s)}$ whose entries are indexed by reduced words in $
\Ga_y$ of length~$m$. As in the definition of $T_{m}^{(s)}$ in \eqref{matrixT}, for $\alpha=a_1\cdots a_m$ and $
\beta=b_1 \cdots b_m$ such words, let the entry $w_{\alpha\beta, m,y} 
^{(s)}$ of $W_{m,y}^{(s)}$ be defined as
\begin{equation}
w_{\alpha\beta, m, y}^{(s)}=
\begin{cases}
0 & \text{if } D_{\alpha,y} \cap {a_1D_{ \beta,y} \cap \La_{y}} \neq \emptyset, \\
\inf_{p \in D_{\alpha,y} \cap a_1 D_{\beta, y} \cap \bP^1(\Hr(y))} |(a_1^{-1})'(p)|_y^{-s} & \text{otherwise.}
\end{cases}
\end{equation}
We define similarly the entry $u_{\alpha \beta, m,y}^{(s)}$ of the matrix $U_{m,y}^{(s)}$ by taking the supremum instead of the infimum. We note that $(W_{m,y}^{(s)})^n \leqslant (T_{m,y}^{(s)})^n \leqslant (U_{m,y}^{(s)})^n$ for any $n \in \N$, meaning these inequalities are true entry-wise.

Following McMullen's \cite[Theorem 2.2]{mcmullen_hausdorff_3}, consider the vector $v_{m,y}$ whose entries are~$\mu(D_{\alpha,y})$ for $\alpha \in \Gamma_y$ of length $m$. Here $\mu$ denotes a \emph{Patterson-Sullivan measure} on $\bP^{\pan}_{\Hr(y)}$ induced by the group $\Gamma_y$. We recall that $\mu$ is supported on the limit set ${\La_{y}}$.

As $D_{\alpha,y} \cap \Lambda_{y}=\bigsqcup_{l(\beta)=m} (D_{a_1\beta,y} \cap \La_{y})$ whenever ${D_{\alpha,y}} \cap {a_1D_{\beta,y}} \cap \La_{y}\neq \emptyset$, using the quasi-invariance property of $\mu$, we have for all $y \in V$ and all such $\beta \in \Ga_V$:  %
\begin{align*}
\mu(D_{\alpha,y}) &= \sum_{l(\beta)=m} \mu(D_{a_1\beta,y})
   = \sum_{l(\beta)=m} {(a_1^{-1}})_* \mu (D_{\beta,y}) 
 = \sum_{l(\beta)=m} \int_{D_{\beta,y}} |{a_1^{-1}}'(z)|_y^{-\delta_{y}} d\mu  \\ & \leqslant \sum_{l(\beta)=m} u_{\alpha \beta m, y}^{\delta_{y}}\int_{D_{\beta,y}} d\mu  \leqslant  \sum_{l(\beta)=m} u_{\alpha \beta, m, y}^{\delta_{y}} \mu(D_{\beta,y}) = r_{\alpha,m,y} v_{m,y}, 
\end{align*}
where $r_{\alpha,m,y}$ is the row of $U_{m,y}^{(\delta_y)}$ corresponding to the word $\alpha$.
Similarly, $\mu(D_{\alpha,y}) \geqslant r_{\alpha,m,y}' v_{m,y}$, where $r_{\alpha,m,y}'$ is the $\alpha$-row of $W_{m,y}^{(\delta_y)}$. 
Since this is true for all $\alpha$, \emph{i.e.} for all rows of~$U_{m,y}^{(\delta_y)}$ and~$W_{m,y}^{(\delta_y)}$, we obtain
$W_{m,y}^{(\delta_y)} v_{m,y} \leqslant v_{m,y} \leqslant U_{m,y}^{(\delta_y)} v_{m,y}$.
As these matrices have non-negative entries, 
{ $(W_{m,y}^{(\delta_y)})^N v_{m,y} \leqslant v_{m,y} \leqslant (U_{m,y}^{(\delta_y)})^N v_{m,y}.$}
The Perron-Frobenius theorem then implies that \begin{equation} \label{PF}
     \lambda((W_{m,y}^{(\delta_y)})^N) \leqslant 1 \leqslant \lambda((U_{m,y}^{(\delta_y)})^N).
\end{equation}
By Corollary \ref{cor_distortion_ball_bis}, we have for $m$ large enough and any $y \in V$:
\begin{equation*}
e^{-r_{\beta,y}\delta_y/4} \leqslant \dfrac{w_{\alpha\beta, m,y}^{(\delta_y)}}{u_{\alpha \beta, m,y}^{(\delta_y)}} \leqslant e^{r_{\beta,y}\delta_y/4}. 
\end{equation*}
Now using the uniform upper-bound $r_{\beta,y} \leqslant Rc^{m-1}$ of Lemma \ref{lem_uniform_bound_locus} with $m$ large enough so that $Rc^{m-1}<1$,  we get: 
\begin{equation} \label{fraccalc}
\exp\left(-\dfrac{R c^{m-1}\delta_y}{4}\right) u_{\alpha \beta, m,y}^{(\delta_y)} \leqslant w_{\alpha \beta, m,y}^{(\delta_y)} \leqslant \exp\left(\dfrac{R c^{m-1}\delta_y}{4} \right) u_{\alpha \beta, m,y}^{(\delta_y)}.  
\end{equation}

Denote by $\tilde{t}_{\alpha \beta, m,y}^{(s)}$, $\tilde{w}_{\alpha \beta, m,y}^{(s)}$, $\tilde{u}_{\alpha \beta, m,y}^{(s)}$ the coefficients of the matrices $(T_{m,y}^{(s)})^N$, $(W_{m,y}^{(s)})^N$, and $(U_{m,y}^{(s)})^N$ respectively. 
By Corollary \ref{postcalcul}, for $m \geqslant N$, we can apply \eqref{fraccalc} $N$-times and  obtain that for all $y\in V$ and all $\alpha, \beta \in \Gamma_y$ chosen  as above: 
\begin{equation} \label{eq_conseq_disto}
\exp\left(-\dfrac{NR c^{m-1}\delta_y}{4} \right) \leqslant \dfrac{\tilde{w}_{\alpha \beta, m,y}^{(\delta_y)}}{\tilde{u}_{\alpha \beta, m,y}^{(\delta_y)}}  \leqslant \exp\left(\dfrac{NR c^{m-1}\delta_y}{4} \right). 
\end{equation}

{By Corollary \ref{cor_local_bound_hausdorff}, we may assume  there exists $A>0$ such that  for any $y \in V$, we have~${\delta_y \leqslant A}$. 
Let $\varepsilon>0$. %
Let $M \in \N$ be large enough  {so that  for} 
$\eta:=\frac{A N R c^{M-1}}{{4}\log \xi} \leqslant \varepsilon,$}
we have $$\xi^{\eta}=\exp \left(\dfrac{AN R c^{M-1}}{4} \right) \geqslant \exp \left(\dfrac{\delta_y N R c^{m-1}}{4} \right)$$ for any~${m \geqslant M}$.   
By \eqref{eq_conseq_disto}, we obtain that 
$\xi^{-\eta} \leqslant \dfrac{\tilde{w}_{\alpha \beta, m,y}^{(\delta_y)}}{\tilde{u}_{\alpha \beta, m,y}^{(\delta_y)}}.$
As the bound depends only on $V$,
we have that $\xi^\eta (W_{m,y}^{(\delta_y)})^N \geqslant (U_{m,y}^{(\delta_y)})^N$, meaning these inequalities are true entry-wise for any $y \in V$ and any $m \geqslant \max(M,N)$.

Assume $\delta_y -\eta>0$. For $\alpha,\beta \in \Ga_y$ reduced words of length $m$, by Proposition \ref{prop_power_T_uniform}, each entry of $T_{m,y}^N$ is bounded by $\xi^{-1}$. As $m \geqslant N$,
by Corollary \ref{postcalcul}, {we obtain that $\tilde{t}_{\alpha\beta, m, y}^{(s)} \leqslant \xi^{-s}$ for $s \in \{\eta, \delta_y - \eta, \delta_y\}$, and $\tilde{t}_{\alpha\beta, m, y}^{(\eta)} \tilde{t}_{\alpha\beta, m, y}^{(\delta_y-\eta)} = \tilde{t}_{\alpha\beta, m, y}^{(\delta_y)}$.  Hence, ${\tilde{t}}_{\alpha \beta, m,y}^{(\eta)} \leqslant \xi^{-\eta}$ and  ${\tilde{t}}_{\alpha \beta, m,y}^{(\delta_y -\eta)} \xi^{-\eta} \geqslant {\tilde{t}}_{\alpha \beta, m,y}^{(\delta_y)}$, implying: 
\begin{equation*}
(T_{m,y}^{(\delta_y - \eta)})^N \geqslant \xi^{\eta} (T_{m,y}^{(\delta_y)})^N \geqslant \xi^{\eta} (W_{m,y}^{(\delta_y)})^N \geqslant (U_{m,y}^{(\delta_y)})^N.  
\end{equation*} 
As the entries of all these matrices are non-negative, the spectral radius is an increasing function, meaning $\lambda((T_{m,y}^{(\delta_y-\eta)})^N) \geqslant \lambda((U_{m,y}^{(\delta_y)})^N) \geqslant 1$ by \eqref{PF}.
}
Similarly, we have ${\tilde{t}}_{\alpha \beta, m, y}^{(\delta_y+\eta)} \leqslant \xi^{-\eta} {\tilde{t}}_{\alpha \beta, m, y}^{(\delta_y)},$ and so $(T_{m,y}^{(\delta_y+\eta)})^N \leqslant \xi^{-\eta} (U_{m, y}^{(\delta_y)})^N \leqslant (W_{m,y}^{(\delta_y)})^N$.
As a consequence, $\lambda((T_{m,y}^{(\delta_y+\eta)})^N) \leqslant \lambda((W_{m,y}^{(\delta_y)})^N) \leqslant 1$ by \eqref{PF}.

In summary, for any $y \in V$ such that $\delta_y>\eta$ and any $m \geqslant \max(M, N)$, we have \begin{equation} \label{specineq}
\lambda((T_{m,y}^{(\delta_y - \eta)})^N) \geqslant 1 \geqslant \lambda((T_{m,y}^{(\delta_y +\eta)})^N).
\end{equation}

By the Perron-Frobenius theory, the function $\lambda(T_{m,y}^{(s)})$ is continuous and strictly decreasing in~$s$, hence $\delta_{m,y}$ is the unique solution to $\lambda(T_{m,y}^{(s)})=1$. By \eqref{specineq},  if $y \in V$ is such that $\delta_y>\eta$, then $\delta_{m,y} \in [\delta_y - \eta, \delta_y + \eta]$. If $y\in V$ is such that $\delta_y \leqslant \eta$, then one still has $\delta_{m,y} \in [\delta_y-\eta, \delta_y+\eta]$, seeing as $\delta_{m,y}\geqslant 0$ and $\delta_{m,y} \leqslant \delta_y+\eta$.
As $\eta \leqslant \varepsilon$, we have thus proved that for any $y \in V$ and any $m \geqslant \max(M,N)$, we have $|\delta_{m,y}-\delta_y| \leqslant \varepsilon$, meaning the convergence is uniform in $m$. 
\end{proof}

\subsection{Continuity of the Hausdorff dimension}

{We refer to Section \ref{section_modulispace} for the definition of the moduli space $\eS_g$ of rank $g$ Schottky groups for~${g \geqslant 1}$}. We now have all the necessary tools to prove the following:

\begin{thm} \label{thm_continuity} Let $g \geqslant 1.$
The function $d: \eS_g \rightarrow \R_{\geqslant 0}, x \mapsto \Hdim{\La_{x}}$, is continuous, where $\Hdim{\La_x}$  denotes the Hausdorff dimension of the limit set $\La_x$ of the Schottky group $\Ga_x$.
\end{thm}

\begin{proof} If $g=1$, then $d$ is the constant zero function. Let us assume $g \geqslant 2.$

If $x \in \eS_g$ is Archimedean, let $\pi(x)=|\cdot|_{\infty}^{\alpha_0}$ for $\alpha_0 \in (0,1]$, where $\pi: \eS_g \rightarrow \M(\Z)$ denotes the projection. Let $\varepsilon>0$. By \cite[Theorem 3.1]{mcmullen_hausdorff_3}, the function $d$ is continuous on $\pi^{-1}(|\cdot|_{\infty}^{\alpha_0})\cong \eS_{g, \Hr(x)}$, so there exists a neighborhood $V_0 \subseteq \pi^{-1}(|\cdot|_{\infty}^{\alpha_0})$ of $x$ such that for any $y \in V_0$ one has $|d(x) -d(y)|<\varepsilon/2$. 

Let $c>0$ be such that $\frac{c}{\alpha_0-c}(d(x)+ \varepsilon/2) < \varepsilon/2$. By \cite[\S~2.3]{poineau_turcheti_universal}, the set $V:=\{y^{\alpha/\alpha_0}: y \in V_0, \alpha \in (\alpha_0-c, \alpha_0+c)\}$ is an open subset of $\eS_g$ containing $x$. We note that for any $z=y^{\alpha/\alpha_0} \in V$, one has $d(z)=\frac{\alpha_0}{\alpha} d(y)$ (see Remark \ref{rem_scaling}). Then $$|d(x)-d(z)|\leqslant |d(x) -d(y)| + |1-\frac{\alpha_0}{\alpha}|d(y) \leqslant \varepsilon/2+ 
\frac{c}{\alpha} d(y) \leqslant  \varepsilon/2+ \frac{c}{\alpha_0-c}(d(x)+ \varepsilon/2) < \varepsilon,$$ 
showing that $d$ is continuous on the (purely Archimedean) neighborhood $V$ of $x$. 

Assume $x$ is non-Archimedean. Then, by Theorem~\ref{thm_uniform_convergence_mcmullen}, there exists a neighborhood $V \subseteq \eS_g$ of $x$ such that the sequence of functions $(y \mapsto \delta_{m,y})_{m \in \N}$ on $V$ converges uniformly to the function~$d_{|V}.$ By Corollary \ref{cor_continuity_deltam}, there exists $N \in \N$ such that the functions $y \mapsto \delta_{m,y}$ are continuous at~$x$ for all $m > N.$ As a uniform limit of continuous functions, $d$ is continuous at~$x$.

\end{proof}
We now prove that, more generally, the Hausdorff dimension of the limit sets varies continuously over the moduli space of Schottky groups over \emph{any} Banach ring (\emph{cf.} Definition \ref{def_banachring}). We recall the construction of this moduli space in Section~\ref{section_modulispace}. 

\begin{thm} \label{thm_contbanach}
Let $g \geqslant 1.$ Let $(A, \|\cdot\|)$ be a Banach ring. The function $d_A: \eS_{g, A} \rightarrow \mathbb{R}_{\geqslant 0}$, ${x \mapsto \Hdim \La_x}$, is continuous. Here $\La_x$ denotes the limit set of the Schottky group $\Ga_x$ associated to~$x \in \eS_{g,A}$. 
\end{thm}

\begin{proof}
 Let $\pi_A: \eS_{g, A} \rightarrow \eS_g$ be the \emph{continuous} projection morphism obtained from Remark \ref{rem_banachring}  and Lemma \ref{lem_basechangemoduli}. Let $x \in \eS_{g, A}$ and $y:=\pi_A(x) \in \eS_g$. Let $\Ga_x \subseteq \PGL_2(\Hr(x))$, respectively $\Ga_y \subseteq \PGL_2(\Hr(y))$, denote the Schottky group associated to $x$, respectively $y$. By construction, $\Ga_y$ is the preimage of $\Ga_x$ via $\PGL_2(\Hr(y)) \subseteq \PGL_2(\Hr(x))$. By Lemma \ref{cor_hausdorff_dim_extension} (see also Remarks~\ref{rem_sphericalmetric} and~\ref{sphericalm}), one obtains that $\Hdim \La_x=\Hdim \La_y$. Consequently, $d_A=d_{\Z} \circ \pi_A$. By Theorem~\ref{thm_continuity},~$d_{\Z}$ is continuous, so $d_A$ is also a continuous function.  
\end{proof} 

\begin{rem} \label{rem_analytic} \begin{enumerate}
\item  It was shown by Ruelle in \cite{ruelle} that $d_{|\eS_{g, \C}}$ is real-analytic. 
\item In a similar direction, if $k$ is non-Archimedean, the combination of 
Propositions \ref{prop_mc_mullen_NA}, \ref{prop_power_T_uniform} and \ref{prop_conseq_perron} 
  implies that $d_{|\eS_{g,k}}$ is continuous. But, moreover, in the proof of Proposition~\ref{prop_conseq_perron}, we see that it can locally be written as $g \circ (f_1, \dots, f_l)$, where $g$ is a real analytic function $U \rightarrow \R$, with $U \subseteq \R^l$ for some $l \geqslant 1$, and $(f_1, \dots, f_l)$ sends a point $y$ to $T_{m,y}^N$ for some $N \in \N$. By Lemma~\ref{lem_mcmullen_matrix_coef_NA},  $f_i: y \mapsto \dfrac{r_{a^i,y}}{r_{b^i, y}}e^{-\modu(D_{(a^i)^{-1},y}, D_{b^i, y})}$, where $a^i, b^i$ are length one words, and~$r_{a,y}$ denotes the radius of a relative twisted Ford disk in the fiber $\bP^{\pan}_{\Hr(y)}$. One notes that due to the definition of these disks, the quantities~$r_{a,y}$ and $e^{-\modu(D_{a,y}, D_{b,y})}$
are of the form $|\alpha|_y$ for some $\alpha \in \mathcal{O}(V)$, where $V \subseteq \eS_{g,k}$ is an open and~$y \in V$.
\end{enumerate}
\end{rem}

\section{Degeneration of Schottky groups} \label{section_degen}

We use the continuity of the Hausdorff dimension obtained in Section \ref{section_unifoorm} to study the asymptotic behavior of the Hausdorff dimension of degenerating families of Schottky groups over $\C$. We then give some examples.

\subsection{Asymptotic behavior of the Hausdorff dimension}
Following \cite{berkovich_non_archimedean} and \cite[Appendix A.3]{boucksom_jonsson_degeneration}, for $r:=\frac{1}{e}$ let us consider the Banach ring
$$A_r := \{ f = \sum c_\alpha t^\alpha \in \C((t)) \ | \  ||f||_{\mathrm{hyb}}:= \sum_{\alpha \in \mathbb{Z}} |c_\alpha|_{\mathrm{hyb}} r^\alpha < +\infty \}. $$
In Section \ref{section_berkovich_Z}, we recalled the notion of the Berkovich spectrum $\M(A_r)$ of $A_r$.

For $s>0$, let $D_s, \overline{D}_s \subseteq \C$ denote the open, respectively closed, disk centered at $0$ and of radius $s$. Set $D_s^{*}:=D_s \backslash \{0\}$ and $\overline{D_s}^*:=\overline{D_s} \backslash \{0\}.$
By \cite[Proposition~A.4]{boucksom_jonsson_degeneration}, there is a homeomorphism $\overline{D}_r \rightarrow \M(A_r),$ $z \mapsto |\cdot|_z$, where for $f \in A_r,$
$$|f|_z=\begin{cases}
e^{-\mathrm{ord}_0(f)} & \text{ if } z=0, \\
|f^{\mathrm{hol}}(z)|_{\infty}^{-\frac{1}{\log |z|_{\infty}}} & \text{ if } z \neq 0.\\
\end{cases}$$
Here $|\cdot|_{\infty}$ denotes the absolute value on $\C$, and 
$f^{\mathrm{hol}}$ denotes the holomorphic function on the open disk $D_r$ induced by the formal Laurent series $f \in A_r$. 

For $z \in {\overline{D}_r}^*$, set $\varepsilon(z)=-\frac{1}{\log |z|_{\infty}}.$
We note that for any such $z$ the completed residue field of~${|\cdot|_z}$ (\emph{cf.} Section \ref{section_berkovich_Z}), is $\mathcal{H}(|\cdot|_z)=(\C, |\cdot|_{\infty}^{\varepsilon(z)})$. For $z=0$, the completed residue field of $|\cdot|_0$ is $\Hr(|\cdot|_0)=(\C((t)), |\cdot|_t)$, where $|\cdot|_t=e^{-\mathrm{ord}_0(\cdot)}$.

Let $g>1$ be an integer and $r'$ a real number such that $0<r'\leqslant r$. Let $p: D \rightarrow \C^{3g-3}$ be a function continuous on $\overline{D}_{r'}^*$, holomorphic on ${D_{r'}}^*$, and meromorphic at $0$. Let us denote its coordinate functions $D \rightarrow \C$ by $\alpha_3, \dots, \alpha_g, \beta_2, \dots, \beta_g, y_1, \dots, y_g$. We will sometimes simply denote $p(z)=(\underline{\alpha}(z), \underline{\beta}(z), \underline{y}(z))$ for~${z \in D}$. 

Note that there is a natural isomorphism $N_{z}: (\C, |\cdot|_{\infty}) \rightarrow (\C, |\cdot|_{\infty}^{\varepsilon(z)}).$
Hence, the point $p(z) \in \C^{3g-3}=\Af_{\C, |\cdot|_{\infty}}^{3g-3, \an}$ induces uniquely a point $N_z(p(z)) \in \Af_{\C, |\cdot|_{\infty}^{\varepsilon(z)}}^{3g-3, \an}=\Af_{\Hr(|\cdot|_z)}^{3g-3, \an}$ for $z \neq 0$. 

As $p$ is meromorphic at $0$, the coordinate functions induce elements of $\C((t))$ (when replacing the complex variable $z$ by a formal indeterminate $t$), so $p(0)$ determines uniquely an element of $\Af_{\C((t))}^{3g-3, \an},$ which is moreover a $\C((t))$-point. 

Recall that $\Af_{\Hr(|\cdot|_z)}^{3g-3, \an}$ can always be embedded in $\Af_{A_r}^{3g-3, \an}$ as the fiber over $|\cdot|_z$ of the projection $\Af_{A_r}^{3g-3, \an} \rightarrow \M(A_r)$, where $z \in \overline{D_r}$.

\begin{prop} \label{prop_continuous_section}
As above, let $p=(\underline{\alpha}, \underline{\beta}, \underline{y}): D \rightarrow \C^{3g-3}$ be a function continuous on $\overline{D}_{r'}^*,$ holomorphic on $D_{r'}^*$, and meromorphic at $0$ for some positive $r'\leqslant r.$ Assume moreover that for any $i=1,2,\dots, g$:  
\begin{enumerate}
\item[(i)] \label{contsect1} $y_i(z)$ are not constant and $y_i(0)=0$,
\item[(ii)] \label{contsect2} for all $u(z), v(z) \in \{\alpha_1(z), \beta_1(z), \dots, \alpha_g(z), \beta_g(z)\} \backslash \{\alpha_i(z), \beta_i(z)\}$, 
$$\mathrm{ord}_{0} ( y_i(z) \cdot [u(z), v(z); \alpha_i(z), \beta_i(z)]) >0,$$
where $\alpha_1(z):=0$, $\beta_1(z):=\infty$, $\alpha_2(z):=1$ for all $z \in D$,  and $[a, b; c, d]$ denotes the cross-ratio of $a,b,c,d \in \bP^1(\C)$, 
\item[(iii)] \label{contsect3} the functions $\alpha_i(z), \beta_i(z)$, $i=1,2,\dots, g$, are all different. 
\end{enumerate}

Let $\widetilde{p}: \overline{D}_{r'} \rightarrow \Af_{A_r}^{3g-3, \an}$ be defined by 
$$z \mapsto \begin{cases}
N_z(p(z)) & \text{ if } z \neq 0,\\
p(0) & \text{ if } z=0.
\end{cases}$$
Then $\widetilde{p}$ is continuous and there exists $ \eta \in (0, r']$ such that $\widetilde{p}_{|D_{\eta}}: D_{\eta} \rightarrow \eS_{g, A_r}$ is well-defined. 
Moreover, $\Hr(N_z(p(z)))=(\C, |\cdot|_{\infty}^{\varepsilon(z)})$ for $z \in \overline{D_{r'}}^*,$ and $\Hr(p(0))=(\C((t)), |\cdot|_t)$.
\end{prop}

\begin{proof} For $z \neq 0$,
the point $N_z(p(z)) \in \Af_{A_r}^{3g-3, \an}$ corresponds to the semi-norm $|\cdot|_{x_z}$ on $A_r[T_1, \dots, T_{3g-3}]=:A_r[\underline{T}]$ given by $P(\underline{T}) \mapsto |P(p(z))|_{\infty}^{\varepsilon(z)}$. Clearly, $\widetilde{p}$ is continuous on $\overline{D}_{r'}^*$. When restricted to $\C$, the semi-norm $|\cdot|_{x_z}$ induces $|\cdot|_{\infty}^{\varepsilon(z)}$, so $\Hr(|\cdot|_{x,z})=(\C, |\cdot|_{\infty}^{\varepsilon(z)})$ by the Gelfand-Mazur Theorem.

The point $p(0)$ corresponds to the semi-norm $|\cdot|_{x_0}$ given by $P(\underline{T}) \mapsto |P(\underline{\alpha}, \underline{\beta}, \underline{y})|_t$, where $(\underline{\alpha}, \underline{\beta}, \underline{y})$ denotes the element in $\C((t))^{3g-3}$ induced by the Laurent series of $p(z)$ as a meromorphic function on $0$. We note that $\C((t))[T_1, \dots, T_{3g-3}]/(\underline{T}-(\underline{\alpha}, \underline{\beta}, \underline{y}))$ is isomorphic to $\C((t))$, and the quotient norm obtained from $|\cdot|_{x_0}$ is precisely $|\cdot|_t$. Hence, $\Hr(|\cdot|_{x_0}) = (\C((t)), |\cdot|_t).$

As $p$ is holomorphic on $D_{r'}^*$  and meromorphic on $\overline{D_{r'}}$, so is $P(p(z))$ seeing as the coefficients of $P \in A_r[\underline{T}]$ satisfy the same properties on $D_r.$ Hence, $P(p(z))$  can be written as $P(p(z))=z^{\mathrm{ord}_0(P(p(z)))}u(z)$, where $u$ is continuous on $\overline{D}_{r'}$, holomorphic on $D_{r'}$, and $u(0) \neq 0$. As $z \rightarrow 0$, one obtains that $|P(p(z))|_{x_z}=|P(p(z))|_{\infty}^{\varepsilon(z)} \rightarrow e^{-\mathrm{ord}_0(P(p))}=|P(p)|_{x_0}$ for any polynomial $P \in A_r[\underline{T}]$. Consequently, $|\cdot|_{x_z} \rightarrow |\cdot|_{x_0}$ as $z \rightarrow 0$, and so $\widetilde{p}$ is continuous on $\overline{D}_{r'}$. 
 
Finally, by Corollary \ref{cor_charac_Schottky}, $\widetilde{p}(0) \in \eS_{g, A_r}$. We can now conclude seeing as $\eS_{g, A_r}$ is open in~$\Af_{A_r}^{3g-3, \an}$ and $\widetilde{p}$ is continuous.  
\end{proof}

Using similar notation, we can now prove a degeneration result.

\begin{thm} \label{degeneration} Let $p \colon D \to \C^{3g-3}$ be a %
function satisfying the hypotheses of Proposition \ref{prop_continuous_section}. 
There exists $\eta > 0 $ such that $p(z) \in \eS_{g, \C}$  for $0<|z| < \eta$ and the Hausdorff dimension~$s(z)$ of the limit set $\Lambda_z$ associated to $p(z)$ satisfies: 
\begin{equation*}
s(z) \sim - \dfrac{s_0}{\log|z|},
\end{equation*}
as $z \rightarrow 0$, where $s_0>0$ is the Hausdorff dimension of the limit set in $\mathbb{P}^{\pan}_{\C((z))}$ of the Schottky group with generators whose Koebe coordinates are determined by the Laurent series induced by~$p(z)$.
\end{thm}

\begin{proof} By Proposition \ref{prop_continuous_section}, there exists $\eta>0$ such that for any $z$ with $0<|z|< \eta$, one has $N_z(p(z)) \in \eS_{g,(\C, |\cdot|_z^{\varepsilon(z)})}$. Consequently, $p(z) \in \eS_{g, \C}$. 

Moreover, by {Theorem \ref{thm_contbanach}}, the map $D_{\eta} \rightarrow \eS_{g, A_r} \rightarrow \R_{\geqslant 0}, z \mapsto N_z(p(z)) \mapsto \Hdim_{\Hr(|\cdot|_z)} \Lambda_{z}$, is continuous, where $\Hdim_{|\cdot|}$ denotes the Hausdorff dimension with respect to the norm $|\cdot|$.  Hence, when $z \rightarrow 0$, we have that $\Hdim_{\Hr(|\cdot|_z)} \Lambda_{z} \rightarrow s_0$. 

We recall that $\Hr(|\cdot|_z)=(\C, |\cdot|_{\infty}^{\varepsilon(z)})$ with  $\varepsilon(z)=-\frac{1}{\log |z|_{\infty}}$ when $z \neq 0$. By Remark \ref{rem_scaling}, $\Hdim_{\Hr(|\cdot|_z)} \Lambda_{z}=\Hdim_{|\cdot|_{\infty}^{\varepsilon(z)}} \La_z=\dfrac{1}{\varepsilon(z)} \Hdim_{|\cdot|_{\infty}} \La_z=s(z)$, and we may conclude.
\end{proof}

\subsection{Examples of degenerations} \label{sect_examples}
To conclude, we give some examples of degenerations to which Theorem \ref{degeneration} applies.
To illustrate how the results can be applied, we show in detail how to treat the degenerations of \emph{Schottky reflection groups}, and then give some further examples for which the same method works.

Fix $g \geqslant 2$, distinct $g+1$ points $c_1, \ldots , c_{g+1}\in \mathbb{S}^1$  on the unit circle, and some integers $k_1, \ldots, k_{g+1} \geqslant 1$.  
Fix $z \in \R^{\times}$ and let $D_{j,z}$ be the disks bounded by circles orthogonal to $\mathbb{S}^1$ at the points $c_j \exp(iz^{k_j}/2), c_j \exp({-iz^{k_j}/2})$ for $j=1,2, \dots, g+1$. 

Let $\sigma_{j}$ be the reflection (\emph{i.e.} inversion) with respect to the boundary of the disk $D_{j,z}$. We set $D_{j,z}' = \sigma_{g+1}(D_{j,z})$ and $\gamma_{j,z} :=  \sigma_j\sigma_{g+1}$ for all $j =1,2,\dots, g$.  
Then for $z$ small enough, the disks $(D_{1,z}, D_{1,z}', \ldots, D_{g,z}, D_{g,z}')$ are disjoint, and define a Schottky figure for the group $\Ga_z$ generated by $\gamma_{1,z}, \ldots, \gamma_{g,z}$ (see \emph{e.g.} \cite[\S~2]{guillope_lin_zworski}). 
The group $\Ga_z$, $z \neq 0,$ is called a \textit{Schottky reflection group} associated to $\underline{c} := (c_1, \ldots, c_{g+1}) \in (\mathbb{S}^1)^{g+1}$ and the weights $\underline{k}:= (k_1, \ldots, k_{g+1}) \in \Z_{\geqslant 1}^{g+1}$.

\begin{thm} \label{thm_explicit} For $g \geqslant 2$, let  $c_1, c_2 \ldots , c_{g+1} \in \mathbb{S}^1$ be distinct points, and $k_1, k_2$ $\ldots$, ${k_{g+1}\geqslant 1}$ any integers. Let $\Ga_z, z \neq 0,$ denote the Schottky reflection group associated to $\underline{c}$ and $\underline{k}$ for $z \in \R^{\times}$ small enough. Then:
\begin{enumerate}
\item[(i)] The family $\Gamma_z, z \in \R^{\times},$ induces a function satisfying the conditions of Theorem \ref{degeneration}, whose image is $\{\Ga_z\}_{z}$ for $|z|$ small enough.

    \item[(ii)]   If $k_i = k$ for all $i \leqslant g$, then the Hausdorff dimension $s(z)$ of the limit set of $\Ga_z$ satisfies: 
  \begin{equation*}
      s(z) \sim \dfrac{-\log u}{2 \log |z^{-1}|},
  \end{equation*}
  as $z\rightarrow 0$, where $u$ is the largest root of $X^{k+k_{g+1}} + (1 - 1/g) X^k - 1/g$. 
  \item[(iii)] If $k_i =1$ for all $i \leqslant g+1$, then the Hausdorff dimension of the limit set of $\Ga_z$ satisfies: 
  \begin{equation*}
      s(z) \sim \dfrac{\log g}{2 \log |z^{-1}|},
  \end{equation*}
  as $z\rightarrow 0$. 
\end{enumerate}
\end{thm}

\begin{rem} Note that when $g=2$ and the points $c_i$ form an equilateral triangle, then we recover McMullen's \cite[Theorem 3.5]{mcmullen_hausdorff_3}. 
\end{rem}

\begin{proof}[Proof of Theorem \ref{thm_explicit}]
Let us prove assertion (i). Without loss of generality, by applying a rotation, we may assume that $c_{g+1}=1.$ Set $l = k_{g+1}$.
The reflections $\sigma_i(z)$ can be given by the matrices $$M_i(z):=\begin{pmatrix}
    1 + \exp(i z^{k_i}) & - 2 c_i \exp(i z^{k_i}/2)  \\
    2 c_i^{-1} \exp(i z^{k_i}/2) & -1 - \exp(i z^{k_i}) 
\end{pmatrix}.$$
 The transformations $\g_i(z)=\sigma_i \sigma_{g+1}$ then have analytic entries in $z$ for any $i=1,2, \dots, g.$ Let $y_i(z)$ denote the third Koebe coordinate of $\g_i(z)$. To determine it, one needs to find the eigenvalues of the matrices $N_i(z):=M_i(z) M_{g+1}(z)$ and compute their quotient. 
Seeing as $\det(M_{g+1}(0)) = \det(M_i(0))= 0$ and  ${\mathrm{Tr} (N_i(0)) = 4 (2 - c_i - c_{i}^{-1}) \in \R_{>0}}$, the discriminant $\mathrm{Tr}(N_i(z))^2 - 4\det N_i(z)$ of the characteristic polynomial of $N_i(z)$ admits a square root which is analytic near zero.
The eigenvalues of $\gamma_i$ are thus analytic, and their quotient is meromorphic at $z=0$.

Since one eigenvalue of $N_i(z)$ vanishes at $z=0$ and the other is $\mathrm{Tr}(N_i(0)) \neq 0$, we obtain $y_i(0)=0$, and hence $y_i(z)$ is analytic at $z=0$. The  two eigenvalues being distinct and analytic on a neighborhood of zero, their eigenspaces $\alpha_i(z), \beta_i(z)$ are also analytic there (see \emph{e.g.} \cite[p.68]{kato}). Hence, the Koebe coordinates $(\alpha_i(z), \beta_i(z), y_i(z))$ of $\g_i(z)$ are analytic at a neighborhood of zero for all $i=1,2,\dots, g$. Moreover, seeing as $\ord_0 \det N_i(z)=\ord_0(M_i(z)M_{g+1}(z))=2l+2k_i$, one has that $\ord_0 y_i(z)=2l+2k_i$.

Let $D_{i,z} := D(u_i(z), r_i(z))$. As the length of the arc passing through $c_ie^{iz^{k_i}/2}$ and $c_i e^{-iz^{k_i}/2}$ is~$z^{k_i}$, and $u_i(z)$ lies in the half-line emanating from the center of the unit disk and passing through~$c_i$, we obtain that $u_i(z)=\dfrac{c_i}{\cos z^{k_i}/2}$ and $r_i(z)=\tan z^{k_i}/2 \sim z^{k_i}/2$ for $i=1,2,\dots, g+1$.  We note that $\ord_0(u_i-u_j)=0$ and $\ord_0 r_i=k_i$ for all~${i \neq j}$ and $i, j \in \{1,2, \dots, g\}$.

{If a circle of center $c \in \C$ and radius $a$ is reflected with respect  to the circle with center $p \in \C$ and radius $b$, then one obtains the circle centered at $p + (c-p)w$ and of radius $a |w|$, where $$w:=\frac{b^2}{|c-p|^2-a^2}.$$ }
Let $D_{i,z}'=:D(v_i(z), R_i(z))$. Then, by using the above paragraph, one can show that ${\ord_0(v_i)=0}$ and $\ord_0 R_i=k_i+2l$. Moreover,  
$\ord_0 (u_i - v_j) = \ord_0 (u_i - v_i) = 0,$ and
$\ord_{0}(v_{i}- v_j) = 2l$, whenever $i \neq j$ and $i, j=1,2,\dots, g.$ We remark that the functions $u_i(z), v_i(z), r_i(z)$ and $R_i(z)$ are all analytic on the unit disk.

Since $\alpha_i(z) \in D_{i,z} $ and $\beta_i(z) \in D_{i,z}'$, we have  ${\ord_0(\beta_i - v_i) \geqslant k_i +2l}$, $\ord_0 (\alpha_i - u_i) \geqslant k_i$. Hence, for $j\neq i$, using the ultrametric inequality: $${\ord_0 ( \beta_j - \beta_i) = \min (\ord_0( \beta_j - v_j), \ord_0(v_j - v_i), \ord_0(v_i - \beta_i)) = 2l}.$$ Here, the equality is due to the fact that $\ord_0(v_j-v_i)< \min(\ord_0(\beta_j-v_j), \ord_0(v_i- \beta_i)).$ Similarly, one has $\ord_0(\alpha_j - \beta_i) = \ord_0(\alpha_j - \alpha_i)= 0$ for \emph{any} $i,j=1,2,\dots, g$. 
Consequently, for $u,v \in \{ \alpha_1, \beta_1, \ldots, \alpha_g, \beta_g \} \setminus \{ \alpha_i, \beta_i\}$: $$\ord_0 (y_i [u,v;\alpha_i,\beta_i])>0,$$ so by Corollary \ref{cor_charac_Schottky}, the transformations $\g_{i,z}, i=1,\dots, g$, generate a Schottky group over~$\C((z))$.
 
As $\ord_0(\alpha_i-\beta_i)=0$ for $i=1,2,\dots, g$, by Lemmas II.3.29 and II.3.30 of \cite{poineau_turcheti_2}, the closed disks $D_i:=D(u_i, e^{-k_i})$ and $D_i':=D(v_i, e^{-(2l+k_i)})$ are twisted Ford disks for $\g_i$ in $\bP^{\pan}_{\C((z))}$. The family of disks $\{D_i, D_i'\}_i$ is mutually disjoint, so it forms a Schottky figure for $\Ga_0$ over $\C((z))$ generated by the $\g_i$. In particular, Theorem \ref{degeneration} is applicable, so assertion (i) holds.  Moreover, $s(z) \sim s_0 / \log|z^{-1}|$ as $z\rightarrow 0$ where $s_0$ is the Hausdorff dimension of the limit set $\Lambda_0$ on $\mathbb{P}^{\pan}_{\C((z))}$.

To show (ii), let us assume that $k_i =k$ for $i \leqslant g$. Let $M(s)$ be the Perron matrix from Definition \ref{matrixM} associated to the Schottky figure $D_i, D_i'$ in~$\bP^{\pan}_{\C((t))}.$ In this case, it is a  $g\times g$ block matrix of the form: 
\begin{equation*}
    M(s) = \left ( \begin{array}{lllll}
    D(s) & A(s) & A(s) &\cdots &A(s) \\ 
    A(s) & D(s) & A(s) & \cdots &A(s) \\ 
    \vdots & \vdots & \vdots && \vdots \\ 
    A(s) & A(s) & A(s) & \cdots & D(s) 
    \end{array} \right ) ,
\end{equation*}
where the $2\times 2$ blocks are  $A(s) := \left ( \begin{array}{ll}
    e^{-2s(k+l)} & e^{-2sk}  \\
    e^{-2sk } & e^{-2s(k+l)}
\end{array} \right )$ and $D(s):= \left ( \begin{array}{cc}
     e^{-2s(k+l)}& 0  \\
     0&e^{-2s(k+l)} 
\end{array} \right)  $.
By Corollary \ref{cor_spectral_formula}, $s=s_0$  if and only if the spectral radius of $M(s_0)$ is $1$. 
Observe that the vector $(1 , \ldots, 1)$ is a positive eigenvector for all $s$, and since $M(s)$ is a Perron matrix, it is the eigenvector for the largest eigenvalue. This implies that $s_0$ is a solution to the equation $g e^{-2 s(k+l)}+ (g-1)e^{-2sk} = 1$, and hence that $s_0 = -\log u /2$, where $u $ is the largest root of $X^{k+l} + (1 - 1/g)X^k -1$.   
 We have thus showed that $s(z) \sim -\log(u)/(2 \log |z^{-1}|)$ as $z\rightarrow 0$. 

Assertion (ii) implies (iii): if $k_i=1$ for $i \leqslant g+1$, then $1/g$ and $-1$ are the roots of the polynomial $X^2 + (1 - 1/g)X - 1/g$, hence $u=1/g$ in (ii) yields the conclusion of (iii). 
\end{proof}

In Table~\ref{table_degen}, we give some further examples where the computation was worked out. 
Let $\alpha_1, \ldots, \alpha_g$,  $\beta_0, \ldots, \beta_g$ be a collection of distinct complex numbers with $\beta_1 \beta_2 \ldots \beta_g \neq 0$, and $y_1, \ldots, y_g \in \C^{\times}$. Let also~${k \in \N_{\geqslant 1}}$. 
To simplify the exposition we choose the Koebe coordinates in $\C^{3g}$ without marking the three points $0,1,\infty$. We continue denoting by $s(z)$ the Hausdorff dimension of the limit set of the given Schottky group $\Ga_z$.

\begin{rem}
   When $g=2$, all  possible geometric collisions are covered in Table~1. 
\end{rem} 

\begin{rem}
    In the half collision case, although  the repelling points collide to the same point as in the case of Schottky reflection groups, the speed of collision differs. 
\end{rem}

\begin{rem}
Note in each of the situations of Theorem \ref{thm_explicit} and the first four situations in Table~\ref{table_degen}, if one takes $g \rightarrow +\infty$, and considers the limit set $\Lambda_{g,z}$ of the group $\Ga_{g,z}$, then the double limit $$\lim_{g\rightarrow +\infty}\lim_{z \rightarrow 0} \dfrac{\Hdim \Lambda_{g,z} \log|z^{-1}|}{\log g}  $$
exists and is a positive rational number. 
\end{rem}

\begin{table}[H] 
    \centering
    \begin{tabular}{|c|c|c|}
    \hline
        Geometric situation &Koebe coordinates of $\Ga_z$  &  Asymptotic formula for $s(z)$ \\
       \hline
       \multirow{2}{*}{All points collide} & $(z \alpha_1, \ldots , z \alpha_g, z\beta_1, \ldots, z\beta_g,$   & \multirow{2}{*}{ $s(z) \sim \dfrac{\log(2g-1)}{ \log|z^{-1}|}$ }\\
       &  $zy_1, \ldots , z y_g )$ & \\
       \hline 
      \multirow{2}{*}{Half collide} & $(\alpha_1, \ldots, \alpha_g, \beta_0 + z \beta_1, \ldots , \beta_0 + z\beta_g,$ & \multirow{2}{*}{$s(z) \sim \dfrac{\log g}{\log |z^{-1}|}$}\\ 
      &   $z^2 y_1, \ldots , z^2 y_g)$ 
      &  \\
       \hline 
       \multirow{2}{*}{Pair by pair collision} & $(\alpha_1 , \ldots, \alpha_g, \alpha_1 + z \beta_1, \ldots ,\alpha_g+ z \beta_g,$ & $s(z) \sim \dfrac{-\log u}{\log |z^{-1}|}$, with $u \in \R$  \\ 
        & $z y_1, \ldots , zy_g)$
        & \text{a root of} $X^3 + \dfrac{X}{2g-2} - \dfrac{1}{2g-2}$ \\ 
        \hline
       \multirow{2}{*}{No collision} & $(\alpha_1, \ldots , \alpha_g, \beta_1, \ldots , \beta_g,$ &  \multirow{2}{*}{$s(z) \sim\dfrac{\log(2g-1)}{k \log |z^{-1}|} $}\\
       &  $z^k y_1, \ldots , z^k y_g)$ & \\
       \hline
       \multirow{1}{*}{$g=2$, one pair collision} & $(\alpha_1 , \alpha_2, \alpha_1 + z\beta_1, \beta_2, z y_1, zy_2)$
       & \multirow{2}{*}{ $s(z)\sim \dfrac{\log 2}{\log |z^{-1}|}$} \\
       \cline{1-2}
       $g=2$, one mixed collision & $(\alpha_1, \alpha_2, \alpha_2 + z \beta_1, \beta_2, z^2 y_1, z^2 y_2)$ 
       & \\
       \hline 
       {$g=2$, three point collision } & {$(\alpha_1, \alpha_2, \alpha_2+ z\beta_1, \alpha_2 + z\beta_2, zy_1, zy_2)$
       } 
       & {$s(z) \sim \dfrac{\log 3}{\log |z^{-1}|}$ }\\
       \hline
    \end{tabular}
    \caption{Some examples of degenerations}
    \label{table_degen}
\end{table}

\addcontentsline{toc}{section}{References}
\begin{small}
\bibliographystyle{alpha}
\bibliography{biblio}

\begin{thebibliography}{BCDS23}

\bibitem[Aka72]{akaza_example}
Tohru Akaza.
\newblock {$(3/2)$}-dimensional measure of singular sets of some {K}leinian
  groups.
\newblock {\em J. Math. Soc. Japan}, 24:448--464, 1972.

\bibitem[Arf17]{arfeux3}
Matthieu Arfeux.
\newblock Dynamics on trees of spheres.
\newblock {\em J. Lond. Math. Soc. (2)}, 95(1):177--202, 2017.

\bibitem[BCDS23]{benard_chaubet_dang}
Léo Bénard, Yann Chaubet, Nguyen~Viet Dang, and Thomas Schick.
\newblock Combinatorial zeta functions counting triangles.
\newblock \texttt{https://arxiv.org/abs/2303.11226}, 2023.

\bibitem[Bea83]{beardon}
Alan~F. Beardon.
\newblock {\em The geometry of discrete groups}, volume~91 of {\em Graduate
  Texts in Mathematics}.
\newblock Springer-Verlag, New York, 1983.

\bibitem[Ber75]{bers}
Lipman Bers.
\newblock Automorphic forms for {S}chottky groups.
\newblock {\em Advances in Math.}, 16:332--361, 1975.

\bibitem[Ber90]{berkovich_spectral}
Vladimir~G. Berkovich.
\newblock {\em Spectral theory and analytic geometry over non-{A}rchimedean
  fields}, volume~33 of {\em Mathematical Surveys and Monographs}.
\newblock American Mathematical Society, Providence, RI, 1990.

\bibitem[Ber09]{berkovich_non_archimedean}
Vladimir~G. Berkovich.
\newblock A non-{A}rchimedean interpretation of the weight zero subspaces of
  limit mixed {H}odge structures.
\newblock In {\em Algebra, arithmetic, and geometry: in honor of {Y}u. {I}.
  {M}anin. {V}ol. {I}}, volume 269 of {\em Progr. Math.}, pages 49--67.
  Birkh\"{a}user Boston, Boston, MA, 2009.

\bibitem[Bes88]{bestvina}
Mladen Bestvina.
\newblock Degenerations of the hyperbolic space.
\newblock {\em Duke Math. J.}, 56(1):143--161, 1988.

\bibitem[BH99]{bridson_haefliger}
Martin~R. Bridson and Andr\'{e} Haefliger.
\newblock {\em Metric spaces of non-positive curvature}, volume 319 of {\em
  Grundlehren der mathematischen Wissenschaften [Fundamental Principles of
  Mathematical Sciences]}.
\newblock Springer-Verlag, Berlin, 1999.

\bibitem[BJ97]{bishop_jones}
Christopher~J. Bishop and Peter~W. Jones.
\newblock Hausdorff dimension and {K}leinian groups.
\newblock {\em Acta Math.}, 179(1):1--39, 1997.

\bibitem[BJ17]{boucksom_jonsson_degeneration}
S\'{e}bastien Boucksom and Mattias Jonsson.
\newblock Tropical and non-{A}rchimedean limits of degenerating families of
  volume forms.
\newblock {\em J. \'{E}c. polytech. Math.}, 4:87--139, 2017.

\bibitem[BR10]{baker_rumely}
Matthew Baker and Robert Rumely.
\newblock {\em Potential theory and dynamics on the {B}erkovich projective
  line}, volume 159 of {\em Mathematical Surveys and Monographs}.
\newblock American Mathematical Society, Providence, RI, 2010.

\bibitem[Chu68]{chuckrow}
Vicki Chuckrow.
\newblock On {S}chottky groups with applications to kleinian groups.
\newblock {\em Ann. of Math. (2)}, 88:47--61, 1968.

\bibitem[CM87]{culler_morgan}
Marc Culler and John~W. Morgan.
\newblock Group actions on {${\bf R}$}-trees.
\newblock {\em Proc. London Math. Soc. (3)}, 55(3):571--604, 1987.

\bibitem[DF19]{dujardin_favre_degen}
Romain Dujardin and Charles Favre.
\newblock Degenerations of {${\rm SL}(2,\Bbb C)$} representations and
  {L}yapunov exponents.
\newblock {\em Ann. H. Lebesgue}, 2:515--565, 2019.

\bibitem[DR23]{dang_riviere_poincare}
Nguyen~Viet Dang and Gabriel Rivière.
\newblock Poincar{\'e} series and linking of {L}egendrian knots.
\newblock \texttt{https://arxiv.org/abs/2005.13235}, 2023.

\bibitem[DSU17]{das_simmons_urbanski}
Tushar Das, David Simmons, and Mariusz Urba\'{n}ski.
\newblock {\em Geometry and dynamics in {G}romov hyperbolic metric spaces},
  volume 218 of {\em Mathematical Surveys and Monographs}.
\newblock American Mathematical Society, Providence, RI, 2017.
\newblock With an emphasis on non-proper settings.

\bibitem[Duc14]{ducros}
Antoine Ducros.
\newblock {\em La structure des courbes analytiques}.
\newblock \texttt{https://webusers.imj-prg.fr/~antoine.ducros/trirss.pdf}.
  Preprint, 2014.

\bibitem[Fal14]{falconer}
Kenneth Falconer.
\newblock {\em Fractal geometry}.
\newblock John Wiley \& Sons, Ltd., Chichester, third edition, 2014.
\newblock Mathematical foundations and applications.

\bibitem[Fav20]{favre_degen}
Charles Favre.
\newblock Degeneration of endomorphisms of the complex projective space in the
  hybrid space.
\newblock {\em J. Inst. Math. Jussieu}, 19(4):1141--1183, 2020.

\bibitem[Fri86]{fried_analytic_torsion}
David Fried.
\newblock Analytic torsion and closed geodesics on hyperbolic manifolds.
\newblock {\em Invent. Math.}, 84(3):523--540, 1986.

\bibitem[Ger80]{gerritzen_80}
Lothar Gerritzen.
\newblock On the space of {S}chottky-{M}umford curves.
\newblock In {\em Seminar on {N}umber {T}heory, 1979--1980 ({F}rench)}, pages
  Exp. No. 22, 5. Univ. Bordeaux I, Talence, 1980.

\bibitem[Ger81]{gerritzen_81}
Lothar Gerritzen.
\newblock Zur analytischen {B}eschreibung des {R}aumes der
  {S}chottky-{M}umford-{K}urven.
\newblock {\em Math. Ann.}, 255(2):259--271, 1981.

\bibitem[GLZ04]{guillope_lin_zworski}
Laurent Guillop\'{e}, Kevin~K. Lin, and Maciej Zworski.
\newblock The {S}elberg zeta function for convex co-compact {S}chottky groups.
\newblock {\em Comm. Math. Phys.}, 245(1):149--176, 2004.

\bibitem[GN07]{grigorchuk_nekrashevych}
Rostislav Grigorchuk and Volodymyr Nekrashevych.
\newblock Self-similar groups, operator algebras and {S}chur complement.
\newblock {\em J. Mod. Dyn.}, 1(3):323--370, 2007.

\bibitem[GvdP80]{gerritzen_van_der_put}
Lothar Gerritzen and Marius van~der Put.
\newblock {\em Schottky groups and {M}umford curves}, volume 817 of {\em
  Lecture Notes in Mathematics}.
\newblock Springer, Berlin, 1980.

\bibitem[Hej75]{hejhal_schottky_teichmuller}
Dennis~A. Hejhal.
\newblock On {S}chottky and {T}eichm\"{u}ller spaces.
\newblock {\em Advances in Math.}, 15:133--156, 1975.

\bibitem[HH97]{hersonsky_hubbard}
Sa'ar Hersonsky and John Hubbard.
\newblock Groups of automorphisms of trees and their limit sets.
\newblock {\em Ergodic Theory Dynam. Systems}, 17(4):869--884, 1997.

\bibitem[Hub94]{huber}
R.~Huber.
\newblock A generalization of formal schemes and rigid analytic varieties.
\newblock {\em Math. Z.}, 217(4):513--551, 1994.

\bibitem[Jos06]{jost_compact}
J\"{u}rgen Jost.
\newblock {\em Compact {R}iemann surfaces}.
\newblock Universitext. Springer-Verlag, Berlin, third edition, 2006.
\newblock An introduction to contemporary mathematics.

\bibitem[Kap09]{kapovich}
Michael Kapovich.
\newblock {\em Hyperbolic manifolds and discrete groups}.
\newblock Modern Birkh\"{a}user Classics. Birkh\"{a}user Boston, Ltd., Boston,
  MA, 2009.
\newblock Reprint of the 2001 edition.

\bibitem[Kat95]{kato}
Tosio Kato.
\newblock {\em Perturbation theory for linear operators}.
\newblock Classics in Mathematics. Springer-Verlag, Berlin, 1995.
\newblock Reprint of the 1980 edition.

\bibitem[KH95]{katok}
Anatole Katok and Boris Hasselblatt.
\newblock {\em Introduction to the modern theory of dynamical systems},
  volume~54 of {\em Encyclopedia of Mathematics and its Applications}.
\newblock Cambridge University Press, Cambridge, 1995.
\newblock With a supplementary chapter by Katok and Leonardo Mendoza.

\bibitem[Kir13]{kirillov}
Alexandre~A. Kirillov.
\newblock {\em A tale of two fractals}.
\newblock Springer, New York, 2013.

\bibitem[Kis23]{kisil}
Vladimir~V. Kisil.
\newblock Cycles cross ratio: an invitation.
\newblock {\em Elem. Math.}, 78(2):49--61, 2023.

\bibitem[Kiw15]{kiwi}
Jan Kiwi.
\newblock Rescaling limits of complex rational maps.
\newblock {\em Duke Math. J.}, 164(7):1437--1470, 2015.

\bibitem[KL18]{kapovich_leeb}
Michael Kapovich and Bernhard Leeb.
\newblock Discrete isometry groups of symmetric spaces.
\newblock In {\em Handbook of group actions. {V}ol. {IV}}, volume~41 of {\em
  Adv. Lect. Math. (ALM)}, pages 191--290. Int. Press, Somerville, MA, 2018.

\bibitem[LP24]{lemanissier_poineau}
Thibaud Lemanissier and Jérôme Poineau.
\newblock Espaces de berkovich globaux : cat\'egorie, topologie, cohomologie,
  2024.

\bibitem[Luo22]{luo_trees}
Yusheng Luo.
\newblock Trees, length spectra for rational maps via barycentric extensions,
  and {B}erkovich spaces.
\newblock {\em Duke Math. J.}, 171(14):2943--3001, 2022.

\bibitem[Lyu13]{lyubich_book}
Mikhail Lyubich.
\newblock Conformal geometry and dynamics of quadratic polynomials.
\newblock {\em to appear}, 2013.

\bibitem[Mar74]{marden}
Albert Marden.
\newblock Schottky groups and circles.
\newblock In {\em Contributions to analysis (a collection of papers dedicated
  to {L}ipman {B}ers)}, pages 273--278. Academic Press, New York-London, 1974.

\bibitem[Mas67]{maskit_characterization_schottky}
Bernard Maskit.
\newblock A characterization of {S}chottky groups.
\newblock {\em J. Analyse Math.}, 19:227--230, 1967.

\bibitem[McM98]{mcmullen_hausdorff_3}
Curtis~T. McMullen.
\newblock Hausdorff dimension and conformal dynamics. {III}. {C}omputation of
  dimension.
\newblock {\em Amer. J. Math.}, 120(4):691--721, 1998.

\bibitem[Mil99]{milnor_complex}
John Milnor.
\newblock {\em Dynamics in one complex variable}.
\newblock Friedr. Vieweg \& Sohn, Braunschweig, 1999.
\newblock Introductory lectures.

\bibitem[MS84]{morgan_shalen1}
John~W. Morgan and Peter~B. Shalen.
\newblock Valuations, trees, and degenerations of hyperbolic structures. {I}.
\newblock {\em Ann. of Math. (2)}, 120(3):401--476, 1984.

\bibitem[MS85]{morgan-shalen2}
John~W. Morgan and Peter~B. Shalen.
\newblock An introduction to compactifying spaces of hyperbolic structures by
  actions on trees.
\newblock In {\em Geometry and topology ({C}ollege {P}ark, {M}d., 1983/84)},
  volume 1167 of {\em Lecture Notes in Math.}, pages 228--240. Springer,
  Berlin, 1985.

\bibitem[MSW02]{mumford_series_wright}
David Mumford, Caroline Series, and David Wright.
\newblock {\em Indra's pearls}.
\newblock Cambridge University Press, New York, 2002.
\newblock The vision of Felix Klein.

\bibitem[Mum72]{mumford}
David Mumford.
\newblock An analytic construction of degenerating curves over complete local
  rings.
\newblock {\em Compositio Math.}, 24:129--174, 1972.

\bibitem[Nie18]{nie_rescaling}
Hongming Nie.
\newblock Rescaling limits in non-{A}rchimedean dynamics.
\newblock {\em Acta Arith.}, 185(4):315--331, 2018.

\bibitem[Pat76]{patterson}
Samuel~J. Patterson.
\newblock The limit set of a {F}uchsian group.
\newblock {\em Acta Math.}, 136(3-4):241--273, 1976.

\bibitem[Pau88]{paulin_topo_gromov}
Fr\'{e}d\'{e}ric Paulin.
\newblock Topologie de {G}romov \'{e}quivariante, structures hyperboliques et
  arbres r\'{e}els.
\newblock {\em Invent. Math.}, 94(1):53--80, 1988.

\bibitem[Pau97a]{paulin_bourbaki}
Fr\'{e}d\'{e}ric Paulin.
\newblock Actions de groupes sur les arbres.
\newblock Number 241, pages Exp. No. 808, 3, 97--137. 1997.
\newblock S\'{e}minaire Bourbaki, Vol. 1995/96.

\bibitem[Pau97b]{paulin_degen}
Fr\'{e}d\'{e}ric Paulin.
\newblock D\'{e}g\'{e}n\'{e}rescence de sous-groupes discrets des groupes de
  {L}ie semi-simples.
\newblock {\em C. R. Acad. Sci. Paris S\'{e}r. I Math.}, 324(11):1217--1220,
  1997.

\bibitem[Poi10]{poineau_berkovich_Z}
J\'{e}r\^{o}me Poineau.
\newblock La droite de {B}erkovich sur {$\bold Z$}.
\newblock {\em Ast\'{e}risque}, (334):viii+xii+284, 2010.

\bibitem[PT21a]{poineau_turcheti_1}
J\'{e}r\^{o}me Poineau and Daniele Turchetti.
\newblock Berkovich curves and {S}chottky uniformization {I}: the {B}erkovich
  affine line.
\newblock In {\em Arithmetic and geometry over local fields---{VIASM} 2018},
  volume 2275 of {\em Lecture Notes in Math.}, pages 179--223. Springer, Cham,
  [2021].

\bibitem[PT21b]{poineau_turcheti_2}
J\'{e}r\^{o}me Poineau and Daniele Turchetti.
\newblock Berkovich curves and {S}chottky uniformization {II}: analytic
  uniformization of {M}umford curves.
\newblock In {\em Arithmetic and geometry over local fields---{VIASM} 2018},
  volume 2275 of {\em Lecture Notes in Math.}, pages 225--279. Springer, Cham,
  [2021].

\bibitem[PT22]{poineau_turcheti_universal}
J\'{e}r\^{o}me Poineau and Daniele Turchetti.
\newblock Schottky spaces and universal {M}umford curves over {$\Bbb {Z}$}.
\newblock {\em Selecta Math. (N.S.)}, 28(4):Paper No. 79, 53, 2022.

\bibitem[Qui06]{quint_overview}
Jean-Fran{\c{c}}ois Quint.
\newblock An overview of {P}atterson-{S}ullivan theory.
\newblock In {\em Workshop The barycenter method, FIM, Zurich}, 2006.

\bibitem[Ray74]{raynaud}
Michel Raynaud.
\newblock G\'{e}om\'{e}trie analytique rigide d'apr\`es {T}ate,
  {K}iehl,{$\cdots $}.
\newblock In {\em Table {R}onde d'{A}nalyse {N}on {A}rchim\'{e}dienne ({P}aris,
  1972)}, volume Tome 102 of {\em Suppl\'{e}ment au Bull. Soc. Math. France},
  pages 319--327. Soc. Math. France, Paris, 1974.

\bibitem[RL03]{letelier_hyperbolic}
Juan Rivera-Letelier.
\newblock Espace hyperbolique {$p$}-adique et dynamique des fonctions
  rationnelles.
\newblock {\em Compositio Math.}, 138(2):199--231, 2003.

\bibitem[Rue82]{ruelle}
David Ruelle.
\newblock Repellers for real analytic maps.
\newblock {\em Ergodic Theory Dynam. Systems}, 2(1):99--107, 1982.

\bibitem[Sul79]{sullivan_density}
Dennis Sullivan.
\newblock The density at infinity of a discrete group of hyperbolic motions.
\newblock {\em Inst. Hautes \'{E}tudes Sci. Publ. Math.}, (50):171--202, 1979.

\bibitem[Sul82]{sullivan_discrete_conformal}
Dennis Sullivan.
\newblock Discrete conformal groups and measurable dynamics.
\newblock {\em Bull. Amer. Math. Soc. (N.S.)}, 6(1):57--73, 1982.

\bibitem[Tat71]{tate}
John Tate.
\newblock Rigid analytic spaces.
\newblock {\em Invent. Math.}, 12:257--289, 1971.

\bibitem[Wei67]{weil}
Andr\'{e} Weil.
\newblock {\em Basic number theory}, volume Band 144 of {\em Die Grundlehren
  der mathematischen Wissenschaften}.
\newblock Springer-Verlag New York, Inc., New York, 1967.

\end{thebibliography}
\end{small}

\bigskip

\bigskip

{\footnotesize%
 \textsc{Nguyen-Bac Dang}, Université Paris-Saclay, CNRS, LMO, 91405, Orsay, France \par
  \textit{E-mail address}: \texttt{nguyen-bac.dang@universite-paris-saclay.fr} 
}

\

{\footnotesize%
 \textsc{Vlerë Mehmeti}, Sorbonne Université and Université Paris Cité, CNRS, IMJ-PRG, F-75005 Paris, France 
 \par
  \textit{E-mail address}: \texttt{vlere.mehmeti@imj-prg.fr} 
}

\end{document}